\pdfoutput=1 

\documentclass[12pt]{extarticle}
\usepackage[utf8]{inputenc}
\usepackage[T1]{fontenc}

\usepackage{caption}
\oddsidemargin \evensidemargin
\let\OLDthebibliography\thebibliography
\renewcommand\thebibliography[1]{
  \OLDthebibliography{#1}
  \setlength{\parskip}{0.4pt}
  \setlength{\itemsep}{4.0pt plus 0.3ex}
}

\usepackage{microtype} 
\usepackage{lmodern}   

\usepackage[english]{babel} 

\usepackage{booktabs} 
\usepackage{multirow}

\usepackage[shortlabels]{enumitem} 
\setlist[itemize]{noitemsep} 

\usepackage{hyperref} 

\usepackage{import}
\usepackage{xifthen}
\usepackage{pdfpages}
\usepackage{transparent}
\usepackage{wrapfig}
\usepackage{hyperref}

\usepackage{subcaption}

\usepackage[mode=build]{standalone}
\usepackage{tikz}
\usetikzlibrary{cd}
\usetikzlibrary{calc}
\usetikzlibrary{decorations.markings}
\tikzstyle{state}=[
rectangle,
minimum size =1.25cm,
thick,
align=center
]
\usepackage{scalefnt}
\usepackage{kbordermatrix}
\usepackage{amsfonts}
\usepackage{comment}
\DeclareMathSymbol{\shortminus}{\mathbin}{AMSa}{"39}

\usepackage{lineno}

\usepackage{xcolor, soulutf8}
\definecolor{bubbles}{rgb}{0.91, 1.0, 1.0}
\definecolor{aquamarine}{rgb}{0.5, 1.0, 0.83}
\definecolor{bubblegum}{rgb}{0.99, 0.76, 0.8}
\definecolor{bluebell}{rgb}{0.69, 0.69, 0.92}
\definecolor{dollarbill}{rgb}{0.72, 0.93, 0.6}

\hypersetup{
    colorlinks=true,
    linkcolor=blue,
    filecolor=magenta,      
    urlcolor=blue,
    citecolor=blue,
}

\usepackage{amssymb}
\usepackage{amsmath}
\usepackage{makecell}

\usepackage{amsthm}
\newtheorem{theorem}{Theorem}
\newtheorem{remark}[theorem]{Remark}
\newtheorem{lemma}[theorem]{Lemma}
\newtheorem{proposition}[theorem]{Proposition}

\theoremstyle{definition}
\newtheorem{examplex}[theorem]{Example}
\newenvironment{example}
{\pushQED{\qed}\examplex}
{\popQED\endexamplex}

\usepackage{algorithm}
\usepackage{algpseudocode}

\makeatletter
\let\c@algorithm\c@theorem

\makeatother

\newcommand{\figref}[1]{Fig.~\ref{#1}}
\newcommand{\exref}[1]{Example~\ref{#1}}
\newcommand{\tabref}[1]{Table~\ref{#1}}
\newcommand{\thmref}[1]{Theorem~\ref{#1}}

\newcommand{\propref}[1]{Proposition~\ref{#1}}
\newcommand{\lemref}[1]{Lemma~\ref{#1}}

\newcommand{\secref}[1]{Section~\ref{#1}}
\newcommand{\algoref}[1]{Algorithm~\ref{#1}}

\newcommand{\pr}{\operatorname{pr}}
\newcommand{\Gal}{\operatorname{Gal}}
\newcommand{\DT}{\rm d\mathbb Z_2}
\newcommand{\tS}{{\rm t}S_3}
\newcommand{\nK}{{\rm n}K_4}
\newcommand{\Z}{\mathbb Z}
\newcommand{\Q}{\mathbb Q}

\newcommand{\C}{\mathbb C}
\newcommand{\bbP}{\mathbb P}

\newcommand{\res}{\operatorname{Res}}

\newcommand{\PH}{\mathit{PH}}
\newcommand{\calL}{\mathcal{L}}
\newcommand{\calX}{\mathcal{X}}
\newcommand{\Pic}{\operatorname{Pic}}
\newcommand{\Tr}{\operatorname{Tr}}

\newcommand{\Hdr}{H_\mathrm{dR}}
\newcommand{\PHdr}{\PH_\mathrm{dR}}
\newcommand{\ud}{\mathrm{d}}
\def\st{\ \middle|\ }

\title{\bf Certified Monodromy Groups for\\ Families of Hypersurfaces} 
\date{}

\author{Eric Pichon-Pharabod and Simon Telen}

\begin{document}

\maketitle

\begin{abstract}
\noindent The monodromy group of a family of algebraic varieties is the group of isometries of the homology of a generic fibre, induced by parallel transport along a loop in the base space.
We present a method to compute the monodromy group of families of generically smooth projective hypersurfaces using certified numerical algebraic geometry. This leads us to develop efficient algorithms for computing the discriminants of these~families, which is of independent interest.
Our flagship example is the case of families of symmetric cubic surfaces. The monodromy group is the Galois group of the 27 lines. We find that, when imposing symmetry, this Galois group is often a small subgroup of the Weyl group $W(E_6)$. 
We also include computational results for symmetric quartic surfaces.
\end{abstract}

\small \noindent
MSC classes: 14D05, 65H14 (Primary) 14Q10, 32S40, 32G20, 55R80 (Secondary)

\normalsize 

\section{Introduction} \label{sec:introduction}
This paper presents a computational study of the monodromy group of a family of projective hypersurfaces via certified numerical algebraic geometry. 
The precise setting is as follows. 
Consider a polynomial $f\in \C[x_0, \dots, x_{m+1}, z_0, \dots, z_{n}]$ such that $f$ is homogeneous of degree $d$ in the variables $x_i$, and homogeneous of degree $e$ in the $z_j$. We define 
\begin{equation} \label{eq:calX} {\cal X} \, = \, \{ (x,z) \in \mathbb{P}^{m+1} \times \mathbb{P}^n \, : \, f(x,z) \, = \, 0 \} .\end{equation}
The fibre over $z^* \in \mathbb{P}^n$ of the coordinate projection $\calX \rightarrow \mathbb{P}^n$ is $X_{z^*} =V(f(x,z^*))\subset \bbP^{m+1}$. The hypersurfaces $X_z$ for $z \in \mathbb{P}^n$ constitute our family.
We make the following assumption:
\begin{equation} \label{eq:assume_generic_smoothness}
\text{There exists $z \in \mathbb{P}^n$ for which $X_z$ is smooth.}
\end{equation}
This is equivalent to the assumption that the discriminant $\nabla = \{z \in \mathbb{P}^n \, : \, X_z \text{ is singular} \}$ is a hypersurface.
The monodromy group of our family $\calX \to \bbP^n$ is the image of the map
 \begin{equation} \label{eq:monodromy_rep_intro} \pi_1(\mathbb{P}^n \setminus \nabla, b) \, \longrightarrow \, {\rm Aut}(H_m(X_b)). \end{equation}
Here $b \in \mathbb{P}^n \setminus \nabla$ corresponds to an arbitrary nonsingular member of our family, $\pi_1(\mathbb{P}^n \setminus \nabla, b)$ is the fundamental group, and the image of a loop $\ell \in \pi_1(\mathbb{P}^n \setminus \nabla, b)$ is the automorphism $\ell_*: H_m(X_b) \rightarrow H_m(X_b)$ of the homology lattice $H_m(X_b)$ induced by parallel transport of cycles along $\ell$. 
Concretely, one can think of $\ell_*$ as follows. 
As the parameters $z$ move along $\ell$, a closed chain $\sigma \subset X_b$ deforms continuously. 
The image of $[\sigma] \in H_m(X_b)$ under $\ell_*$ is the class of the (unique up to homology) deformed chain after completing the loop. 
Once a basis of $H_m(X_b)$ is fixed, the monodromy group is a subgroup of ${\rm GL}(r,\mathbb{Z})$, where $r =  {\rm rank} \, H_m(X_b)$. 
By ``computing'' the monodromy group, we mean computing generators of this matrix group.

The action of monodromy on $H_m(X_b)$ is governed by a differential system, the Gauss--Manin connection. 
One of our main contributions lies in the effective numerical integration of the solutions to this system with certified upper bounds on the accumulated error. 
When the error bounds are small enough to isolate lattice points, this certifies the integration steps and allows us to certify the action of monodromy. 
This method is summarised in Algorithm \ref{alg:main}. 

\begin{algorithm}[tbh]
\caption{Computation of the Monodromy Group of $\calX\to \bbP^n$}
\label{alg:main}
\begin{algorithmic}[1]

\State Compute the discriminant $\nabla$ and choose a generic line $L \simeq \mathbb{P}^1 \subset \mathbb{P}^n$.

\State Choose $b \in L \setminus \nabla$ and compute a set of generators $\ell_1, \ldots, \ell_r$ of
$\pi_1(L \setminus \nabla, b)$.

\State Compute a period matrix
$\Pi$ for $X_b$, with certified precision.

\State Compute the analytic continuation $\tilde{\Pi}_i$ of $\Pi$
along $\ell_i$ for $i = 1, \ldots, r$, with certified~precision.

\State Recover integer matrices $M_i$ from $\tilde{M}_i = \Pi^{-1}\tilde{\Pi}_i$ by rounding.

\State Return the matrix group generated by
$M_1, \ldots, M_r$.

\end{algorithmic}
\end{algorithm}

The first step is to compute the discriminant $\nabla$. 
Computing discriminants is a notoriously hard problem in computer algebra. 
For this step, and in most of our examples, we will focus on the case where $f$ is linear in $z$ ($e = 1$).
The defining equation $f$ can be characterised by fixing $n+1$ linearly independent homogeneous forms $h = (h_0, \ldots, h_n) \in \mathbb{C}[x_0,\ldots, x_{m+1}]_d^{n+1}$ of degree $d$. That is, $f = \sum_{i = 0}^n z_i \, h_i$.  
The $h_i$ span the linear system ${\cal L}_h \simeq \mathbb{P}^n$ of hypersurfaces 
\[ X_z \, = \, \{x \in \mathbb{P}^{m+1} \, : \, z_0 \, h_0(x) + z_1 \, h_1(x) + \cdots + z_n \, h_n(x) \, = \, 0 \}, \quad z \, = \, (z_0: \cdots: z_n) \in \mathbb{P}^n. \] 
This is a linear subsystem of $\mathbb{P} H^0(\mathbb{P}^{m+1}, {\cal O}(d))$, and we denote its discriminant by $\nabla_h$.
The methods described in Section \ref{sec:discriminants} make use of the linearity in $z_0, \ldots, z_n$ to break down the computation of $\nabla_h$ into two steps, leading to a significant gain of efficiency compared to the naive, general approach via variable elimination. 
This allows us, in particular, to compute the discriminant for any family of symmetric cubic surfaces (see below). 
The line $L$ mentioned in step 1 of Algorithm \ref{alg:main} is ``generic'' when it intersects the discriminant $\nabla$ transversally. 
To check this, it suffices to compare the degree of $\nabla$ with the number of points in $L \cap \nabla$. 
Computing the full discriminant gives more information, but this is independently interesting.

Step 2 is to compute a basis for the fundamental group of the punctured line $\pi_1(L \setminus \nabla, b)$. 
Here we are implicitly using Lefschetz's Theorem \ref{thm:lefschetz_zariski}, which says that the images of $\ell_1, \ldots, \ell_r$ in $\mathbb{P}^n \setminus \nabla$ generate the fundamental group $\pi_1(\mathbb{P}^n \setminus \nabla, b)$. 
This means that $\ell_1, \ldots, \ell_r$ can be used to find the representation \eqref{eq:monodromy_rep_intro}. This is explained in more detail in Section \ref{sec:lefschetz}.

Step 3 requires the computation of a period matrix $\Pi$ for the smooth member $X_b$ of our family. 
Here we use the method from \cite{LairezEtAl_2024} and the corresponding implementation. 
Computing $\Pi$ involves the computation bases for $H_m(X_b)$ and for the de Rham cohomology $H^m_{\rm dR}(X_b)$.

Next, in step 4, we determine the action of $\pi_1( \mathbb{P}^n \setminus \nabla, b)$ on the period matrix $\Pi$ by analytically continuing $\Pi(t)$ along the loop $\ell_i$. 
Our loop is a continuous map $\ell_i: [0,1] \rightarrow L \setminus \nabla$ with $\ell_i(0) = \ell_i(1) = b$, and $\Pi(t)$ is a period matrix for $X_{\ell_i(t)}$. 
The continuation is computed by numerically solving the Picard--Fuchs differential equation, see Section \ref{sec:numerical_integration}.
For this we use the methods of \cite{ChudnovskyChudnovsky_1990, Mezzarobba_2016,VanDerHoeven_1999}, and in particular the implementation of \cite{Mezzarobba_2016}.

It is crucial that steps 2 and 4 are performed with certified bounds of numerical precision. 
This ensures that, in step 5, the matrices $\tilde{M}_i$ are approximately integer, in the sense that rounding its entries certifiably gives the correct matrix $M_i$ representing $(\ell_i)_* \in {\rm Aut}(H_m(X_b))$. \\

A main computational result of this paper is the successful application of \algoref{alg:main} for linear systems of symmetric cubic surfaces. 
The polynomials in $h$ defining the linear system are quaternary cubics: $h = (h_0, h_1, \ldots, h_n), \, h_i \in \mathbb{C}[x_0,x_1,x_2,x_3]_3$. 
For instance, if $h = (x_0^3, x_0^2x_1, \ldots, x_3^3)$ consists of all 20 cubic monomials, then ${\cal L}_h \simeq \mathbb{P}^{19}$ is the complete linear system of cubic surfaces. 
In this context, the monodromy group of such a family turns out to be the Galois group of the enumerative problem associated with the 27 lines on $X_b$ \cite{harris1979galois}, that is, the group generated by all permutations of the 27 lines induced by closed loops $\ell \in \pi_1({\cal L}_h \setminus \nabla_h, b)$. 
Details are recalled in \secref{sec:cubicsurfaces}. We denote this group by ${\rm Gal}({\cal L}_h)$. For generic linear subsystems ${\cal L}_h \subseteq \mathbb{P}^{19}$ of dimension $n > 0$, our Galois group equals the Weyl group of $E_6$, i.e., ${\rm Gal}({\cal L}_h) = W(E_6)$. 
In that sense, cases with ${\rm Gal}({\cal L}_h) \subsetneq W(E_6)$ are particularly interesting. 
We identify all such ${\cal L}_h$ arising as invariant subspaces from a class of finite group actions. 

Let $G$ be a group acting linearly on $\mathbb{C}[x_0,x_1,x_2,x_3]_3$ and let $h_G = (h_0, \ldots, h_n)$ be a basis for the invariant subspace $\mathbb{C}[x_0,x_1,x_2,x_3]_3^G = \{f \in \mathbb{C}[x_0,x_1,x_2,x_3]_3 \, : \, g \cdot f = f \text{ for all } g \in G \}$. 
We shall write ${\cal L}_G = {\cal L}_{h_G}$ and $\nabla_G = \nabla_{h_G}$ to simplify notation. Notice that ${\cal L}_G$ is a linear subspace of the quaternary cubics. In general, it is different from the the space of $G$-stable cubic surfaces, which may be reducible. 
The symmetric group $S_5$ acts naturally on Sylvester pentahedral normal forms for general cubic surfaces \cite[\S84]{Segre_1942}
\begin{equation}\label{eq:pentahedral}
 a_0 \, y_0^3 +   a_1 \, y_1^3+ a_2 \, y_2^3+ a_3 \, y_3^3+ a_4 \, y_4^3 \, = \, y_0 + y_1 + y_2 + y_3 +y_4 \, = \, 0
\end{equation}
by permuting coordinates.
To descend this to a linear action on $\C[x_0, x_1, x_2, x_3]_3$, we choose coordinates on the hyperplane $y_0 + y_1 + y_2 + y_3 +y_4=0$, namely $[y_0:y_1:y_2:y_3]$. We call this action the Sylvester action of $S_5$ on $\C[x_0, x_1, x_2, x_3]_3$.

\begin{example}
The double transposition $\tau=(15)(23)\in S_5$ acts on $\C[x_0, x_1,x_2,x_3]$ by
\[ x_0 \mapsto  -x_0-x_1-x_2-x_3, \quad  x_1 \mapsto x_2, \quad x_2 \mapsto x_1, \quad  x_3 \mapsto x_3 . \]
The space of cubic polynomials $\C[x_0,x_1,x_2,x_3]_3^{\langle\tau\rangle}$ invariant under the action of $\tau$ has projective dimension $\dim {\cal L}_{\langle \tau \rangle} = 9$. Writing $x_+ = x_0+ x_1 + x_2 + x_3$, a basis $h_{\langle \tau\rangle}$ is given by
\begin{equation*}
\begin{gathered}
x_1^3 + x_2^3\,,\quad
-x_+^3 + x_0^3\,,\quad
x_0x_1^2 -x_+x_2^2\,,\quad
x_0x_2^2 -x_+x_1^2\,,\quad
x_1x_0^2 + x_2x_+^2\,,\\
x_2x_0^2 + x_1x_+^2\,,\quad
x_1x_2(x_0-x_+)\,,\quad
x_+x_0(x_1+x_2)\,,\quad
x_1x_2(x_1+x_2)\,,\quad
x_+x_0(x_0-x_+)\,.
\end{gathered} \qedhere
\end{equation*}
\end{example}

A leading objective of the present text is the computation of the Galois groups of lines lying on surfaces in ${\cal L}_G$ for $G$ a subgroup of $S_5$, which we summarise in the following theorem.
\begin{theorem} \label{thm:main_new}
The Galois group ${\rm Gal}({\cal L}_G)$ of the $G$-invariant quaternary cubics, for $G$ any subgroup of the Sylvester action of $S_5$, is as listed in \tabref{tab:main}. The table shows the dimension $n$ of ${\cal L}_G \simeq \mathbb{P}^n$, the name and order of ${\rm Gal}({\cal L}_G)$, and the degrees of the components of $\nabla_G$. 
\end{theorem}

\begin{table}[ht]
\setlength{\tabcolsep}{3pt}
\footnotesize
\centering
\begin{tabular}{cccccc}
\toprule
$G$ & $\dim {\cal L}_G$ & ${\rm Gal}({\cal L}_G)$ & $|{\rm Gal}({\cal L}_G)|$ & $\deg \nabla_G$ &  cit.\\ \midrule
  $S_4$  &  2                &    $\Z_2^2$      &          $4$           &   $1^3,3^1$   & \cite{BrazeltonRaman_2024}\\
  $A_4$  &  2                &    $\Z_2^2$        &           $4$          &   $1^3,3^1$         \\
  $D_6$  & 3                  &            $D_6$       &            $12$            &    $1^3,3^1$     & \cite{Landi_2025}   \\
  $D_5$  & 1                 &               $\Z_1$    &           $1$            &   $2^1$         \\
  $D_4$  &  3                &    $\Z_2^3$        &            $8$         &   $1^4,3^1$         \\
  $S_3$  &  6                &        $S_3^2$   &               $36$         &   $1^1,4^1,8^1$  & \cite{Landi_2025}      \\
  $\tS$  &  3       &            $D_6$      &            $12$           &  $1^3, 3^1$          \\
  $\Z_6$  & 4                  &        $\Z_6\times S_3$        &            $36$          &    $1^2,2^1,3^1$        \\ 
  $\Z_5$  & 3                 &               $\Z_5$  &          $5$      &   $4^1$         \\ 
  $K_4$  &  4                &    $\Z_2^4$      &              $16$       &   $1^5, 3^1$         \\   \bottomrule
\end{tabular}
\quad 
\begin{tabular}{cccccc}
\toprule
$G$ & $\dim {\cal L}_G$ & ${\rm Gal}({\cal L}_G)$ & $|{\rm Gal}({\cal L}_G)|$ & $\deg \nabla_G$ &  cit.\\ \midrule
  $\nK$  & 7     &      $\Z_2^2\times S_4$     &         $96$           &     $2^1, 4^3$       \\
  $\Z_4$  &   4              &       $\Z_2^2\times \Z_4$      &      $16$          &    $1^3, 2^1, 3^1$        \\
  $\Z_3$  &   7              &      $\Z_3\times S_3^2$    &         $108$           &    $2^1,4^1,8^1$        \\
  $\Z_2$  &   12              &         $W(F_4)$       &      $1152$       &   $10^1, 12^1$    & \cite{Landi_2025}     \\
  ${\DT}$  &  9          &       $D_4\times S_4$       &      $192$        &  $4^2,8^1$          \\
  $\Z_1$  & 19                &      $W(E_6)$      &       $51840$      &   $32^1$    & \cite{Jordan_1870}     \\
  $S_5$  & 0                &      --      &       --      &   --         \\
  $A_5$  & 0                &      --      &       --      &   --         \\
  $F_5$  & 0                &      --      &       --      &   --         \\
  \\
  \bottomrule
\end{tabular}
\caption{Galois groups and discriminant degrees of $G$-invariant cubic surfaces.}
\label{tab:main}
\end{table}

Details about all considered groups are given in \secref{sec:cubicsurfaces}. The notation in the table is as follows.
In the column labeled ${\rm Gal}({\cal L}_G)$, the isomorphism classes of the Galois groups are given in terms of symmetric groups on $d$ elements ($S_d$), cyclic groups of order $n$ ($\Z_n$), dihedral groups of the $n$-gon ($D_n$) and Weyl groups of Lie groups $L$ ($W(L)$).
In the column labeled $\deg \nabla_G$, the string $1^2,2^1,3^1$ means that the discriminant $\nabla_G$, viewed as a reduced subscheme of ${\cal L}_G \simeq \mathbb{P}^n$ over $\Q$, has two components of degree one, one component of degree two and one component of degree three. 
The column labeled cit.~contains references to the first appearances of the corresponding result in the literature, when it exists.

The study of smooth cubic surfaces and their 27 lines is classical. In particular, the entry $G = \Z_1 = \{ 1\}$ of \tabref{tab:main} is well known. At the same time, cubic surfaces remain an active area of research, see for instance \cite{ranestad2020twentyseven}. In fact, this project was motivated by the recent work \cite{BrazeltonRaman_2024}, in which the authors show that the Galois group ${\rm Gal}({\cal L}_{S_4})$ is the Klein group $K_4 \subsetneq W(E_6)$ \cite[Theorem 1.2]{BrazeltonRaman_2024}. 
Furthermore, also motivated by \cite{BrazeltonRaman_2024}, the recent work \cite{Landi_2025} extended this computation to three other symmetric families while we were writing this paper.
Hence, also the entries $G = S_4, S_3, D_6$ and $\Z_2$ in \tabref{tab:main} are not new.
An earlier computation of the monodromy groups of conics lying on a quartic surface with a certain symmetry was carried out by Bouyer in \cite{Bouyer_2020} --- a result that we also recover in \secref{sec:furtherexamples}.
We also reproduce results of \cite{ MedranoMartinDelCampo_2022a,MedranoMartinDelCampo_2022} for monodromy groups of certain families of cubic and quartic surfaces.

A punchline of this paper is that Algorithm \ref{alg:main} proves such results via certified numerical computation. This helps to identify interesting cases with extra symmetry. 
One can then try to find theoretical arguments leading to an alternative proof or more geometric insight. 

Computing Galois groups using numerical homotopy continuation is the topic of \cite{hauenstein2018numerical,leykin2009galois}. 
In particular, \cite[Section 3.1.4]{hauenstein2018numerical} recovers the Galois group $W(E_6)$ of the 27 lines. 
The authors propose to compute the generators of the Galois group by numerical path tracking along a set of generators of the fundamental group of ${\cal L}_h \setminus \nabla_h$ (in our notation). 
Certifying this computation requires certified path tracking and a certified witness set (in the sense of \cite{hauenstein2018numerical}) for the discriminant hypersurface $\nabla_h$. 
In terms of performance, discussions with the authors of \cite{GuillemotLairez_2024} made us aware that the computations of monodromy groups of lines of the cubic surfaces for the examples we consider in this paper are within the reach of recent improvements of certified path tracking methods and implementations, with comparable efficiency.
Algorithm \ref{alg:main} differs from path-tracking as it rests on period computations and differential tools using Picard-Fuchs equations instead and uses semi-numerical algorithms introduced in \cite{LairezEtAl_2024}.  
We still need witness sets of $\nabla_h$, which we obtain by computing its defining equation.

We emphasise that the general scope of Algorithm \ref{alg:main} is different from computing enumerative Galois groups.
Namely, we compute the action of monodromy on homology groups.
For cubic surfaces, these two objectives happen to coincide.
This fails in other cases.
For instance, lines contained in algebraic varieties need not be isolated in their homology class. 
It is known that there are 2\,875 lines on generic threefolds of degree five \cite{Schubert_1886}, but their second homology group has rank~$1$.
There is therefore no hope of recovering the action of monodromy on the lines from the action of monodromy on the second homology group.
Conversely, the action of monodromy on enumerative data of a variety may be insufficient to recover the action of monodromy on the homology. 
Indeed, a first obstruction is that the monodromy group may be infinite.
This is the case for families of K3 surfaces, to which Algorithm \ref{alg:main} is applied in \secref{sec:quartic}.
There, we recover in particular the action of monodromy on the finite number of conics generically lying on these surfaces --- this is only a quotient of the full (infinite) monodromy group, which we also compute. 
Similar to the cubic surface case, one could alternatively set up a system of parametric polynomial equations whose solutions are these conics. 
In our experience, recovering the monodromy group via certified path tracking on such a parametric system is infeasible with state-of-the-art implementations such as \cite{GuillemotLairez_2024}. 
More precisely, for $\mathbb Z_2$-invariant quartic surfaces, which have the lowest number of conics, our runs for a single loop of $\pi_1(\mathbb P^n\setminus \nabla_h)$ did not terminate for some of the 56 conics after 7~days.

Our paper is outlined as follows. 
Sections \ref{sec:discriminants}-\ref{sec:Picard-Fuchs} explain the steps of Algorithm \ref{alg:main} in detail. 
\secref{sec:discriminants} contains our method for computing the defining equation of $\nabla_h$. 
Via this equation we certify that a candidate \emph{generic} line $L$ is indeed sufficiently generic to apply Lefschetz's theorem (\thmref{thm:lefschetz_zariski}), which justifies step 2 in Algorithm \ref{alg:main}. 
\secref{sec:deRham} explains how to compute a period matrix for $X_b$. Analytic continuation of this period matrix along a loop $\ell_i$ is done by solving the Picard--Fuchs equation, as explained in \secref{sec:Picard-Fuchs}. 
\secref{sec:cubicsurfaces} contains the details of our computation which lead to \thmref{thm:main_new}. In Section \ref{sec:crystal_new}, we discuss Galois groups under crystallographic symmetries. 
In \secref{sec:quartic}, we discuss monodromy groups of certain families of K3 surfaces with symmetries, and compute the Galois groups of configurations of curves on these surfaces.
Our code and data can be downloaded at \cite{code}.

\section*{Acknowledgements}
The problem of computing the action of monodromy on the lines of cubic surfaces through period computations was first suggested to the first author by Duco van Straten. 
Duco van Straten also made the first author aware of \cite{BrazeltonRaman_2024}, which was the main motivation for carrying out this project.
We thank Marc Mezzarobba for kindly sharing code and for helpful e-mail correspondence, which made the computations of Sections \ref{sec:quartic} and \ref{sec:crystallographic_quartic} possible.
We are additionally grateful to Philip Candelas, Alice Garbagnati, Alexandre Guillemot, Pierre Lairez and Rafael Mohr for enlightening discussions and helpful comments.
This work has made use of the Infinity Cluster hosted by Institut d'Astrophysique de Paris for the computations of \secref{sec:crystallographic_quartic}. We thank Stephane Rouberol for his help with running our code on the cluster.
We are finally grateful to the four anonymous referees for their helpful comments on a previous version of this text, which helped us to improve the exposition.

\section{Computing discriminants} \label{sec:discriminants}

A subtask in our computations is to find a set of generators for $\pi_1(\mathbb{P}^n \setminus \nabla,b)$. 
\thmref{thm:lefschetz_zariski} below allows us to restrict to a one-dimensional subfamily $L \subseteq \mathbb{P}^n$ containing $b$ and compute $\pi_1(L \setminus \nabla, b)$ instead. 
This requires $L$ to be \emph{generic}, in the sense that it intersects the discriminant hypersurface $\nabla$ transversally. 
To verify this, it suffices to show that the intersection $L \cap \nabla$ consists of $\deg \nabla$ distinct points. Here $\deg \nabla$ is the degree as a reduced $(n-1)$-dimensional scheme.  
Checking genericity is easy once we know this degree. 
In the context of~Table \ref{tab:main}, $\nabla_h = \nabla_G$ is a highly non-generic linear section of the degree $32$ discriminant for the complete linear system of cubic surfaces. 
The degree of its underlying reduced scheme is often smaller than $32$. 
This section explains how to compute the defining equation of the discriminant $\nabla_h$, which is step 1 in Algorithm \ref{alg:main}, in the case of a linear systems of hypersurfaces $\calL_h$. 
This gives more information than needed, as we only need to know the degree, but it is independently interesting. 

Let $h = (h_0, h_1, \ldots, h_n)$ be $n + 1$ linearly independent degree $d$ forms in $m+2$ variables. That is, $h_i \in \mathbb{C}[x_0, \ldots, x_{m+1}]_d$. The linear system ${\cal L}_h \simeq \mathbb{P}^n$ spanned by $h$ is a subsystem of $\mathbb{P} H^0(\mathbb{P}^{m+1}, {\cal O}(d))$, the degree $d$ hypersurfaces in $\mathbb{P}^{m+1}$. The identification ${\cal L}_h \simeq \mathbb{P}^n$~is
\[ {\cal L}_h \, \ni \, X_z \, = \, \{ x \in \mathbb{P}^{m+1} \, : \, z_0 \, h_0(x) + z_1 \, h_1(x) + \cdots + z_n \,  h_n(x) \, = \, 0 \} \, \, \sim \, \, (z_0: z_1 : \cdots : z_n) \,\in \,  \mathbb{P}^n. \] 
The \emph{discriminant locus} $\nabla_h$ of this linear system consists of the singular hypersurfaces in ${\cal L}_h$. 
We shall assume throughout the section that $\nabla_h$ is a hypersurface, so that it is described by a single equation $\Delta_h = 0$, with $\Delta_h \in \mathbb{C}[z_0, \ldots, z_n]$. 
This is equivalent to assuming that a generic element of ${\cal L}_h$ is smooth. 
To test this assumption, it suffices to find one smooth hypersurface $X_z \in {\cal L}_h$. 
The polynomial $\Delta_h$ is called the \emph{discriminant polynomial} of ${\cal L}_h$. 
It is defined up to a nonzero scalar factor. Our goal in this section is to compute $\Delta_h$. 
We assume that the reader is familiar with algorithmic algebraic geometry at the level of the textbook~\cite{CoxLittleOSheaIVA}. All computations reported in this section are performed using the Julia package \texttt{Oscar.jl} \cite{OSCAR}. 

A point $z \in \mathbb{P}^n$ belongs to $\nabla_h$ if and only if there exists $x \in \mathbb{P}^{m+1}$ so that $(x,z) \in Y$, with 
\[ Y =  \left \{ (x, z) \in \mathbb{P}^{m+1} \times \mathbb{P}^n \, \st \, \frac{\partial}{\partial x_i}(z_0  h_0(x) + z_1  h_1(x) + \cdots + z_n   h_n(x))  =  0, \, i = 0, \ldots, m+1 \right \}. \]
In other words, we have $\nabla_h = \pr_2(Y)$, where $\pr_2: Y \rightarrow \mathbb{P}^n$ is the natural coordinate projection. 
An algorithm for computing $\Delta_h$ is as follows. 
Consider the ideal $I(Y)$ in the ring 
$\mathbb{C}[x,z] = \mathbb{C}[x_0, \ldots, x_{m+1}, z_0, \ldots, z_n]$ generated by the $m+2$ defining equations of $Y$. 
First, compute the saturation of $I(Y)$ by the irrelevant ideal $\mathfrak{m} = {\langle x_0, \ldots, x_{m+1} \rangle}$ of $\mathbb{P}^{m+1}$: \[I(Y) : \mathfrak{m}^\infty = \{ p \in \mathbb{C}[x,z] \, :\, p \cdot \mathfrak{m}^k \subseteq I(Y) \text{ for some } k \in \mathbb{N} \}.\]
From this, eliminate the variables $x_0, \ldots, x_{m+1}$. 
Both steps are standard operations in computer algebra. 
The radical of the elimination ideal $(I(Y) : \mathfrak{m}^\infty) \cap \mathbb{C}[z_0, \ldots, z_n]$ is~$\langle \Delta_h \rangle$. 

A first improvement to this algorithm is to replace ``saturation and elimination'' by ``elimination on each affine chart of $\mathbb{P}^m$''. Concretely, we compute the intersection of the following $m+2$ elimination ideals: $(I(Y) + \langle x_i - 1 \rangle) \cap \mathbb{C}[z_0, \ldots, z_n]$, $i = 0, \ldots, m+1$. This may seem to be more of a detour, but it is in fact significantly more efficient in our experiments. 
\begin{example} \label{ex:S4}
The following polynomials generate the linear system of $S_4$-invariant cubics:
\[ \begin{matrix} h_0 \, = \, x_0^3 + x_1^3 + x_2^3 + x_3^3, \\
h_1\, = \,  x_0^2 x_1 + x_0^2 x_2 + x_0^2 x_3 + x_0 x_1^2 + x_0 x_2^2 + x_0 x_3^2 + x_1^2 x_2 + x_1^2 x_3 + x_1 x_2^2 + x_1 x_3^2 + x_2^2 x_3 + x_2 x_3^2, \\
 h_2 \, = \, x_0 x_1 x_2 + x_0 x_1 x_3 + x_0 x_2 x_3 + x_1 x_2 x_3.\end{matrix} \]
 In particular, ${\cal L}_{S_4}$ has dimension $n = 2$. The generator of $(I(Y) : \mathfrak{m}^\infty) \cap \mathbb{C}[z_0, z_1, z_2]$ is  \small
\[ (3 z_0 + z_1 - z_2)^2   (z_0 + 3 z_1 + z_2)  (3 z_0 - 3 z_1 + z_2)  (9 z_0^3 + 9 z_0^2 z_1 - 3 z_0^2 z_2 - 9 z_0 z_1^2 - 6 z_0 z_1 z_2 + 4 z_0 z_2^2 + 7 z_1^3 - 3 z_1^2 z_2).  \]
\normalsize
Saturating and then eliminating takes about 15 seconds. Computing four elimination ideals instead takes about 0.03 seconds. For our purposes, we disregard the scheme structure and replace the exponent $2$ of the first linear factor by $1$ in the discriminant polynomial $\Delta_h$. 
\end{example}
This method generalises easily to the arbitrary families of hypersurfaces from the Introduction. 
Unfortunately, this naive elimination algorithm only terminates in reasonable time for small cases. 
To compute the discriminants in \tabref{tab:main}, we use an alternative approach specific to linear systems. 
We analyse the irreducible components of $Y$ by considering the Jacobian matrix $J_h$ of $h$: 
\[ J_h(x) \, = \, \left ( \frac{\partial h_j}{\partial x_i}\right )_{\substack{i = 0, \ldots, m+1 \\ j = 0, \ldots, n}} \, \, \in \, \mathbb{C}[x_0, \ldots, x_{m+1}]^{(m+2) \times (n+1)}. \]
Its rank over the function field $\mathbb{C}(x_0, \ldots, x_{m+1})$ is denoted by $r$. We define 
\[ U \, = \, \{ x \in \mathbb{P}^{m+1} \, : \, {\rm rank} \, J_h(x) = r \} \quad \text{and} \quad V \, = \, \mathbb{P}^{m+1} \setminus U. \]
Notice that $V \subset \mathbb{P}^{m+1}$ is a closed subvariety defined by the vanishing of the $r \times r$-minors of $J_h(x)$. By construction, we have $Y = Y_0 \cup Y_1$, where $Y_0 = \overline{Y \cap (U \times \mathbb{P}^n)}$ and $Y_1 = Y \cap (V \times \mathbb{P}^n)$. As a consequence, the discriminant is $\nabla_h = \pr_2(Y_0) \cup \pr_2(Y_1)$. In our computations, it has proved effective to compute $\pr_2(Y_0)$ and $\pr_2(Y_1)$ separately. We start with $\pr_2(Y_1)$. 

For this component, our strategy is a rather straightforward symbolic elimination which terminates in less than two minutes for each example in \tabref{tab:main}. First, compute the minimal primes $\mathfrak{p}_1, \ldots, \mathfrak{p}_\ell$ of the ideal $I(V)$ generated by the $r \times r$ minors of $J_h(x)$. Next, for each prime $\mathfrak{p}_i$, one could compute $J_i = (I(Y) + \mathfrak{p}_i) : \mathfrak{m}^\infty$ and eliminate $x_0, \ldots, x_{m+1}$ from $J_i$. Again, for efficiency reasons, we prefer to perform $m+2$ eliminations instead. For $j = 0, \ldots, m+1$, we compute $E_{ij} = (I(Y) + \mathfrak{p}_i + \langle x_j - 1 \rangle ) \cap \mathbb{C}[z_0,\ldots, z_n]$. The elimination ideal $E_{ij}$ defines a variety in $\mathbb{P}^n$. Its components of dimension $n-1$ are hypersurfaces contained in $\pr_2(Y_1)$. Since we know that $\nabla_h$ is a hypersurface, components of smaller dimension can be discarded. Hence, intersecting all minimal primes of $E_{ij}$ of the appropriate dimension for $i = 1, \ldots, \ell$ and $j = 0, \ldots, m+1$ gives the defining equation of $\pr_2(Y_1)$.

We now turn to the computation of $\pr_2(Y_0)$. This is nonempty only when $r < n+1$. 

\begin{lemma} \label{lem:Y0}
If $r = n+1$, then $Y_0 = \emptyset$. If $r < n + 1$, then the variety $Y_0$ is irreducible of dimension $m+1+n-r$. Hence, $\pr_2(Y_0) \neq \emptyset$ if and only if $r < n+1$, and it is irreducible.   
\end{lemma}

\begin{proof}
Let $\pi_1: Y \rightarrow \mathbb{P}^{m+1}$ be the coordinate projection $(x,z) \mapsto x$. Notice that the fibre $\pi_1^{-1}(x)$ is the linear space defined by $J_h(x) \cdot z = 0$. Therefore, if ${\rm rank}(J_h(x)) = n+1$, we have $\pi_1^{-1}(x) = \emptyset$. If this is true for all $x \in U$, then $Y \cap (U \times \mathbb{P}^n) = \emptyset$ and hence $Y_0 = \emptyset$. If $r < n+1$, then $\pi_1: Y \cap (U \times \mathbb{P}^n) \rightarrow U$ is surjective. By our description of the fibres, it is a $\mathbb{P}^{n-r}$-bundle over $U$. Hence it is irreducible of dimension $\dim U + n-r$, and so is $Y_0$. 
\end{proof}

\begin{example} \label{ex:S3}
For the linear system ${\cal L}_{S_4}$ from Example \ref{ex:S4}, we have $Y_0 = \emptyset$. The Jacobian matrix $J_h$ has rank $3 = n+1$. Instead, let us consider the invariants of the $S_3$-action which permutes the first three variables $x_0,x_1,x_2$. The linear system ${\cal L}_{S_3} \simeq \mathbb{P}^6$ is generated by
\[ \begin{matrix} h_0 \, = \,  x_3^3, \quad h_1 \, = \, x_0 x_3^2 + x_1 x_3^2 + x_2 x_3^2, \quad h_2 \, = \, x_0^2 x_3 + x_1^2 x_3 + x_2^2 x_3, \\ 
 h_3 \, = \, x_0 x_1 x_3 + x_0 x_2 x_3 + x_1 x_2 x_3, \quad 
 h_4 \, = \, x_0^3 + x_1^3 + x_2^3, \\ 
 h_5 \, = \, x_0^2 x_1 + x_0^2 x_2 + x_0 x_1^2 + x_0 x_2^2 + x_1^2 x_2 + x_1 x_2^2, \quad 
 h_6 \, = \, x_0 x_1 x_2.
 \end{matrix} 
 \]
The rank of $J_h$ is four, so $Y_0$ is irreducible of dimension $3 + 6 - 4 = 5$. 
\end{example}

Motivated by Lemma \ref{lem:Y0}, we assume that $r < n+1$ for the rest of the discussion. We shall identify $\pr_2(Y_0) \subset \mathbb{P}^n$ as the projective dual variety of a unirational variety $Z \subset (\mathbb{P}^n)^*$. That variety is obtained as the closure of the image of the rational map $\phi_h$ defined by $h$: 
\[ \phi_h: U \dashrightarrow (\mathbb{P}^n)^*, \quad \phi_h(x) \, = \, (h_0(x): h_1(x): \cdots: h_n(x)). \]
In symbols, we set $Z = \overline{\phi_h(U)}$. The base locus of this map is $B = \{ x \in U \, : \, h_0(x) = \cdots = h_n(x) = 0 \}$. Let $Z_{\rm sm} \subseteq Z$ denote the open subset of smooth points of $Z$, and let $U_{\rm sm} = \phi_h^{-1}(Z_{\rm sm})$ be the corresponding open subset of $U$. For $x \in U_{\rm sm}$, the projectivisation of the row span of $J_h(x)$, viewed as an $(r-1)$-plane in $(\mathbb{P}^n)^*$, is the tangent space of $Z$ at $\phi_h(x)$. In particular, we have $\dim Z = r-1$. Recall that the projective dual variety $Z^* \subset \mathbb{P}^n$ of $Z$ is the closure of all points $z \in \mathbb{P}^n$ such that the hyperplane $H_z = \{ u \in (\mathbb{P}^n)^* \, : \, z_0 \, u_0 + z_1 \, u_1 + \cdots + z_n \, u_n = 0 \}$ contains the tangent space $T_p Z$ of $Z$ at some point $p \in Z_{\rm sm}$. 

\begin{lemma} \label{lem:projdual}
With the above notation, the variety $\pr_2(Y_0)$ is the projective dual variety of $Z$. 
\end{lemma}
\begin{proof}
This follows from the following chain of equalities: 
\begin{align} \label{eq:chainofeq}
\begin{split}
Z^* & \, = \, {\rm cl} \{ z \in \mathbb{P}^n \, : \, H_z \text{ contains } T_p Z \text{ for some } p \in Z_{\rm sm} \} \\ 
& \, = \, {\rm cl} \{ z \in \mathbb{P}^n \, : \, H_z \text{ contains } T_p Z \text{ for some } p \in \phi_h(U_{\rm sm}) \} \\
& \, = \, {\rm cl} \{ z \in \mathbb{P}^n \, : \, \exists \,  x \in U_{\rm sm} \text{ such that } J_h(x) \cdot z = 0 \} \\ 
& \, = \, \overline{ \pr_2(Y \cap (U_{\rm sm} \times \mathbb{P}^n))} \\ 
& \, = \, \pr_2(\overline{Y \cap (U_{\rm sm} \times \mathbb{P}^n)}) \, = \, \pr_2(Y_0).
\end{split}
\end{align}
Here ${\rm cl}$ denotes the Zariski closure. 
The first equality is the definition of $Z^*$. 
The second equality holds because $\phi_h(U_{\rm sm})$ is dense in the smooth points $Z_{\rm sm}$. 
The third equality uses $p = \phi_h(x)$ for some $x \in U_{\rm sm}$ and follows from the observation that the tangent space $T_p Z$ is the projectivised row span of $J_h(x)$. 
The fourth equality holds by definition of $Y$. 
The fifth equality uses the fact that the coordinate projection of a closed subvariety in $\mathbb{P}^{m+1} \times \mathbb{P}^n$ is closed. 
The last equality holds because, by the same argument as that in the proof of \lemref{lem:Y0}, $Y \cap (U_{\rm sm} \times \mathbb{P}^n)$ is irreducible of dimension $m+1+n-r$. 
Hence its closure, which is trivially contained in $Y_0$, must be equal to $Y_0$ for dimension reasons. 
\end{proof}

We can parametrise ${\rm pr}_2(Y_0)$ as follows. Let $K_h(x) \in \mathbb{C}(x_0,\ldots, x_{m+1})^{(n+1) \times (n+1-r)}$ be a kernel matrix of $J_h(x)$ over the field of rational functions in $x$. That is, $K_h(x)$ has rank $n+1-r$ and $J_h(x)K_h(x) = 0$. By \eqref{eq:chainofeq}, the image of the rational map 
\begin{equation} \label{eq:psih} \psi_h: \mathbb{P}^{m+1} \times \mathbb{P}^{n-r} \, \dashrightarrow \, \mathbb{P}^n, \quad \quad \psi_h(x, v) \, = \, K_h(x) \cdot v\end{equation}
is dense in $Z^* = {\rm pr}_2(Y_0)$. Here $K_h(x) \cdot v$ is the matrix-vector product of $K_h(x)$ with an $(n+1-r)$-vector of homogeneous coordinates for $v$, and the result is interpreted as a vector of homogeneous coordinates on $\mathbb{P}^n$. As a consequence, the dimension of ${\rm pr}_2(Y_0)$ is the rank of the Jacobian matrix of $\psi_h$ over $\mathbb{C}(x,v)$ minus one. If $\dim {\rm pr}_2(Y_0) < n-1$, then ${\rm pr}_2(Y_0) \subset {\rm pr}_2(Y_1)$ by the fact that $\nabla_h$ is pure of dimension $n-1$. In that case, the elimination algorithm outlined above for computing ${\rm pr}_2(Y_1)$ computes the full discriminant $\Delta_h$. 
\begin{example} \label{ex:S3-2}
For the linear system ${\cal L}_{S_3}$ from Example \ref{ex:S3}, our algorithm computes that ${\rm pr}_2(Y_1)$ is a hypersurface of degree $13 = 1 + 8 + 4$ with three components: 
\begin{align*}
\Delta_{h,1} \, = \, & \, 3 z_4- 3 z_5 + z_6, \\ 
\Delta_{h,2} \, = \, & \, 2187 z_0^2 z_4^6 - 1458 z_0^2 z_4^5 z_6 - \cdots + 4 z_2 z_3^5 z_5 z_6 + 15 z_3^6 z_4^2 + 12 z_3^6 z_4 z_5 - 4 z_3^6 z_4 z_6, \\
\Delta_{h,3} \, = \,
& \, 9 z_0^2 z_4^2 + 36 z_0^2 z_4 z_5 + 6 z_0^2 z_4 z_6 + 36 z_0^2 z_5^2 + 12 z_0^2 z_5 z_6 + z_0^2 z_6^2 - 18 z_0 z_1 z_2 z_4 - 36 z_0 z_1 z_2 z_5, \\ 
& \, - 6 z_0 z_1 z_2 z_6 - 18 z_0 z_1 z_3 z_4 - 36 z_0 z_1 z_3 z_5 - 6 z_0 z_1 z_3 z_6 + 4 z_0 z_2^3 + 12 z_0 z_2^2 z_3 + 12 z_0 z_2 z_3^2 \\ &+ 4 z_0 z_3^3 + 12 z_1^3 z_4 + 24 z_1^3 z_5 + 4 z_1^3 z_6 - 3 z_1^2 z_2^2 - 6 z_1^2 z_2 z_3 - 3 z_1^2 z_3^2.
\end{align*}
The octic $\Delta_{h,2}$ has 233 terms. 
It can be found at \cite{code}. 
The Jacobian matrix of $\psi_h$ has rank $5$, so ${\rm pr}_2(Y_0)$ has codimension 2. 
In fact, ${\rm pr}_2(Y_0)$ is the intersection of the hyperplane $\Delta_{h,1} = 0$ and the octic hypersurface $\Delta_{h,2} = 0$.  
The ideal $\langle \Delta_{h,1}, \Delta_{h,2} \rangle$ is primary, and its radical is the vanishing ideal of ${\rm pr}_2(Y_0)$. 
Since $\dim Y_0 = 5$ by Example \ref{ex:S3}, we learn that the generic fibre of ${\rm pr}_2: Y_0 \rightarrow {\rm pr}_2(Y_0)$ has dimension one. 
We have $\Delta_h = \prod_{i=1}^3 \Delta_{h,i}$.
\end{example}

If $\dim {\rm pr}_2(Y_0) = n-1$, then its defining equation is a factor of $\Delta_h$, and we must explain how to compute it. The following lemma states how to certify correctness of a candidate.  

\begin{lemma}
If $\dim {\rm pr}_2(Y_0) = n-1$, then its defining equation is the unique, up to scalar multiple, irreducible polynomial $\Delta_{h,0} \in \mathbb{C}[z_0, \ldots, z_n]$ satisfying $\Delta_{h,0} \circ \psi_h = 0$. If $h_i$ has coefficients in $\mathbb{Q}$ for $i = 0, \ldots, n$, then $\Delta_{h,0}$ can be scaled to have integer coefficients. 
\end{lemma}

\begin{proof}
This is an immediate consequence of the fact that ${\rm pr}_2(Y_0)$ is parametrised by $\psi_h$. 
\end{proof}
In particular, if $\Delta_{h,0} \in \mathbb{Z}[z_0, \ldots, z_n]$ is a candidate for the defining equation of ${\rm pr}_2(Y_0)$, then it can be certified a posteriori by checking that $\Delta_{h,0}$ is irreducible, and it vanishes on the coordinates of the parametrisation $\psi_h$, which are rational functions with coefficients in $\mathbb{Q}$. 

It remains to explain how we compute a candidate polynomial $\Delta_{h,0}$. For this we use interpolation. We fix a finite set of monomials $z^\alpha, \alpha \in A \subset \mathbb{N}^n$ and make the Ansatz
\begin{equation} \label{eq:ansatz} \Delta_{h,0} \, = \, \sum_{\alpha \in A} c_\alpha \, z^\alpha\end{equation}
for unknown coefficients $c_\alpha$.
 The parametrisation \eqref{eq:psih} makes it easy to sample rational points on ${\rm pr}_2(Y_0)$. Each such a sample point $z^*$ gives a $\mathbb{Q}$-linear condition on the coefficients $c_\alpha$ by requiring that \eqref{eq:ansatz} vanishes at $z^*$. If the irreducible defining equation of ${\rm pr}_2(Y_0)$ fits the Ansatz \eqref{eq:ansatz}, then after gathering sufficiently many sample points, there is only one solution to these linear equations. If it does not, then sufficiently many samples make our linear equations infeasible, meaning that only $c_\alpha = 0$ is a solution, and we must change the Ansatz. 
 
 Choosing the exponents $A$ may be challenging. A first strategy is to include all monomials of a fixed degree $\delta$. One can then simply increase $\delta$ until a solution is found. A slightly more sophisticated approach takes the multi-homogeneity of the discriminant into account. This applies when the generators $h_i$ of the linear system are homogeneous with respect to a $(\mathbb{C}^*)^k$ action for some $k \geq 2$. Such an action induces a $\mathbb{Z}^k$-grading on the polynomial ring $\mathbb{C}[z_0, \ldots, z_n]$, with respect to which $\Delta_h$ is homogeneous. We illustrate this with two examples.  
 
\begin{example} \label{ex:multigrading}
 The polynomials $h_0, \ldots, h_6$ from Example \ref{ex:S3} are homogeneous with respect to the $(\mathbb{C}^*)^2$-action $(\lambda, \mu) \cdot (x_0, x_1, x_2, x_3) = (\lambda x_0, \lambda x_1, \lambda x_2, \mu x_3)$. Their bidegrees are 
 \[ (0, 3), \quad (1, 2), \quad (2, 1), \quad (2, 1), \quad (3, 0), \quad (3, 0), \quad \text{ and } (3, 0) \]
 respectively. The discriminants $\Delta_{h,i}$ in Example \ref{ex:S3-2} are homogeneous with respect to the corresponding $\mathbb{Z}^2$-grading given by $\deg(z_0) = (0,3), \, \deg(z_1) = (1,2), \, \ldots, \, \deg(z_6) = (3,0)$.
There are only $22$ monomials of bidegree $(6,6)$ in this grading, all of which appear in $\Delta_{h,3}$. 
\end{example}
\begin{example} \label{ex:multigrading2}
For $G = \mathbb{Z}_3$, we have $n = 7$ and our linear system is generated by 
\[ \begin{matrix} x_3^3, \, \,  x_0 x_3^2 + x_1 x_3^2 + x_2 x_3^2, \, \,  x_0^2 x_3 + x_1^2 x_3 + x_2^2 x_3, \, \,  x_0 x_1 x_3 + x_0 x_2 x_3, \, \,  x_1 x_2 x_3 \\
x_0^3 + x_1^3 + x_2^3, \, \,  x_0 x_1^2 + x_0^2 x_2 + x_1 x_2^2, \, \,  x_0^2 x_1 + x_1^2 x_2 + x_0 x_2^2, \, \,  x_0 x_1 x_2.\end{matrix} \]
These are homogeneous with respect to the $(\mathbb{C}^*)^2$ action $(\lambda, \mu) \cdot (x_0,x_1,x_2,x_3) = (\lambda x_0, \lambda x_1, \lambda x_2, \mu x_3)$. The induced bigrading on $\mathbb{C}[z_0, \ldots, z_7]$ is $\deg(z_0) = (0,3), \, \deg(z_1) = (1,2), \, \ldots, \, \deg(z_7) = (3,0)$. We find that $\Delta_{h,0}$ is homogeneous of bidegree $(18,6)$. Out of the 667 monomials of bidegree $(18,6)$, 510 appear in $\Delta_{h,0}$, see \cite{code}. The number 667 should be compared to the number of monomials in 8 variables of total degree 8, which is 6435.
\end{example}

In our current implementation, we identified multigradings like that of Examples \ref{ex:multigrading} and \ref{ex:multigrading2} manually. 
A list is found at \cite{code}. 
This is crucial to make an efficient Ansatz \eqref{eq:ansatz}.

\section{The fundamental group and Lefschetz's theorem} \label{sec:lefschetz}

This section briefly explains our strategy for computing generators $\ell_1, \ldots, \ell_r \in \pi_1(\mathbb{P}^n \setminus \nabla, b)$, i.e., step 2 in Algorithm \ref{alg:main}. 
As in the Introduction, let $\calX$ be the family of hypersurfaces from \eqref{eq:calX}, and let $\nabla$ be its discriminant locus.
Choose $b\in \bbP^n \setminus \nabla$.
When $n = 1$, the fundamental group $\pi_1(\mathbb{P}^1 \setminus \nabla, b)$ is easily computed: $\nabla$ consists of finitely many points, and simple loops around these points generate $\pi_1(\bbP^1\setminus\nabla, b)$.
For $n >1$ we can restrict to a line $L$ containing $b$: the inclusion $L\subset \bbP^n$ induces a group morphism
\begin{equation}
\psi \, \colon \, \pi_1(L\setminus\nabla,b) \, \to \,  \pi_1(\bbP^n\setminus\nabla,b)\,,
\end{equation}
which is surjective when $L$ is sufficiently generic. This follows from the next theorem.

\begin{theorem}\label{thm:lefschetz_zariski}
Let $\nabla$ be a hypersurface in $\bbP^n$ with $n\ge 2$. 
Let $H \subset \mathbb{P}^n$ be a generic hyperplane.  
The inclusion $H\setminus\nabla\to \bbP^n\setminus\nabla$ induces a surjective group homomorphism
\begin{equation}
\psi \, \colon \, \pi_1(H\setminus\nabla) \, \to \,  \pi_1(\bbP^n\setminus\nabla)\,.
\end{equation}
\end{theorem}

The word \emph{generic} in Theorem \ref{thm:lefschetz_zariski} means that $H$ is transversal to a Whitney stratification of $\nabla$. This holds for a dense open subset of hyperplanes in the dual projective space $(\mathbb{P}^n)^*$. 

\thmref{thm:lefschetz_zariski} dates back to Lefschetz \cite{Lefschetz_1924}. 
A more modern treatment can be found in \cite[\S7.4.1]{Lamotke_1981}.
A stronger result of Zariski \cite{Zariski_1937} states that $\psi$ is in fact an isomorphism when $n\ge 3$.
A rigorous proof of this theorem was first given in \cite{HammTrang_1973}, and a more topological proof was found by Chéniot \cite{Cheniot_1974}.
When $n=2$, a result of Van Kampen \cite{VanKampen_1933} gives the kernel of $\psi$, and again, a topological proof was given in \cite{Cheniot_1973a}. Surjectivity of $\psi$
suffices for our purposes.

Repeatedly applying \thmref{thm:lefschetz_zariski} gives the following strategy for computing generators of $\pi_1(\mathbb{P}^n \setminus \nabla, b)$. Randomly choose a line $L \subset \mathbb{P}^n$ and compute the finite set of points $\nabla \cap L$.
Next, compute loops $\ell_1, \dots,\ell_r$ in $L$ around each of the points of $\nabla \cap L$.
By \thmref{thm:lefschetz_zariski}, the loops $\ell_1, \dots,\ell_r$ generate $\pi_1(\bbP^n\setminus\nabla,b)$.

The correctness of the outlined algorithm relies on the fact that $L$ can be constructed as an iterated intersection $( \cdots ((H_1 \cap H_2) \cap H_3) \cap \cdots \cap H_{n-1})$ in which each intermediate intersection $H_1 \cap \cdots \cap H_i$ is a generic hyperplane in $H_1 \cap \cdots \cap H_{i-1} \simeq \mathbb{P}^{n-i+1}$ in the sense of \thmref{thm:lefschetz_zariski}. 
To certify this, it suffices to check that the intersection $\nabla \cap L$ consists of as many points as the degree of the hypersurface~$\nabla$. 
In our setting, we compute the degree of~$\nabla$ by computing its defining equation.
This computation was detailed in \secref{sec:discriminants}.

The group $\pi_1(L \setminus \nabla,b)$ is algorithmically obtained as the first homology group of the Voronoi graph of $\nabla \cap L$ inside $L \simeq \mathbb{P}^1$ --- specifically, we used the algorithm described in \cite[\S3.5.3]{LairezEtAl_2024}. 
Using Voronoi cycles instead of arbitrary loops around the discriminant $\nabla \cap L$ has the numerical advantage that we stay as far away from the discriminant as possible.

\section{De Rham cohomology and period matrices} \label{sec:deRham}

This section recalls facts about de Rham cohomology and period matrices which we will leverage in the next section to compute the action of monodromy on the homology.

We consider any smooth projective hypersurface $X\subset \mathbb P^{m+1}$ defined over a number field $K$.
For our applications, $K$ can be taken to be $\Q$ throughout the text.
Let $\Hdr^k(X)$ be the $k$-th algebraic de Rham cohomology group over $K$ of the hypersurface $X$, and let $H_k(X)$ be its $k$-th singular homology group with coefficients in $\Z$.
The \emph{de Rham pairing} or \emph{integration pairing} $\Hdr^k(X) \times H_k(X) \rightarrow \mathbb{C}$ sends a pair $(\omega,\gamma) \in \Hdr^k(X) \times H_k(X)$ to the value of the integral $\int_\gamma\omega$. 
Complex numbers in the image of this map are called \emph{periods} \cite{CarlsonEtAl_2017,Griffiths_1969}. 
The de Rham pairing is perfect, i.e., it realises $\Hdr^k(X)\otimes_K \C$ and $H_k(X) \otimes_\Z \C$ as duals of each other in a canonical way \cite{DeRham_1931, Grothendieck_1966}.
The \emph{$k$-th period matrix} of $X$ is the matrix of the de Rham pairing.
When $X$ is a projective hypersurface of dimension $m$, only the $m$-th period matrix is interesting. Indeed, as a consequence of Lefschetz's Hyperplane theorem and Poincaré duality, all other matrices are either $0\times0$ or $1\times 1$ matrices \cite[p. 53 and 156]{GriffithsHarris_1978}.
In what follows, by \emph{the period matrix} of $X$ we mean the $m$-th period matrix.

The period matrix depends on a choice of bases for $\Hdr^m(X)$ and $H_m(X)$.
The homology group $H_m(X)$ has the structure of an integral lattice, with the topological intersection product $(\gamma_1, \gamma_2)\mapsto \gamma_1 \cdot \gamma_2$.
We will always work with an integral basis of the homology.
Although this is not particularly relevant to this text, we want to mention that the cohomology has additional structure stemming from the Hodge filtration.
The method we use to obtain a basis of cohomology, Griffiths--Dwork reduction, preserves this information as well.

The cohomology group $\Hdr^{m}(X)$ can be decomposed into two main parts: one coming from the cohomology of the ambient space $\Hdr^{m+1}(\bbP^{m+1})$, and one coming from the cohomology of the complement $\Hdr^{m+1}(\bbP^{m+1}\setminus X)$. Let us start with the former.

The cohomology of the projective space is well understood: $\Hdr^{m+1}(\bbP^{m+1}) \simeq K$ if $m$ is even and trivial otherwise. Hence, the corresponding direct summand of $\Hdr^m(X)$ is either one- or zero-dimensional. 
If $m$ is even, a generator is given by the dual of the homology class $H \in H_m(X)$ of the intersection of $X$ with a generic linear subspace of complex dimension $m/2+1$. 
That is, the corresponding form in $\Hdr^{m}(X)$ is the cocycle $\omega_0: \gamma \mapsto  \gamma \cdot  H$ which sends an $m$-cycle $\gamma$ to its intersection product with $H$.
In particular, $H$ is invariant under monodromy, and so is $\omega_0$. 
When $m$ is odd, we set $H=0$ for consistency. 

The other piece is called the \emph{primitive de Rham cohomology}, or simply \emph{primitive cohomology} of $X$. It is denoted by $\PHdr^m(X)$ and is typically much richer than the ambient piece. 
It is the subspace of $\Hdr^m(X)$ of forms annihilated by the hyperplane class:
\begin{equation}
\PHdr^m(X) := \left\{\omega\in \Hdr^m(X) \st \int_{H}\omega = 0\right\}
\end{equation}
There is a linear isomorphism called the \emph{residue map}, see  \cite[\S5.3]{CoxKatz_1999} and \cite[\S8]{Griffiths_1969}:
\begin{equation}
\res \, \colon \, \Hdr^{m+1}(\bbP^{m+1}\setminus X)\, \, \stackrel{\sim}{\smash{\longrightarrow}\rule{0pt}{0.4ex}} \, \,  \PHdr^m(X)\,.
\end{equation}
The space $\Hdr^{m+1}(\bbP^{m+1}\setminus X)$, and therefore $\PHdr^m(X)$, can be computed in terms of rational functions \cite{Grothendieck_1966}. 
That is, when $f \in K[x_0, \dots, x_{m+1}]_d$ is the defining equation of $X$, we have
\begin{equation}\label{eq:primitive_homology}
\Hdr^{m+1}(\bbP^{m+1}\setminus X) \, \simeq \,  \frac{ \operatorname{Vect}_K \left\{ \frac{a}{f^k} \st k\ge 0\text{ and } a\in K[x]_{kd - m- 2} \right\} }{ \operatorname{Vect}_K \left\{ \frac{\partial}{\partial x_i}\frac{b}{f^k} \st k\ge 0\text{ and } 0\le i\le m+1\text{ and } b\in K[x]_{kd - m- 1} \right\} }, 
\end{equation}
where $x=(x_0, \dots, x_{m+1})$. The class of $\frac{a}{f^k}$ represents the $(m+1)$-form $\frac{a}{f^k}\, \Omega_{m+1}$ where
 \begin{equation} \label{eq:Omega}
 \Omega_{m+1} \, = \,  \sum_{i=0}^{m+1} (-1)^ix_i  \, \ud x_0\wedge\cdots\widehat{\ud x_i}\wedge \cdots \ud x_{m+1}
\end{equation}
  is the volume form on $\bbP^{m+1}$.
  We note that the Hodge filtration coincides with the filtration by the pole order $k$ in \eqref{eq:primitive_homology} by a result of Griffiths \cite{Griffiths_1969}.
  A basis of the primitive cohomology given by the class of rational functions can be computed via Griffiths--Dwork reduction. We recommend \cite[\S5.3]{CoxKatz_1999} and \cite{Lairez_2016} for more details and an introduction to this topic.
  
\begin{example}\label{ex:fermat_cohomology}
The Fermat cubic surface $X$ is defined by the equation $f=x_0^3+x_1^3+x_2^3+x_3^3 = 0$. 
The Griffiths--Dwork reduction yields the following basis of primitive cohomology:
\begin{equation}
\begin{gathered}
\omega_1 \, = \, \res \frac{x_0x_1}{f^2}\, \Omega_3\,, \quad 
\omega_2 \, = \, \res \frac{x_0x_2}{f^2}\, \Omega_3\,, \quad 
\omega_3 \, = \, \res \frac{x_0x_3}{f^2}\, \Omega_3\,, \\
\omega_4 \, = \, \res \frac{x_1x_2}{f^2}\, \Omega_3\,, \quad 
\omega_5 \, = \, \res \frac{x_1x_3}{f^2}\, \Omega_3\,, \quad 
\omega_6 \, = \, \res \frac{x_2x_3}{f^2}\, \Omega_3\,.
\end{gathered}
\end{equation}
The classes of $\Hdr^2(X)$ are purely of Hodge type $(1,1)$, as every cohomology form can be written with a pole order of 2.
This reflects the fact that $X$ is a smooth rational surface.
\end{example}

\subsection{Group action on the homology}\label{sec:group_action}

Let $G$ be a group acting by linear automorphisms on $\mathbb P^{m+1}$ and let $f\in K[x_0, \dots, x_{m+1}]_d$ be invariant under the action of $G$, i.e. $g^*f = f(g^{-1}\, x) = f$ for all $g\in G$. Further assume that the vanishing locus $X:=V(f)$ is smooth.
Then, from the observations made above, we can compute the action of $G$ on the cohomology $\Hdr^m(X)$ of $X$ in the following manner.
First of all, the image of a linear subspace under a linear automorphism is still a linear subspace, and thus $H$ is invariant under the action of $G$. In particular, $g^*\omega_0 = \omega_0$, where $\omega_0$ is the cocycle $\gamma \mapsto \gamma \cdot H$. The action on ${\rm PH}_{\rm dR}^m(X)$ is computed via the following formula: for $g\in G$, we have that
\begin{equation}\label{eq:group_action}
g^*\res \left(\frac{a}{f^{k}}\Omega_{m+1}\right)  
= \res \left(\frac{g^*a}{g^*f^{k}}g^*\Omega_{m+1}\right) 
= \det(g) \res \left(\frac{g^*a}{f^{k}}\Omega_{m+1}\right)\,,
\end{equation}
for all $k\ge 0$ and $a\in K[x_0, \dots, x_{m+1}]_{kd-m-2}$.
Hence, given a basis of $\Hdr^m(X)$ of primitive cohomology classes represented by pairs $(a_i, k_i)$, along with $\omega_0$ in even dimension, one can use Griffiths--Dwork reduction to compute the matrix $\operatorname{Mat}_{g^*}$ of the action of $g$ on $\Hdr^m(X)$.
\begin{example}\label{ex:Fermat_cubic_action_cohomology}
Let $X = V(x_0^3+x_1^3+x_2^3+x_3^3)\subset \bbP^3$ be the Fermat cubic surface.
We compute the action on $\Hdr^2(X)$ of the transposition $\tau = (x_0\,x_1)$ which acts by the linear automorphism
\begin{equation}
 [x_0:x_1:x_2:x_3]\mapsto [x_1:x_0:x_2:x_3]\,.
\end{equation}
We work in the basis $\omega_1, \dots, \omega_6$ computed in \exref{ex:fermat_cohomology}.
The formula \eqref{eq:group_action} yields that the action of $\tau$ is given by $\tau^*\omega_0 = \omega_0$ and
\begin{equation*}
\tau^*\omega_1 = -\omega_1, \quad
\tau^*\omega_2 = -\omega_4, \quad
\tau^*\omega_3 = -\omega_5, \quad 
\tau^*\omega_4 = -\omega_2, \quad
\tau^*\omega_5 = -\omega_3, \quad
\tau^*\omega_6 = -\omega_6. \qedhere
\end{equation*}
\end{example}
To recover the action on the homology, we turn to periods.
 We have that
$\int_{g_*\gamma}g^*\omega = \int_\gamma \omega$
for all $\omega\in \Hdr^m(X)$ and $\gamma \in H_m(X)$.
This translates to the matrix equation
\begin{equation} \label{eq:Matg}
	\operatorname{Mat}_{g^*}\Pi\operatorname{Mat}_{g_*} = \Pi\,,\qquad \text{i.e.,}\qquad
	\Pi^{-1}\operatorname{Mat}^{-1}_{g^*}\Pi \, = \,  \operatorname{Mat}_{g_*}  \in \operatorname{GL}(H_m(X_b))\,.
\end{equation}

We compute $\Pi^{-1}\operatorname{Mat}^{-1}_{g^*}\Pi$ numerically with certified bounds of precision. 
Doing so with precision less than one half is sufficient to recover $\operatorname{Mat}_{g_*}$ exactly, as it has integer coefficients.

\begin{example}
We continue \exref{ex:Fermat_cubic_action_cohomology}.
An integral basis of homology $h, \eta_1, \dots, \eta_6$ is computed using \cite{LairezEtAl_2024}, along with a numerical approximation of the period matrix in this basis of homology, and the basis of cohomology $\omega_0, \dots, \omega_6$.
The period matrix is given by:
\begin{equation}
\Pi = {\footnotesize \left(\begin{array}{ccccccc}
3.0 & 1.0 & 1.0 & 1.0 & 1.0 & 1.0 & 1.0 \\
0.00 & 0.60 \text{-} 0.35i & 0.00 & 0.00 & \text{-}0.60 \text{-} 0.35i & 0.00 & 0.70i \\
0.00 & 0.00 & 0.60 \text{+} 0.35i & \text{-}0.70i & 0.00 & \text{-}0.60 \text{+} 0.35i & 0.00 \\
1.2 & 0.60 \text{+} 0.35i & 0.60 \text{-} 0.35i & 0.60 \text{-} 0.35i & 0.60 \text{+} 0.35i & 0.60 \text{-} 0.35i & 0.60 \text{+} 0.35i \\
\text{-}1.2 & \text{-}0.60 \text{+} 0.35i & \text{-}0.60 \text{-} 0.35i & \text{-}0.60 \text{-} 0.35i & \text{-}0.60 \text{+} 0.35i & \text{-}0.60 \text{-} 0.35i & \text{-}0.60 \text{+} 0.35i \\
0.00 & 0.00 & \text{-}0.60 \text{+} 0.35i & \text{-}0.70i & 0.00 & 0.60 \text{+} 0.35i & 0.00 \\
0.00 & \text{-}0.60 \text{-} 0.35i & 0.00 & 0.00 & 0.60 \text{-} 0.35i & 0.00 & 0.70i
\end{array}\right)
}\,,
\end{equation}
where the shown digits are certified to be correct.
One then computes the product $\operatorname{Mat}_{\tau_*} = \Pi^{-1}\operatorname{Mat}_{\tau^*}\Pi$ which results in the following integral matrix:
\begin{equation}
\operatorname{Mat}_{\tau_*} ={\footnotesize \left(\begin{array}{rrrrrrr}
3 & 1 & 0 & 1 & 1 & 2 & 1 \\
-1 & -1 & 0 & 0 & 0 & -1 & 0 \\
0 & 0 & 1 & 0 & 0 & 0 & 0 \\
-1 & 0 & 0 & -1 & 0 & -1 & 0 \\
-1 & 0 & 0 & 0 & -1 & -1 & 0 \\
-2 & -1 & 0 & -1 & -1 & -1 & -1 \\
-1 & 0 & 0 & 0 & 0 & -1 & -1
\end{array}\right)}\,.
\end{equation}
The precision is such that the entries are certified to be correct.
\end{example}

\section{Picard-Fuchs equations} \label{sec:Picard-Fuchs}

We now explain how to use numerical computations of period matrices to compute the action of monodromy along a loop $\ell$ on $H_m(X_b)$ exactly and certifiably.

We turn to the case of a one parameter family of generically smooth hypersurfaces $\{X_t\}_{t\in \mathbb P^1}$. 
That is, we consider families \eqref{eq:calX} with $n = 1$. 
Similarly to \secref{sec:deRham} define the \emph{relative\footnote{We note that here (and below) ``relative'' does not refer to a pair, but to the fact that we are working over the base $\bbP^1$. We will not be considering relative (co)homology of a pair in this text.}  primitive (de Rham) cohomology} $\mathcal{PH}^m(X_t/K(t))$ of our family by extending scalars:
  \begin{equation}\label{eq:primitive_homology_relative}
\mathcal{PH}^{m}(X_t) \, \simeq \, \frac{ \operatorname{Vect}_{K(t)} \left\{ \frac{a}{f_t^k} \st k\ge 0\text{ and } a\in K(t)[x]_{kd - m- 2} \right\} }{ \operatorname{Vect}_{K(t)} \left\{ \frac{\partial}{\partial x_i}\frac{b}{f_t^k} \st k\ge 0\text{ and } 0\le i\le m+1\text{ and } b\in K(t)[x]_{kd - m- 1} \right\} }\,,
\end{equation}
where $f_t = K(t)[x_0, \dots, x_{m+1}]$ is the defining equation of $X_t$.
This defines a sheaf on $\bbP^1$ which is equipped with the \emph{Gauss--Manin connection}
\begin{equation}
\nabla_t \, \colon \,  \mathcal{PH}^{m}(X_t) \, \to \,  \mathcal{PH}^{m}(X_t)\,,
\end{equation}
a connection derived from the differentiation in $K(t)$. This is not to be confused with our notation $\nabla$ or $\nabla_h$ for the discriminant locus.
The action of $\nabla_t$ on the cohomology is the one induced by differentiation of the rational function representing an $m$-form via the residue: 
\begin{equation}
	\nabla_t \res\left(\frac{a}{f_t^k} \Omega_{m+1}\right) = \res\left(\frac{\partial}{\partial t}\left(\frac{a}{f_t^k}\right) \Omega_{m+1}\right).
\end{equation}
In particular, Griffiths--Dwork reduction allows to compute a connection matrix from a basis of $\mathcal{PH}^{m}(X_t)$ given in terms of rational functions in $t$. See \cite{BostanEtAl_2013} for details on this computation.

Given an $m$-form $\omega_t \in \mathcal{PH}^{m}(X_t)$, we consider all its successive derivatives $\nabla_t\,\omega_t, \nabla_t^2\,\omega_t, \dots$. As $\mathcal{PH}^{m}(X_t)$ is finite-dimensional, there exists an integer $r\ge 0$ such that the classes $\omega_t, \nabla_t\,\omega_t, \dots, \nabla_t^r\,\omega_t$ are $K(t)$-linearly dependent in $\mathcal{PH}^{m}(X_t)$. The linear relation is a differential equation called the \emph{Picard--Fuchs equation} for $\omega_t$. It is written as $\mathcal P \,  \omega_t = 0$, with 
\begin{equation}\label{eq:picard_fuchs_equation}
\mathcal P \,  = \, a_r(t)\nabla_t^r + \cdots + a_1(t)\nabla_t + a_0(t) \, \, \, \in \, K[t]\langle\nabla_t \rangle.
\end{equation}
The equality $\mathcal P \,  \omega_t = 0$ holds in $\mathcal{PH}^{m}(X_t)$. It means that the form $\mathcal P \, \omega_t$ is an exact differential.

The Gauss--Manin connection behaves well with respect to integration.
More precisely, let $\gamma_t$ be a section over an open set of the integral relative homology $\mathcal H_m(X_t)$, viewed as a sheaf on $\bbP^1$.
We may define the relative period $\pi(t) = \int_{\gamma_t}\omega_t$.
We have
\begin{equation}
	\int_{\gamma_t} \nabla_t \,\omega_t \,  = \, \frac{\ud}{\ud t} \int_{\gamma_t}\omega_t\,.
\end{equation}
In particular, we find that $\mathcal P \,  \pi(t) = 0$, where here we abuse notation slightly by writing $\mathcal P$ for the operator of \eqref{eq:picard_fuchs_equation} with $\nabla_t$ replaced by $\ud/\ud t$.

\begin{example}
We consider the family defined by the equation $f_t = x_0^3+x_1^3+x_2^3+x_3^3 - tx_0x_1x_2$.
The basis of primitive cohomology of \exref{ex:fermat_cohomology} extends to a basis of the relative primitive cohomology by simply replacing the denominator $f$ by $f_t$:
\begin{equation}
\begin{gathered}
\omega_1(t) = \res \frac{x_0x_1}{f_t^2}\Omega_3\,, \quad 
\omega_2(t) = \res \frac{x_0x_2}{f_t^2}\Omega_3\,, \quad 
\omega_3(t) = \res \frac{x_0x_3}{f_t^2}\Omega_3\,, \\
\omega_4(t) = \res \frac{x_1x_2}{f_t^2}\Omega_3\,, \quad 
\omega_5(t) = \res \frac{x_1x_3}{f_t^2}\Omega_3\,, \quad 
\omega_6(t) = \res \frac{x_2x_3}{f_t^2}\Omega_3\,.
\end{gathered}
\end{equation}
The Picard--Fuchs equation of $\omega_1$ is given by
$\mathcal P =  (t^3 - 27) \nabla_t + 2t^2$.
Indeed, we have
\begin{equation*}
\begin{aligned}
\left((t^3-27)\frac{\partial}{\partial t} +2t^2\right)\frac{x_0x_1}{f_t^2} 
\, = &\,\,\,  \frac{1}{3} \frac{\partial}{\partial x_0}\left(\frac {x_1(27 x_1 x_2-t^2 x_0^2)}{f_t^2}\right)
+ 2t \frac{\partial}{\partial x_1}\left(\frac {x_1 x_2^2}{f_t^2}\right) \\
&+ \frac{t}{3} \frac{\partial}{\partial x_2}\left(\frac {3 x_1^3 + 3 x_3^3 + t x_0 x_1 x_2}{f_t^2} + \frac{1}{f_t}\right)
-t \frac{\partial}{\partial x_3}\left(\frac {x_2^2 x_3}{f_t^2}\right)
\,,
\end{aligned} 
\end{equation*}
which shows that the left hand side is zero in primitive cohomology, see \eqref{eq:primitive_homology_relative}. \qedhere
\end{example}

\subsection{Numerical integration methods and monodromy}\label{sec:numerical_integration}

We now explain how to compute the action of monodromy on $H_m(X_b)$ along a given path $\ell$, which is the last step to compute the monodromy representation \eqref{eq:monodromy_rep_intro}.
We choose a basis for the primitive cohomology of $X_b$ using Griffiths--Dwork reduction. This basis is of the form $\omega_i(t) = {a_i}{f_t^{-k_i}} \, \Omega_{m+1}, i = 1, \ldots, s$.
If $m$ is even we complete this list to a basis of $\mathcal{H}^m(X_t)$ by adding the dual $\omega_0$ of $H$, and we set $\omega_0=0$ when $m$ is odd.
The methods of \cite{LairezEtAl_2024} and \cite{Sertoz_2019} allow to compute a basis $\gamma_0 = H, \gamma_1, \dots, \gamma_s$ of $H_m(X_b)$, the corresponding intersection product and the period matrix of $X_t$ in the bases $\omega_0(b), \dots, \omega_s(b)$ and $\gamma_0, \dots, \gamma_s$.
The periods $\Pi_{ij}(b) = \int_{\gamma_j} \omega_i(b)$ are computed with certified bounds of precision.

For each $1\le i\le s$, we can compute the Picard--Fuchs equation $\mathcal P_i$ corresponding to $\omega_i$ and denote by $r$ its order.
Using the Gauss-Manin connection matrix, we are also able to compute numerical approximations of the periods $\int_{\gamma_j}\nabla_t^k\,\omega_i(b)$ for all $j$ and $0\le k \le r$, which are the values of the derivatives of $\Pi_{ij}(t) := \int_{\gamma_j(t)}\omega_i(t)$, where $\gamma_j(t)$ is the parallel transport of $\gamma_j$ along a path in a simply connected neighbourhood of $b$ in $\bbP^1\setminus\nabla$.

The relative period $\Pi_{ij}(t)$ is a solution of a differential equation $\mathcal P_i$ with specified initial conditions.
We use numerical analytic continuation methods for differentially-finite functions to compute the action of the monodromy along a loop $\ell\in \pi_1(\bbP^1\setminus\nabla)$ \cite{ChudnovskyChudnovsky_1990, Mezzarobba_2016,VanDerHoeven_1999}.
Doing this for each entry of $\Pi$, we analytically continue $\Pi$ along $\ell$ and find another period matrix $\tilde\Pi$ with entries $\tilde\Pi_{ij} = \int_{\ell_*\gamma_j}\ell^*\omega_i(b)$. Note that the periods for $\omega_0$ are invariant under monodromy.

The form $\ell^*\omega_i$ is the result of the action of monodromy along $\ell$ on $\omega_i$.
But we have chosen $\omega_i$ to be rational in $t$, and thus $\ell^*\omega_i(b) = \omega_i(b)$. Therefore, by linearity, we find that
\begin{equation}
\tilde \Pi_{ij} = \int_{\ell_*\gamma_j}\omega_i(b)\,,
\end{equation}
which in matrix form reads $
\tilde \Pi = \Pi \operatorname{Mat}_{\ell_*}$, and thus $\Pi^{-1} \tilde\Pi =\operatorname{Mat}_{\ell_*} \in \operatorname{GL}(H_m(X_b))$. 
Again, we may evaluate the left hand-side numerically, and it is sufficient to do so with a certified resulting precision of less than one half to recover $\operatorname{Mat}_{\ell_*}$ exactly.

In practice we may gain performance by computing the action of monodromy on the period matrix by doing numerical analytic continuation for a single cyclic form (with respect to differentiation) instead of for each entry of $\Pi(t)$. 
See \cite[Section 3.4]{LairezEtAl_2024} for details. 
The basis of homology, the corresponding period matrix and the coefficients of the hyperplane class $H$ in this basis are computed using the \texttt{lefschetz-family} package\footnote{\url{https://github.com/ericpipha/lefschetz-family}} \cite{LairezEtAl_2024}, and the numerical analytic continuation is performed using Marc Mezzarobba's implementation in \texttt{ore\_algebra} \cite{Mezzarobba_2016, KauersEtAl_2015}, both in SageMath \cite{sagemath}. The computation of the period matrix of the Clebsch surface took around 50 seconds, while the numerical analytic continuation part took 15 minutes in total for all the subgroups of $S_5$, ranging from 10 seconds for $D_6$ to four minutes for the generic case $\Z_1$. All these computations were done on a 2021 MacBook Pro with Apple M1 chips Pro running on 10 cores.

\section{Symmetric cubic surfaces} \label{sec:cubicsurfaces}

A cubic surface $X \subset \mathbb{P}^3$ is defined by the vanishing $V(f)$ of a quaternary cubic form~$f$: 
\begin{align} \label{eq:f}
\begin{split}
 f & \, = \, z_0 \, x_0^3 + z_1 \, x_0^2x_1 + z_2 \, x_0^2 x_2 + z_3 \, x_0^2 x_3 + z_4 \, x_0x_1^2 + z_5 \, x_0x_1x_2 + z_6 \, x_0x_1x_3 + z_7 \, x_0x_2^2 \\
 & \, \, \, \, \, \, \, \, + z_8 \, x_0x_2x_3 + z_9 \, x_0x_3^2 + z_{10} \, x_1^3 + z_{11} \, x_1^2x_2 + z_{12} \, x_1^2x_3 + z_{13} \, x_1x_2^2 + z_{14} \, x_1 x_2 x_3 \\
 & \, \, \, \, \, \, \, \, + z_{15} \, x_1x_3^2 + z_{16} \, x_2^3 + z_{17} \, x_2^2x_3 + z_{18} \, x_2 x_3^2  + z_{19} \, x_3^3 \, \, \, \in \mathbb{C}[x_0,x_1,x_2,x_3].
 \end{split}
 \end{align}
In our set-up from previous sections, $X_z = V(f) \subset \mathbb{P}^3$ is a member of the complete linear system ${\cal L}_h = \mathbb{P}H^0(\mathbb{P}^3, {\cal O}(3)) \simeq \mathbb{P}^{19}$, where $h = (x_0^3, \ldots, x_3^3)$ contains the 20 cubic monomials. 
Cayley \cite{cayley1849triple} and Salmon \cite{salmon1849triple} showed that every smooth cubic surface $X_z$, for $z \in \mathbb{P}^{19} \setminus \nabla_h$, contains 27 complex lines. 
This is a standard result in classical algebraic geometry. 
For a detailed discussion, the reader can consult Dolgachev's book \cite[Chapter 9]{dolgachev2012classical}. 
A standard illustration of this fact shows the 27 lines on the Clebsch surface, see Figure \ref{fig:clebsch}.

\begin{figure}[h]
\includegraphics[width=6cm]{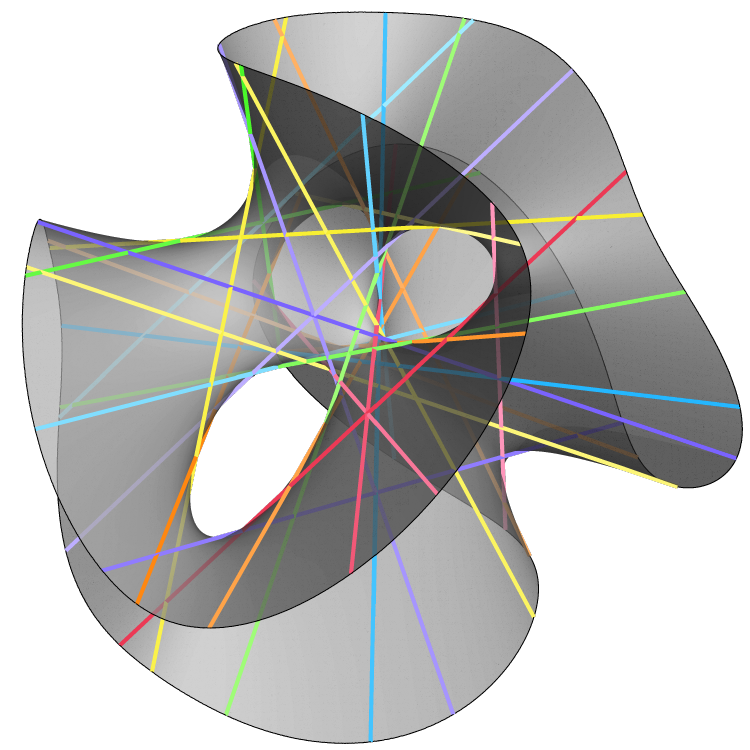}
\centering
\caption{The Clebsch surface and its twenty-seven (real) lines.}
\label{fig:clebsch}
\end{figure}

The 27 lines are relevant to our story because each line $L$ is isolated in its homology class, in the sense that no two lines belong to the same class. Moreover, the lines generate the homology. 
Therefore, the monodromy group can be represented as a subgroup of the symmetric group $S_{27}$; it acts by permuting the lines, see Section \ref{subsec:homology_cubic_surface}. 
That group is the \emph{Galois group} of the splitting field $K/k({\cal L}_h)$ of the lines~\cite{harris1979galois,Sottile_2025}, i.e., the minimal extension of the function field of ${\cal L}_h \simeq \mathbb{P}^{19}$ such that the 27 lines are defined over $K$.

It is well known that the Galois group of \eqref{eq:f} is the Weyl group $W(E_6)$ of order 51\,840 \cite{Jordan_1870}. 
This is a small subgroup of the full symmetric group $S_{27}$. It happens to be the largest subgroup which respects incidences and symmetries between the 27 lines.

As discussed in the Introduction, less is known about the Galois groups of $G$-invariant cubic surfaces. 
Our goal in this section is to explain \thmref{thm:main_new}. 
This is an application of Algorithm \ref{alg:main} to $G$-invariant linear systems of cubic surfaces. 
The steps of the algorithm were presented in detail in previous sections. 
Here, we elaborate on the specific features of cubic surfaces. 
We start by recalling general facts about the homology $H_2(X)$ of a cubic surface in Section \ref{subsec:homology_cubic_surface}. 
Next, in Section \ref{sec:results}, we give details about the group actions in \thmref{thm:main_new}.

\subsection{Homology of cubic surfaces}\label{subsec:homology_cubic_surface}
We recall some well-known facts about the homology of a smooth cubic surface $X \subset \mathbb{P}^3$. 
The middle homology $H_2(X)$ of $X$ is an odd, unimodular integral lattice of rank 7 with signature $(1,6)$. A basis $b_0, \ldots, b_6$ for $H_2(X)$ can be described from the description of $X$ as the blow-up of the complex projective plane $\mathbb{P}^2$ at six points. We briefly recall this construction. 

Let $e_1, \ldots, e_6$ be six points in $\mathbb{P}^2$, such that no three lie on a line, and no conic passes through all of them. The blow-up of $\mathbb{P}^2$ at these points is the image of the map $\phi: \mathbb{P}^2 \dashrightarrow \mathbb{P}^3$ whose coordinates form a basis for the linear system of ternary cubics vanishing at $e_1, \ldots, e_6$. That is, $\phi$ is given by $\phi(p) = (g_0(p) : g_1(p): g_2(p): g_3(p))$, where $g_0, \ldots, g_3$ are linearly independent cubic polynomials in three variables which vanish at our six points $e_1, \ldots, e_6$. The polynomials $g_i$ satisfy one cubic relation $f(g_0, g_1,g_2,g_3) = 0$ which defines the image of $\phi$ --- a cubic surface $X$ in $\mathbb{P}^3$. The blow-down morphism $\pi: X \rightarrow \mathbb{P}^2$ is the inverse of $\phi$. The exceptional divisors are $E_i = \pi^{-1}(e_i)$, which form six of the lines on $X$. The remaining 21 lines are the strict transforms of the fifteen lines connecting $e_i$ and $e_j$, and of the six conics passing through all but one of the $e_i$. The map $\phi: \mathbb{P}^2 \dashrightarrow X$ is called a \emph{plane model} for $X$.

A plane model for $X$ always exists. It comes with a labelling of the 27 lines, and with a basis $b_0, \ldots, b_6$ for $H_2(X)$. The homology class $b_0 \in H_2(X)$ is the pullback of the hyperplane class in $\mathbb{P}^2$ along the blow-down morphism $\pi: X \rightarrow \mathbb{P}^2$, and $b_1, \ldots, b_6$ are the homology classes of the six exceptional divisors $E_1, \ldots, E_6$.

A line $L$ is isolated in its homology class, meaning that any two distinct lines on $X$~have a different homology class. We shall write $L$ for the class of a line on $X$, and $H$ for the class of a curve obtained as the intersection of $X$ with a generic plane in $\mathbb{P}^3$. Classes of lines are characterised by their self-intersection and their intersection with $H$; we have $L^2 = -1$ and $L\cdot H = 1$.
There are 27 such classes. Their coordinates in the basis $b_0, \ldots, b_6$ from above~are
\setcounter{MaxMatrixCols}{30}
\setlength{\arraycolsep}{3.5pt}
\renewcommand{\kbldelim}{(}
\renewcommand{\kbrdelim}{)}
\[ \footnotesize
\hspace{-0.1cm}\kbordermatrix{
    &  1 &  2 &  3 &  4 &  5 &  6 &  7 &  8 &  9 &  {10} &  {11} &  {12} &  {13} &  {14} &  {15} &  {16} &  {17} &  {18} &  {19} &  {20} &  {21} &  {22} &  {23} &  {24} &  {25}&  {26} &  {27} \\
b_0& 0 & 0 & 0 & 0 & 0 & 0 & 2 & 2 & 2 & 2 & 2 & 2 & 1 & 1 & 1 & 1 & 1 & 1 & 1 & 1 & 1 & 1 & 1 & 1 & 1 & 1 & 1 \\
b_1&1 & 0 & 0 & 0 & 0 & 0 & 0 & \shortminus 1 & \shortminus1 & \shortminus1 & \shortminus1 & \shortminus1 & \shortminus1 & \shortminus1 & \shortminus1 & \shortminus1 & 0 & 0 & 0 & 0 & 0 & 0 & 0 & 0 & \shortminus1 & 0 & 0\\
b_2 & 0 & 1 & 0 & 0 & 0 & 0 & \shortminus1 & 0 & \shortminus1 & \shortminus1 & \shortminus1 & \shortminus1 & \shortminus1 & 0 & 0 & 0 & \shortminus1 & \shortminus1 & \shortminus1 & 0 & 0 & 0 & 0 & 0 & 0 & \shortminus1 & 0 \\
b_3& 0 & 0 & 1 & 0 & 0 & 0 & \shortminus1 & \shortminus1 & 0 & \shortminus1 & \shortminus1 & \shortminus1 & 0 & \shortminus1 & 0 & 0 & \shortminus1 & 0 & 0& \shortminus1 & \shortminus1 & 0 & 0 & 0 & 0 & 0 & \shortminus1 \\
b_4& 0 & 0 & 0 & 1 & 0 & 0 & \shortminus1 & \shortminus 1& \shortminus1 & 0 & \shortminus1 & \shortminus1 & 0 & 0 & 0 & 0 & 0 & \shortminus1 & 0 & \shortminus1 & 0 & \shortminus1 & \shortminus1 & 0 & \shortminus1 & 0 & 0 \\
b_5& 0 & 0 & 0 & 0 & 1 & 0 & \shortminus1 & \shortminus1 & \shortminus1 & \shortminus1 & 0 & \shortminus1 & 0 & 0 & \shortminus1 & 0 & 0 & 0 & 0 & 0 & \shortminus1 & 0 & \shortminus1 & \shortminus1 & 0 & \shortminus1 & 0 \\
b_6& 0 & 0 & 0 & 0 & 0 & 1 & \shortminus1 & \shortminus1 & \shortminus1 & \shortminus1 & \shortminus 1 & 0 & 0 & 0 & 0 & \shortminus1 & 0 & 0 & \shortminus1 & 0 & 0  & \shortminus1 & 0 & \shortminus1 & 0 & 0 & \shortminus1
  }.
\]
The hyperplane class has coefficients $H = \left(3,\,-1,\,-1,\,-1,\,-1,\,-1,\,-1\right)$ in our basis. In the above ordering, the first six lines are the exceptional divisors.  The other columns are obtained as follows. The strict transform of the conic passing through all points except $e_i$ has class $2b_0 - \sum_{j \neq i} b_{j}$. The strict transform of the line through $e_i$ and $e_j$ has class $b_0 - b_{i} - b_{j}$.

Consider a family of generically smooth cubic surfaces $X_z\subset\bbP^3$ with $z$ a parameter in some projective space $\bbP^n$ and assume that $X = X_b$ for some base point $b\in \bbP^n$. 
We denote by $\nabla \subset \mathbb{P}^n$ the discriminant hypersurface of this family: $\nabla = \{ z \in \mathbb{P}^n \, : \, X_z \text{ is singular} \}$. 

The action of monodromy preserves the intersection product, as well as the hyperplane class $H$.
In particular, the monodromy acts on the lines by permuting them. Since the lines generate $H_2(X_b)$, this permutation characterises the monodromy action on $H_2(X_b)$ entirely.
More concretely, the action of monodromy on $H_2(X_b)$ along a loop $\ell\in \pi_1(\bbP^n\setminus \nabla, b)$ is given by a $7\times7$ matrix $M_\ell \in {\rm Aut}(H_2(X))$, and there exists a unique $\sigma_\ell\in S_{27}$ such that for all $1\le i\le 27$, we have that $M_\ell L_i = L_{\sigma_\ell(i)}$.
If $\ell_1, \dots, \ell_r$ generate $\pi_1(\bbP^n\setminus\nabla, b)$, then $\sigma_{\ell_1}, \ldots, \sigma_{\ell_r}$ generate the Galois group of our family. In the notation from the Introduction, the family $X_z$ is given by a linear system ${\cal L}_h$, and we have $\Gal(\mathcal L_{h}) = \langle \sigma_{\ell_1}, \dots, \sigma_{\ell_r} \rangle$.

\subsection{Subgroups of $S_5$ acting on cubic surfaces} \label{sec:results}

This section gives more details regarding the groups appearing in \thmref{thm:main_new} and Table \ref{tab:main}.
We consider cubic surfaces whose defining equation is invariant under the linear automorphisms induced by conjugacy classes of subgroups of $S_5$. 
There are 19 such conjugacy classes. They form a poset with respect to inclusion as depicted in the left part of  \figref{fig:tower_inclusion}.

\begin{figure}[h]
\hspace{-0.3cm}\makebox[\textwidth][c]{
\begin{subfigure}{.45\textwidth}
    \centering
    \begin{tikzpicture}[scale=0.58, transform shape]
\begin{scope}
   \path [fill=gray!25] plot [smooth cycle, tension=.5] coordinates {(2,-2.5)  (0,-4.5) (1,-6) (-.5, -8) (2,-10.5) (9,-6)};
\end{scope}

\begin{scope}[every node/.style={rectangle,thick}]
    \node (S5) at (2,-1) {$S_5$};
    \node (A5) at (-2,-3) {$A_5$};
    \node (D6) at (0,-3) {$D_6$};
    \node (S4) at (2,-3) {$S_4$};
    \node (F5) at (6,-3) {$F_5$};
    \node (A4) at (.5,-4.5) {$A_4$} ;
    \node (D4) at (6,-4.5) {$D_4$} ;
    \node (Z6) at (-4,-6) {$\Z_6$} ;
    \node (tS3) at (-2,-6) {$\tS$} ;
    \node (S3) at (2,-6) {$S_3$} ;
    \node (D5) at (0,-6) {$D_5$} ;
    \node (K4) at (4,-6) {$K_4$} ;
    \node (nK4) at (6,-6) {$\nK$} ;
    \node (Z4) at (8,-6) {$\Z_4$} ;
    \node (Z5) at (-2,-8) {$\Z_5$} ;
    \node (Z3) at (.5,-8) {$\Z_3$} ;
    \node (Z2) at (3,-8) {$\Z_2$} ;
    \node (DT) at (6,-8) {$\DT$} ;
    \node (Z1) at (2,-10) {$\Z_1$} ;
\end{scope}

\begin{scope}[>={Stealth[black]},
              every node/.style={fill=white,circle},
              every edge/.style={draw=black}]
\path [->] (A5) edge (S5);
\path [->] (S4) edge (S5);
\path [->] (F5) edge (S5);
\path [->] (D6) edge (S5);
\path [->] (A4) edge (A5);
\path [->] (tS3) edge (A5);
\path [->] (D5) edge (A5);
\path [->] (S3) edge (S4);
\path [->] (A4) edge (S4);
\path [->] (D4) edge (S4);
\path [->] (D5) edge (F5);
\path [->] (Z4) edge (F5);
\path [->] (Z6) edge (D6);
\path [->] (tS3) edge (D6);
\path [->] (S3) edge (D6);
\path [->] (nK4) edge (D6);
\path [->] (Z3) edge (A4);
\path [->] (K4) edge (A4);
\path [->] (K4) edge (D4);
\path [->] (nK4) edge (D4);
\path [->] (Z4) edge (D4);
\path [->] (Z2) edge (Z6);
\path [->] (Z3) edge (Z6);
\path [->] (Z3) edge (tS3);
\path [->] (DT) edge (tS3);
\path [->] (Z3) edge (S3);
\path [->] (Z2) edge (S3);
\path [->] (Z5) edge (D5);
\path [->] (DT) edge (D5);
\path [->] (Z2) edge (nK4);
\path [->] (DT) edge (nK4);
\path [->] (DT) edge (K4);
\path [->] (DT) edge (Z4);
\path [->] (Z1) edge (Z5);
\path [->] (Z1) edge (Z3);
\path [->] (Z1) edge (Z2);
\path [->] (Z1) edge (DT);

\end{scope}
\end{tikzpicture}

\end{subfigure}
\quad \, 
    \begin{subfigure}{.45\textwidth}
    \centering
    \begin{tikzpicture}[scale=0.58, transform shape]

\begin{scope}
   \path [fill=gray!25] plot [smooth cycle, tension=.5] coordinates {(2,-2.5)  (0,-4.5) (1,-6) (-.5, -8) (2,-10.5) (9,-6)};
\end{scope}

\begin{scope}[every node/.style={rectangle,thick}]
    \node (S5) at (2,-1) {$\Z_1$};
    \node (A5) at (-2,-3) {$\Z_1$};
    \node (F5) at (6,-3) {$\Z_1$};
    \node (S4) at (2,-3) {$\Z_2^2$};
    \node (A4) at (.5,-4.5) {$\Z_2^2$} ;
    \node (D6) at (0,-3) {$D_6$};
    \node (D5) at (0,-6) {$\Z_1$} ;
    \node (D4) at (6,-4.5) {$\Z_2^3$} ;
    \node (S3) at (2,-6) {$S_3^2$} ;
    \node (tS3) at (-2,-6) {$D_6$} ;
    \node (Z6) at (-4,-6) {$\Z_6\times S_3$} ;
    \node (Z5) at (-2,-8) {$\Z_5$} ;
    \node (K4) at (4,-6) {$\Z_2^4$} ;
    \node (nK4) at (6,-6) {$\Z_2^2\times S_4$} ;
    \node (Z4) at (8,-6) {$\Z_2^2\times \Z_4$} ;
    \node (Z3) at (.5,-8) {$\Z_3\times S_3^2$} ;
    \node (Z2) at (3,-8) {$W(F_4)$} ;
    \node (DT) at (6,-8) {$D_4\times S_4$} ;
    \node (Z1) at (2,-10) {$W(E_6)$} ;
\end{scope}

\begin{scope}[>={Stealth[black]},
              every node/.style={fill=white,circle},
              every edge/.style={draw=black}]
    \path [->] (S5) edge (A5);
    \path [->] (S5) edge (S4);
    \path [->] (S5) edge (F5);
    \path [->] (S5) edge (D6);
    \path [->] (A5) edge (A4);
    \path [->] (A5) edge (tS3);
    \path [->] (A5) edge (D5);
    \path [->] (S4) edge (S3);
    \path [->] (S4) edge (A4);
    \path [->] (S4) edge (D4);
    \path [->] (F5) edge (D5);
    \path [->] (F5) edge (Z4);
    \path [->] (D6) edge (Z6);
    \path [->] (D6) edge (tS3);
    \path [->] (D6) edge (S3);
    \path [->] (D6) edge (nK4);
    \path [->] (A4) edge (Z3);
    \path [->] (A4) edge (K4);
    \path [->] (D4) edge (K4);
    \path [->] (D4) edge (nK4);
    \path [->] (D4) edge (Z4);
    \path [->] (Z6) edge (Z2);
    \path [->] (Z6) edge (Z3);
    \path [->] (tS3) edge (Z3);
    \path [->] (tS3) edge (DT);
    \path [->] (S3) edge (Z3);
    \path [->] (S3) edge (Z2);
    \path [->] (D5) edge (Z5);
    \path [->] (D5) edge (DT);
    \path [->] (nK4) edge (Z2);
    \path [->] (nK4) edge (DT);
    \path [->] (K4) edge (DT);
    \path [->] (Z4) edge (DT);
    \path [->] (Z5) edge (Z1);
    \path [->] (Z3) edge (Z1);
    \path [->] (Z2) edge (Z1);
    \path [->] (DT) edge (Z1);
\end{scope}
\end{tikzpicture}
\end{subfigure}}

\caption{\textit{(left)} The poset of inclusions satisfied by the conjugacy classes of the 19 subgroups of $S_5$. Subgroups of $S_4$ are shaded in grey. \textit{(right)} The corresponding reversed poset of inclusions satisfied by the Galois groups of $\mathcal L_G$ for the subgroups $G$ of $S_5$.}
\label{fig:tower_inclusion}
\end{figure}
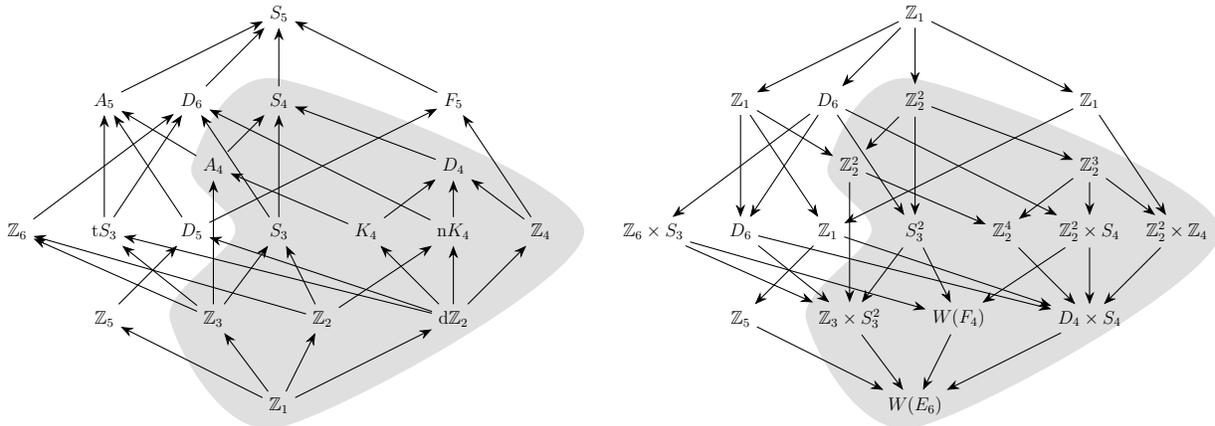

Concrete realisations of these groups are given in \tabref{tab:subgroups_S5}. 
Note that some groups, like $\DT$ and $\Z_2$, are isomorphic as abstract groups, but not conjugate in $S_5$.
When discussing (conjugacy classes of) subgroups in $S_5$, we will use the notation of \tabref{tab:subgroups_S5} to designate these groups. For example, $\Z_2$ will never be used to mean $\DT$.
Finally, note that two subgroups of $S_4$ are conjugates of each other in $S_4$ if and only if they are conjugates in $S_5$.
\begin{table}[h]
\footnotesize
\centering
\begin{tabular}{cccccc}
\toprule
\multirow{2}{*}{Notation} & \multirow{2}{*}{$G$} & \multirow{2}{*}{$|G|$} & \multicolumn{2}{c}{Multiplicity } & \multirow{2}{*}{Notes}\\ 
 &  &   & in $S_4$&in $S_5$ & \\ 
\midrule
$S_5$ & $\langle (14532),(12)\rangle$   &                    120      &   -  &     1           \\
$A_5$ & $\langle (14532),(123)\rangle$   &                    60     &   -  &    1           \\
$F_5$ & $\langle (12345),(2354)\rangle$   &                    20      &     -  &   6        & Frobenius group   \\
$D_6$ & $\langle (123), (12),(45)\rangle$   &               12      &     -  &   10           \\
$D_5$ & $\langle (14532),(12)(34)\rangle$   &              10       &   -  &     6           \\
$\tS$ & $\langle (125),(12)(34)\rangle$   &                  6      &     -  &   10      & twisted $S_3$     \\
$\Z_6$ & $\langle (123),(45)\rangle$   &                         6      &     -  &   10           \\
$\Z_5$ & $\langle (14532)\rangle$   &                             5     &     -  &       6          \\\midrule
$S_4$ & $\langle (1234),(12)\rangle$   &                    24      &    1  &    5           \\
$A_4$ & $\langle (123),(12)(34)\rangle$   &                12      &   1   &    5           \\
$D_4$ & $\langle (1324),(12)\rangle$   &                    8      &    3    &  15           \\
$S_3$ & $\langle (123),(12)\rangle$   &                       6      &    4  &    10           \\
$K_4$ & $\langle (12)(34),(13)(24)\rangle$   &           4      &   1   &    5          & normal Klein four-group \\
$\nK$ & $\langle (12),(34)\rangle$   &                      4      &   3   &    15       & non-normal Klein four-group    \\
$\Z_4$ & $\langle (1324) \rangle$   &                         4      &    3  &    15           \\
$\Z_3$ & $\langle (123)\rangle$   &                           3      &    4  &    10           \\
$\Z_2$ & $\langle (12)\rangle$   &                           2      &    6   &   10          \\
$\DT$ & $\langle (12)(34)\rangle$   &                       2      &   3  &     15     & double transposition      \\
$\Z_1$ & $\langle\operatorname{id}\rangle$   &      1      &   1   &   1           \\\bottomrule
\end{tabular}
\caption{List of all the nineteen subgroups $G$ of $S_5$ up to conjugation. The multiplicity is the number of conjugate subgroups. The permutations act on the coordinates of the cubic surface in pentahedral normal form, see \eqref{eq:pentahedral}.}
\label{tab:subgroups_S5}
\end{table}

\begin{remark}
The realisations of the conjugacy classes of subgroups of $S_4$ given in \tabref{tab:subgroups_S5} respect the inclusion relations shown in Figure \ref{fig:tower_inclusion} (left). In other words, for each pair of groups $G_1, G_2< S_4$ of \tabref{tab:subgroups_S5}, we have that $G_1<G_2$ if (and only if) a conjugate (as a subgroup of $S_4$) of $G_1$ lies in $G_2$. 
It is not possible to find such realisations for all subgroups of $S_5$. Indeed, suppose we have realisations 
of all conjugacy classes of subgroups of $S_5$ satisfying the above properties, then $\Z_6$ contains $\Z_2$ and $\Z_3$ and has the same order as the subgroup generated by $\Z_2$ and $\Z_3$ and therefore it is equal to this subgroup. Similarly, $S_3$ is also equal to this group, which is a contradiction. 
\end{remark}

Recall from the Introduction that we write ${\cal L}_G$ for the linear system of $G$-invariant cubics. The monodromy groups for ${\cal L}_G$, with $G$ any of the subgroups above, are given in \tabref{tab:main}. 
These groups are not independent of each other. 
If $G$ is a group acting on $\mathbb{C}[x_0,x_1,x_2,x_3]_3$ and $G'<G$, then $\mathcal L_{G}\subset \mathcal L_{G'}$ and thus the Galois group $\Gal(\mathcal L_{G})$ can be seen as a subgroup of $\Gal(\mathcal L_{G'})$.
In other words, the maps $G \mapsto \mathcal{L}_G$ and $G \mapsto {\rm Gal}(\mathcal{L}_G)$ are inclusion reversing, see \figref{fig:tower_inclusion} (right).
The inclusion ${\cal L}_{G} \subset {\cal L}_G$ implies $\mathcal L_{D_6} = \mathcal L_{\tS}$, $\mathcal L_{S_4} = \mathcal L_{A_4}$ and $\mathcal L_{S_5} = \mathcal L_{A_5} =\mathcal L_{F_5}$ as these are linear of the same dimension. \emph{A fortiori}, their monodromy groups are the same.

Up to scale, there is a single invariant cubic polynomial for the action of $S_5$, namely:
\begin{equation}\label{eq:clebsch}
f_0 \, = \,  x_0^3 + x_1^3 + x_2^3 + x_3^3 - (x_0 + x_1 + x_2 + x_3)^3\,.
\end{equation}
We can therefore use the corresponding (smooth) cubic surface $X_0 := V(f_0)$ as a base fibre for all our monodromy computations for subgroups of $S_5$. We express all the monodromy groups as subgroups of $W(E_6)\subset S_{27}$ acting on the 27 lines of $X_0$. The surface $X_0$ is called the \emph{Clebsch surface} \cite[\S9.5.4]{dolgachev2012classical}. Among its distinguishing features are the fact that all 27 lines are defined over the real numbers, and $X_0$ has the maximal number of 10 \emph{Eckardt points}, i.e., points on $X$ where three of the 27 lines meet. The Clebsch surface is depicted in \figref{fig:clebsch}. 

Since any subgroup $G$ of the Sylvester
action acts linearly on $\mathbb{C}[x_0,x_1,x_2,x_3]$, it also induces an action on the 27 lines of a $G$-invariant cubic surface $X$. Upon choosing an ordering of the lines, this induces a map 
\begin{equation}
\varphi\colon G\to W(E_6)\subset S_{27}\,.
\end{equation}
In \secref{appendix:permutations_symmetric}, for each group $G$ in \tabref{tab:main}, we list the actions of $\varphi(G)$ and $\Gal(\mathcal L_G)$ as subgroups of $S_{27}$.
In particular for all subgroups of $S_4$, since the representatives of \tabref{tab:subgroups_S5} respect the inclusions in the left part of \figref{fig:tower_inclusion}, the listed subgroups of $S_{27}$ are directly comparable and satisfy the tower of inclusions described in \figref{fig:tower_inclusion} (right). 
For the others, a conjugation by the action of $S_5$ might be needed. This is also detailed in \secref{appendix:permutations_symmetric}.

\section{Further examples}\label{sec:furtherexamples}

In previous sections, we have explained the steps of Algorithm \ref{alg:main} and we have explained how \thmref{thm:main_new} is obtained from it.
Indeed, since all computations are certified, the results available at \cite{code} constitute a proof of \thmref{thm:main_new}. This is our first and main example of a  family of monodromy groups that can successfully be computed using Algorithm \ref{alg:main}. 
This section offers more computational results obtained via Algorithm \ref{alg:main} in different~contexts.

\subsection{Crystallographic cubic surfaces} \label{sec:crystal_new}

Crystallographic groups of dimension $d$ are finite subgroups of $\operatorname{GL}_d(\Q)$.
In particular, all subgroups of $S_5$ are examples with $d = 4$.
Crystallographic groups are classified for $d\le 4$ in the tables of \cite{BrownEtAl_1978} and are contained in the database \texttt{CrystCat}, which is part of the computer algebra system \texttt{gap}.
These groups have realisations as subgroups of $\operatorname{GL}_d(\Z)$.

For $d=4$, there are 227 such groups up to conjugation in $\operatorname{GL}_4(\Q)$. 
They are indexed by two integers $s$ and $q$ following the indexing of \texttt{CrystCat}. 
The indices correspond respectively to the \emph{crystal system} and to the \emph{$\Q$-class}.
There are 33 crystal systems, and each has several $\Q$-classes.
We will denote these groups by $C_{N}$ where $N=s.q$, e.g. $C_{21.3}$ is the crystallographic group corresponding to crystal system $21$ and $\Q$-class $3$.
The corresponding groups act on $\bbP^3$ by linear automorphisms. In particular, they act faithfully on the lines of smooth cubic surfaces defined by polynomials which are invariant under the action.
However, 146 of these groups have no nonzero invariants.
In 50 cases, there exist invariant polynomials, but they all define singular cubic surfaces.
In 6 cases among the remaining 31, $\mathcal L_{C_N}$ is zero-dimensional. 
\begin{theorem}
Following the notation of \tabref{tab:main},  \tabref{tab:crystallographic} reports these 31 groups, the corresponding monodromy group, the intersection between the action of the group and the monodromy group, i.e., $I_{G} = \Gal(\mathcal L_{G})\cap G$, and the degree of the discriminant locus.
\end{theorem}

\begin{table}[p]
\footnotesize
\centering
\makebox[\linewidth][c]{
\begin{tabular}{cccccccc}
\toprule
\multirow{2}{*}{$N$} & \multirow{2}{*}{$C_N$} &\multirow{2}{*}{\hspace{-.5em} $\dim {\cal L}_{C_N}$\hspace{-.5em}} & \multirow{2}{*}{$\Gal(\mathcal L_{C_N})$}& \multirow{2}{*}{$I_{C_N}$} &\multicolumn{2}{c}{Orbit structure} & \multirow{2}{*}{$\deg \nabla_G$} \\ 
 &  & & & & $G$ & $\Gal(\mathcal L_{C_n})$ &  \\ \midrule
\textbf{31.4} & $S_5$ & 0 & $-$ & $-$ & $12^{1},15^{1}$ & $1^{27}$ & $-$ \\[.2em]
\textbf{31.3} & $A_5$ & 0 & $-$ & $-$ & $6^{2},15^{1}$ & $1^{27}$ & $-$ \\[.2em]
\textbf{31.1} & $F_5$ & 0 & $-$ & $-$ & $2^{1},5^{1},10^{2}$ & $1^{27}$ & $-$ \\[.2em]
29.7 & $\operatorname{SO}^+_4(\mathbb F_2)$ & 0 & $-$ & $-$ & $6^{1},9^{1},12^{1}$ & $1^{27}$ & $-$ \\[.2em]
29.4 & $\Z_3^2 \rtimes\Z_4$ & 0 & $-$ & $-$ & $6^{3},9^{1}$ & $1^{27}$ & $-$ \\[.2em]
29.3 & $S_3^2$ & 0 & $-$ & $-$ & $3^{2},6^{2},9^{1}$ & $1^{27}$ & $-$ \\[.2em]
29.1 & $\Z_3\times S_3$ & 1 & $\Z_3$ & $\Z_3$ & $3^{2},6^{2},9^{1}$ & $1^{9},3^{6}$ & $2^1$ \\[.2em]
\textbf{27.3} & $D_5$ & 1 & $\Z_1$ & $\Z_1$ & $1^{2},5^{5}$ & $1^{27}$ & $2^1$ \\[.2em]
\textbf{27.1} & $\Z_5$ & 3 & $\Z_5$ & $\Z_5$ & $1^{2},5^{5}$ & $1^{2},5^{5}$ & $4^1$ \\[.2em]
\textbf{24.3} & $S_4$ & 2 & $\Z_2^2$ & $\Z_1$ & $3^{1},12^{2}$ & $1^{3},2^{6},4^{3}$ & $1^3,3^1$ \\[.2em]
\textbf{24.1} & $A_4$ & 2 & $\Z_2^2$ & $\Z_1$ & $3^{1},6^{2},12^{1}$ & $1^{3},2^{6},4^{3}$ & $1^3,3^1$ \\[.2em]
22.8 & $S_3^2$ & 1 & $\Z_3$ & $\Z_1$ & $6^{3},9^{1}$ & $1^{9},3^{6}$ & $1^2$\\[.2em]
22.5 & $\Z_3 \rtimes S_3$ & 1 & $\Z_3$ & $\Z_1$ & $3^{6},9^{1}$ & $1^{9},3^{6}$ & $1^2$ \\[.2em]
22.3 & $\Z_3\times S_3$ & 2 & $\Z_3^2$ & $\Z_3$ & $6^{3},9^{1}$ & $3^{3},9^{2}$ & $1^1,2^1$ \\[.2em]
22.1 & $\Z_3^2$ & 3 & $\Z_3^3$ & $\Z_3^2$ & $3^{6},9^{1}$ & $9^{3}$ & $2^2$ \\[.2em]
21.3 & $D_6$ & 1 & $\Z_2$ & $\Z_2$ & $1^{1},2^{2},3^{2},4^{1},6^{2}$ & $1^{7},2^{10}$ & $1^2$ \\[.2em]
21.1 & $\Z_6$ & 3 & $\Z_2\times \Z_6$ & $\Z_6$ & $1^{1},2^{4},3^{2},6^{2}$ & $1^{1},2^{2},4^{1},6^{1},12^{1}$ & $2^2$ \\[.2em]
17.1 & $S_3$ & 3 & $S_3$ & $\Z_1$ & $1^{3},2^{3},3^{4},6^{1}$ & $1^{3},2^{3},3^{4},6^{1}$ & $4^1$ \\[.2em]
\textbf{14.8} & $D_6$ & 3 & $D_6$ & $\Z_2$ & $3^{1},6^{4}$ & $1^{3},2^{3},6^{3}$ & $1^3,3^1$ \\[.2em]
\textbf{14.3} & $\tS$ & 3 & $D_6$ & $\Z_1$ & $3^{7},6^{1}$ & $1^{3},2^{3},6^{3}$ & $1^3,3^1$ \\[.2em]
\textbf{14.2} & $\Z_6$ & 4 & $\Z_6\times S_3$ & $\Z_6$ & $3^{1},6^{4}$ & $3^{1},6^{1},18^{1}$ & $1^2,2^1,3^1$ \\[.2em]
\textbf{12.3} & $D_4$ & 3 & $\Z_2^3$ & $\Z_2$ & $1^{1},2^{1},4^{4},8^{1}$ & $1^{3},2^{2},4^{5}$ & $1^4,3^1$  \\[.2em]
\textbf{12.1} & $\Z_4$ & 4 & $\Z_2^2\times \Z_4$ & $\Z_4$ & $1^{1},2^{3},4^{5}$ & $1^{1},2^{1},4^{2},8^{2}$ & $1^3,2^1,3^1$ \\[.2em]
11.1 & $\Z_3$ & 7 & $\Z_3\times S_3^2$ & $\Z_3$ & $1^{9},3^{6}$ & $9^{1},18^{1}$ & $8^1$ \\[.2em]
\textbf{8.3} & $S_3$ & 6 & $S_3^2$ & $\Z_1$ & $3^{3},6^{3}$ & $3^{3},18^{1}$ & $1^1,4^1,8^1$ \\[.2em]
\textbf{8.1} & $\Z_3$ & 7 & $\Z_3\times S_3^2$ & $\Z_3$ & $3^{9}$ & $9^{1},18^{1}$ & $2^1,4^1,8^1$ \\[.2em]
\textbf{5.1} & $K_4$ & 4 & $\Z_2^4$ & $\Z_2^2$ & $1^{3},2^{6},4^{3}$ & $1^{3},4^{6}$ & $1^5,3^1$ \\[.2em]
\textbf{4.1} & $\nK$ & 7 & $\Z_2^2\times S_4$ & $\Z_2^2$ & $1^{1},2^{5},4^{4}$ & $1^{1},2^{2},6^{1},16^{1}$ & $2^1,4^3$ \\[.2em]
\textbf{3.1} & $\DT$ & 9 & $D_4\times S_4$ & $\Z_2$ & $1^{7},2^{10}$ & $1^{1},4^{1},6^{1},16^{1}$ & $4^2,8^1$ \\[.2em]
\textbf{2.1} & $\Z_2$ & 12 & $W(F_4)$ & $\Z_2$ & $1^{3},2^{12}$ & $3^{1},24^{1}$ & $10^1,12^1$ \\[.2em]
\textbf{1.1} & $\Z_1$ & 19 & $W(E_6)$ & $\Z_1$ & $1^{27}$& ${27}^{1}$& $32^1$ \\[.2em]
  \bottomrule
\end{tabular}}
\caption{The crystallographic groups $C_N$ such that $\mathcal L_{C_N}$ is non-empty and its generic element is smooth, along with the corresponding monodromy groups, and intersections between the permutation group induced by the action of the group (see Section \ref{sec:group_action}) and the monodromy group: $I_{G} = \Gal(\mathcal L_{G})\cap G$. The groups $\Z_3 \rtimes S_3$, $\Z_3^2 \rtimes\Z_4$ and $\operatorname{SO}^+_4(\mathbb F_2)$ are respectively a semidirect product of $\Z_3$ with $S_3$ (group id \href{https://www.lmfdb.org/Groups/Abstract/18.4}{18.4}), a semidirect product of $\Z_3^2$ with $\Z_4$ (group id~\href{https://www.lmfdb.org/Groups/Abstract/36.9}{36.9}) and the special orthogonal group of plus type in dimension 4 over $\mathbb F_2$ (group id \href{https://www.lmfdb.org/Groups/Abstract/72.40}{72.40}).
The groups in bold are the subgroups of $S_5$. In the column labelled ``Orbit structure'', the string $3^{1},6^{2},12^{1}$ means that the monodromy action on the $27$ lines has one orbit of size three, two of size six and one of size twelve.
}\label{tab:crystallographic}
\end{table}

We note that we obtain the following Galois groups that were not present in \tabref{tab:main}: $\Z_2, \Z_3, \Z_2^2, \Z_6, \Z_2\times\Z_6, \Z_3^3$, and $\Z_3^2$.

Yet another computation of a monodromy group of a family of cubic surfaces appears in \cite{MedranoMartinDelCampo_2022}, which we also independently recover with the methods presented here: 
the unitary group $U_3(\mathbb F_4)\subset W(E_6)$ is the monodromy group of cubic surfaces that are triple covers of $\mathbb P^2$ ramified along a cubic curve, that is, with defining equation of the form $x_0^3 - f_3$ for $f_3\in \C[x_1,x_2,x_3]_3$ a polynomial of degree three. The corresponding linear system has projective dimension $n=10$. The discriminant has one component of degree 12, and one linear component.

These results raise a natural question: can every subgroup of $W(E_6)$ be realised as a monodromy group of a family/pencil of generically smooth cubic surfaces?
Note that it is known that every subgroup of $W(E_6)$ appears as a Galois group of a cubic surface defined over $\Q$ \cite[Theorem 0.1]{ElsenhansJahnel_2015}.
Up to abstract isomorphism (as opposed to ``up to conjugation''), the smallest subgroup we are missing is $D_4$.

We mention a strategy to construct families with a given Galois group. The generators of ${\rm Gal}({\cal L}_G)$ found in our computation correspond to the $\deg \nabla_G$ intersection points of a generic line $L$ with the discriminant $\nabla_G$. Degenerating $L$ to a line $L'$ which intersects $\nabla_G$ non-transversally, we can remove specific generators. This is left for future research. 

The data of the monodromy groups of crystallographic cubic surfaces is available at \cite{code}.

\subsection{Quartic surfaces}\label{sec:quartic}

In this section we apply our methods to smooth quartic surfaces in $\mathbb{P}^3$, which are K3 surfaces.
Similar to Section \ref{sec:introduction}, we let $G$ be a group acting on $\C[x_0, \dots, x_3]_4$ and we define $\mathcal L^4_G$ to be the linear system of quartic surfaces whose defining equation is invariant under $G$.
We use the methods mentioned above to compute the monodromy representation $\pi_1(\mathcal L^4_G \setminus \nabla_G, b)\to \operatorname{Aut}H_2(X_b)$ on the K3 lattice.
Unlike for cubic surfaces, the monodromy group is not necessarily finite, and the monodromy action is not entirely determined by the induced action on rational curves lying on the surface.
However, a generic K3 surface in these families contains a finite number of smooth curves of a given degree, and these are permuted by monodromy.
A result of this section is the computation of the Galois group $\Gal(\calL_G^4)$ associated to the enumerative problems of algebraic cycles on $G$-symmetric quartics.

For example, it is known that a generic surface of the family of \emph{Heisenberg-invariant K3 surfaces} contains 320 conics \cite{Eklund_2018}, and the Galois group of the associated enumerative problem is $\Z_2^{10}$ \cite{Bouyer_2020}.
With the methods presented here, these two facts can be readily computed.
We use \algoref{alg:main} to compute the action of monodromy on the homology classes of these conics.
Our method is fundamentally different from that of \cite{Bouyer_2020} in that we never compute the equations of the conics.
This avoids the use of symbolic methods which might in general involve extensions over large number fields. 

In comparison with cubic surfaces, the computations reported here were more demanding. 
The monodromy computations for some of the considered groups took multiple days. Furthermore, we could not compute the discriminant in several cases, and therefore some results remain uncertified. 
Details are given below. 

We start with some preliminaries and present our computational results in \secref{sec:quartic_results}. 
For more background, we point to \cite{Huybrechts_2016} (Sections 14 and 15) and \cite{Dolgachev_1983}.

\subsubsection{Monodromy of quartic surfaces}
The homology lattice of a smooth quartic surface has rank 22 and its intersection product is the unique even unimodular lattice of signature $(3,19)$, that is, $U^3\oplus E_8(-1)^2$, where
\begin{equation}
\begin{gathered}
U = \left(\begin{array}{rr}
 & 1 \\
 1 &   \\
\end{array}\right)\qquad
E_8(-1) = \scriptsize \left(\begin{array}{rrrrrrrr}
 -2& 1 &  &  &  &  &  &  \\
  1&-2 &1  &  &  &  &  &  \\
  &1 &-2  & 1 &  &  &  &  \\
  & & 1 &-2  &1  &  &  &  \\
  & &  & 1 &-2  &1  &1  &  \\
  & &  &  & 1 &-2  &  &  \\
  & &  &  & 1 &  &-2  & 1 \\
  & &  &  &  &  & 1 &-2 
\end{array}\right) \normalsize \,.
\end{gathered}
\end{equation}
Smooth quartic surfaces in $\mathbb{P}^3$ are of K3 type: they admit a unique (up to scaling) non-vanishing holomorphic 2-form.
If $f$ is the defining equation of such a surface, then the holomorphic form is given by the residue of $\Omega_3/f$, with $\Omega_3$ as in \eqref{eq:Omega}.
The existence of such a form implies that, unlike cubic surfaces, smooth quartic surfaces are not rational.
In particular, their homology lattice $H_2(X)$ splits into two parts over $\Q$: $H_2(X)_{\Q} = \Pic(X)_{\Q} \oplus \Tr(X)_{\Q}$. The \emph{Picard lattice} $\Pic(X)$ (or \emph{Picard group} or \emph{N\'eron-Severi lattice}) is generated by homology classes of algebraic curves. Its orthogonal complement, the \emph{transcendental lattice}, is $\Tr(X)$.

When considering a family $\{X_z\}_{z\in \bbP^n}$ of generically smooth quartic surfaces instead of a single one, the \emph{generic Picard lattice}, which we will denote by $\Pic(X_z)$, is the set of homology classes $\gamma_z\in H_2(X_z)$ such that for all $z_0 \in \mathbb{P}^n \setminus \nabla$ we have $\gamma_{z_0} \in \Pic(X_{z_0})$. 

The generic Picard lattice defines a subspace of $H_2(X_z)$ which is stable under monodromy.
Similarly, the \emph{generic transcendental lattice} $\Tr(X_z)$ is the orthogonal complement of the generic Picard lattice. 
Since monodromy preserves the intersection product, $\Tr(X_z)$ is also stable under monodromy. 
Hence, the monodromy representation splits up into two parts
\begin{equation} \label{eq:splitmonodromy}
\pi_1(\bbP^n\setminus\nabla, b) \, \to \,   \operatorname{Aut}(\Pic(X_z)) \oplus \operatorname{Aut}(\Tr(X_z)) \, \subset \,  \operatorname{Aut}(H_2(X_b))\,.
\end{equation}

With the methods presented in the previous sections of this text, we are able to compute this monodromy representation by providing a generic line $L\subset \bbP^n$ containing $b$, a basis of $\pi_1(L \setminus \nabla, b)$ and the $22\times 22$ matrices of the action of monodromy on $H_2(X_b)$ (in a given basis of homology).
The monodromy group is not necessarily finite. 
However, there are only finitely many smooth rational curves of a given degree lying on a smooth quartic surface.
In the next section, we will give a method for identifying the generic Picard lattice $\Pic(X_z)$ in terms of our description of $H_2(X_z)$ for families of surfaces with certain symmetry groups and use this to identify the homology classes of these curves. 
This way we can compute the shadow of the action of monodromy that is given by their permutation.

\subsubsection{Identifying the Picard lattice and rational curves}\label{sec:identify_pic}

Similar to the case of lines on cubics surfaces (see \secref{sec:cubicsurfaces}), the following proposition allows to identify homology classes of rational curves in $H_2(X)$.
\begin{proposition}\label{prop:identifying_curves}
Smooth rational curves of degree $d$ on a quartic K3 surface $X$ are isolated in their homology class $\gamma$, and are identified by the fact that $\gamma\in \Pic(X)$, $\gamma^2 = -2$, $\gamma\cdot H=d$, and $\gamma\cdot \eta\ge 0$ for all classes $\eta$ of smooth rational curves of degree strictly less than $d$.
\end{proposition}
A proof of this statement as well as an efficient algorithm for computing such classes assuming the knowledge of $\Pic(X)$ is given in \cite[\S3]{LairezSertoz_2019}.

We now explain how we compute the generic Picard lattices of our families. 
In short, the approach is as follows. 
We obtain a lower bound on the Picard lattice of $X$ from the group action, and we show that this bound is reached for a given surface $X_b$ in the family by showing that a determinant of periods is non-zero. 
This allows us to certifiably identify the generic Picard lattice of $X_t$ in our representation of $H_2(X_t)$.

The first ingredient is a proposition due to Nikulin \cite{Nikulin_1979}, used to find a lower bound on the Picard lattice.
Let $X$ be a K3 surface and $\omega\in H^2(X)$ a non-zero holomorphic differential form.
An automorphism $\sigma$ of $X$ is \emph{symplectic} if $\sigma^*\omega = \omega$ and \emph{non-symplectic} otherwise.
\begin{proposition} \label{prop:symplectic}
Let $H_2(X)^{\sigma_*}$ be the sublattice of $H_2(X)$ invariant under $\sigma_*$.
If $\sigma$ is symplectic, then
$\Tr(X) = \Pic(X)^\perp \subset H_2(X)^{\sigma_*}$.
If $\sigma$ is non-symplectic, then
${H_2(X)^{\sigma_*}\subset \Pic(X)}$.
\end{proposition}

In this proposition, $\sigma_*: H_2(X) \rightarrow H_2(X)$ is induced by $\sigma: X \rightarrow X$. 
Given a linear automorphism $g \in \operatorname{GL_4}(\C)$ with $g^*f = f$, we see from \eqref{eq:group_action} and the fact that the holomorphic form $\omega$ is $\res\left({\Omega_3}/{f}\right)$ that the automorphism induced by $g$ is symplectic if and only if $\det(g)=1$.
For each $g$ we obtain a sublattice of $\Pic(X)$: if $g$ is non-symplectic, then ${H_2(X)^{g_*}\subset \Pic(X)}$ and if $g$ is symplectic, then $({H_2(X)^{g_*})^\perp \subset \Pic(X)}$. 
The saturation of the sum of these lattices for all $g \in G$ gives a lower bound $\Lambda \subset \Pic(X)$, and hence an upper bound $\Tr(X) \subset \Lambda^\perp$. 
The $g$-invariant sublattice $H_2(X)^{g_*}$ is obtained from the kernel of the matrix ${\rm Mat}_{g_*} -I_{22}$, which in turn is computed using the method of \secref{sec:group_action} from an approximate period matrix via \eqref{eq:Matg}. 
The period matrix $\Pi$ is obtained using the methods from \cite{LairezEtAl_2024}. 

An upper bound for $\Pic(X)$ is obtained via a theorem attributed to Lefschetz \cite[p.~163]{GriffithsHarris_1978}.
\begin{theorem}[Lefschetz's theorem on $(1,1)$-classes for K3 surfaces]
A class $\gamma\in H_2(X)$ belongs to $\Pic(X)$ if and only if $\int_\gamma \omega = 0$ for any nonzero holomorphic form $\omega$ on $X$.
\end{theorem}
Since we are able to compute periods numerically with certified error bounds, we can prove that a class $\gamma$ is \emph{not} in $\Pic(X)$ by showing that the corresponding period $\int_\gamma \omega$ is nonzero.
To compute an upper bound on the generic Picard lattice, we pick a generic line $L\subset {\cal L}_h$ in the parameter space and consider the one-parameter family $\{X_t\}_{t\in L}$.
We let $\omega_t = \res\left({\Omega_3}/{f_t}\right)$ be a  section of the bundle of holomorphic forms for our family $X_t = V(f_t)$. To verify that $\gamma_t \notin \Pic(X_t)$, it is sufficient to show that one of the derivatives $\partial^k_t\int_{\gamma_t}\omega_t = \int_{\gamma_t} \nabla_t^k \omega_t$ does not vanish.
In particular if $\gamma_{1,t}, \dots, \gamma_{r,t}\in H_2(X_t)$ are such that $\det ( \int_{\gamma_{i,t}}\nabla_t^{k-1}\omega_t)_{1\le k,i\le r}$ is nonzero, then the generic transcendental lattice has rank at least $r$.

Finally, we note that the transcendental lattice $\Tr(X_z)$ is \emph{primitive}, meaning that it is not properly contained in a sublattice of $H_2(X_z)$ of the same rank.
In particular, if ${\rm rank} \Tr(X_z) =   {\rm rank} \, \Lambda^\perp$, then $\Tr(X_z) =  \Lambda^\perp$. 
We summarise these steps in the following:
\begin{enumerate}
\itemsep0em
\item Compute a sublattice $\Lambda \subset \Pic(X_z)$ via Proposition \ref{prop:symplectic}. We know that $\Tr(X_z)\subset \Lambda^\perp$.
\item Compute a basis $\gamma_{1,t}, \dots, \gamma_{r,t}$ of $\Lambda^\perp$.
\item If $\det\left(\int_{\gamma_{i,t}}\nabla_t^{k-1}\res \left(\frac{\Omega_3}{f_t}\right)\right)_{1\le i,k\le r} \neq 0$ at some $t_0 \in L\subset \bbP^n$, then $\Tr(X_z) =  \Lambda^\perp$. 
\item We recover the Picard lattice as $\Pic(X_z) = \Lambda$ as both lattices are primitive.
\end{enumerate}
In all the cases we considered, this approach recovered the generic Picard lattice.

\subsubsection{Computational results}\label{sec:quartic_results}

We compute monodromy groups acting on conics contained in a quartic surface $X \subset \mathbb{P}^3$ whose defining equation is invariant under the action of some finite group $G$. For the $G$-invariant families considered here, these conics generate the generic Picard lattice. Hence, we compute the restricted monodromy action to ${\rm Pic}(X)$. 

As the lattice $\Pic(X)\cap H^\perp$ is negative definite and $H$ is invariant under monodromy, the restriction of the monodromy group to $\Pic(X)$ is finite.
However, the knowledge of the action of monodromy on $\Pic(X)$ is not sufficient to recover the action of monodromy on the full $H_2(X)$, unlike in the case of cubic surfaces.
The monodromy group on the Picard lattice is a subgroup of the orthogonal group of $\Pic(X)\cap H^\perp$, and we denote its index by ${\cal I}_G$.

\tabref{tab:quartic} contains the data pertaining to the families and Galois groups of the Picard lattices of surfaces defined by $G$-invariant quarternary quartics, where $G$ is a subgroup of $S_5$ or the Heisenberg group $\mathcal H$ (see below for the definition of $\mathcal H$). Here, the Galois group ${\rm Gal}(\calL^4_G)$ is the subgroup of ${\rm Aut}({\rm Pic}(X_z))$ obtained as the projection of the image of \eqref{eq:splitmonodromy}.
The table shows the dimension $\dim{\cal L}^4_G$ of ${\cal L}^4_G \simeq \mathbb{P}^{\dim{\cal L}^4_G}$, the generic Picard rank $\rho_G$ of the family, the isomorphism class of the Galois group ${\rm Gal}({\cal L}^4_G)$ and its order $|{\rm Gal}({\cal L}^4_G)|$, the generic number of conics $|c_G|$ contained in a generic surface in ${\cal L}^4_G$, the orbit structure of these conics (see below), the index ${\mathcal I}_G$ mentioned above, and the degrees of the components of $\nabla_G$.

Computing the discriminants for these families is more challenging than in the case of cubic surfaces. This is to be expected from the fact that the discriminant of a generic quartic surface has degree 108, compared to which the degree 32 for cubic surfaces is small. Our implementation of the methods from Section \ref{sec:discriminants} does not terminate within reasonable time for many rows in the table, namely, those with a question mark to the right. The reported numbers are obtained by restricting to a line, which we picked at random. The impact of this difficulty on the certifiable correctness of the table are summarised in the following theorem. 

\begin{theorem}\label{thm:S5_quartics}
In \tabref{tab:quartic}, the columns $\dim\mathcal L^4 _G$, $\rho_G$ and $|c_G|$ are certified for all rows.
The rows where $\deg\nabla_G$ does not have a question mark are fully certified.
For those with a question mark, we give the discriminant of the restriction to a line which we picked at random.
If this line is generic in the sense of \thmref{thm:lefschetz_zariski}, then the row is certified.
If the line is not generic, then the Galois group given in this row is certified to be the one of the subspace given by this line.
In particular, the Galois group of the full family is necessarily larger.
For all families, the generic surface contains no lines (i.e., no rational curve of degree one). 
In cases with $c_G >0$, the conics generate the generic Picard lattice.
\end{theorem}

\begin{table}[h]
\footnotesize
\centering
\makebox[\linewidth][c]{
\setlength{\tabcolsep}{3pt}
\begin{tabular}{cccccccccc}
\toprule
$G$ & $\dim {\cal L}^4_G$ & $\rho_G$ & ${\rm Gal}({\cal L}^4_G)$ &$|{\rm Gal}({\cal L}^4_G)|$ & $|c_G|$ & Orbit structure & ${\mathcal I}_G$ & \multicolumn{2}{c}{$\deg \nabla_G$ }\\ \midrule
  $\mathcal H$  &  4    &  16      &  $\Z_2^{10}$&  1024  &    320   &  $32^{10}$ & 40320 &  $1^{15},3^1$       \\
  $S_5$ and $A_5$ & 1 & 19       &   $\Z_2\times D_6$     & 24&        560 & $2^{10},4^{15},12^{40}$   & 360&  $1^5$  \\
  $F_5$ & 2 & 18 &  $\Z_2\times S_5$ & 240 & 400 & $40^{10}$ &20 & $1^3,2^3$ \\
  $S_4$ and $A_4$  & 4 & 17  &$\operatorname{GL}_2(\Z_4)\rtimes\Z_2^2$ & 384 & 344 & $2^{4},8^{3},12^{6},32^{6},48^{1}$ & 288 & $1^4,2^1,5^2$ \\
  $D_5$ & 4 & 17 & -- & 14400& 0  & $0^{1}$ &20& \textcolor{gray}{$2^1,6^2$ }&\textcolor{gray}{(?)} \\
  $D_6$ & 6 & 16 &  $D_6\cdot W(F_4)$ & 13824 & 200 & $2^{1},6^{3},36^{1},48^{3}$ & 6 & $1^4,2^2,5^1,8^1$ \\
  $\Z_5$ & 6 & 17  & -- & $72000$ & 0 & $0^1$ & $4$ &  \textcolor{gray}{$4^1,12^1$ }&\textcolor{gray}{(?)} \\
  $D_4$ & 7 & 15 &  -- & 24576 & 200 & $8^{1},16^{2},32^{2},96^{1}$ & 16 &  \textcolor{gray}{$1^5,2^3,5^1,10^1$ }&\textcolor{gray}{(?)} \\
  $\tS$ & 7 & 15  & -- & 622080 & 0 & $0^{1}$ &6&  \textcolor{gray}{$1^2,2^1,5^2,12^1$ }&\textcolor{gray}{(?)} \\
  $\Z_6$ & 7 & 16 &  -- & 41472 & 200 & $2^{1},18^{1},36^{1},144^{1}$ & 2 &  $1^3,2^3,5^1,8^1$ & \\
  $\Z_4$ & 9 & 14  & -- & $245760$ & 160 & $160^1$ & $2$ &  \textcolor{gray}{$1^2,2^3,24^1$ }&\textcolor{gray}{(?)} \\
  $S_3$ & 10 & 14  & $\Z_2\times W(E_6)$  & 103680 & 164 & $2^{1},54^{3}$ & 6 &  \textcolor{gray}{$1^1,6^2,21^1$ }&\textcolor{gray}{(?)} \\
  $K_4$ & 10 & 13  & -- & 393216 & 128 & $128^{1}$ & 648 &   \textcolor{gray}{$1^4,2^6,20^1$ }&\textcolor{gray}{(?)}  \\
  $\Z_3$ & 12 & 13  & -- & $39191040$ & 0 & $0^1$ & $2$ &  \textcolor{gray}{$6^1,34^1$ }&\textcolor{gray}{(?)} \\
  $\nK$ & 13 & 12 &  -- & 1769472 & 104 &$8^{1},48^{2}$ & 2 &  \textcolor{gray}{$1^2,2^1,6^1,10^2,14^1$ }&\textcolor{gray}{(?)} \\
  $\DT$ & 18 & 9  & $W(E_8)$ & $696729600$ & 0 & $0^1$ & 1 &  \textcolor{gray}{$6^2,48^1$ }&\textcolor{gray}{(?)} \\
  $\Z_2$ & 21 & 8 &  $W(E_7)$ & 2903040 & 56 &$56^{1}$ & 1 &  \textcolor{gray}{$1^1,27^1,40^1$ }&\textcolor{gray}{(?)} \\
  $\Z_1$ & 34 & 1 &  $\Z_1$ & 1 & 0 &$0^{1}$ & 1 & $108^1$ \\
  \bottomrule
\end{tabular}
}
\caption{Galois groups, conics and discriminant degrees of $G$-invariant quartic surfaces. The cardinal of the full automorphism group of $\Pic(X)\cap H^\perp$ is the product of the numbers in $|{\rm Gal}({\cal L}^4_G)|$ and $\mathcal I_G$.}
\label{tab:quartic}
\end{table}

In the column labelled ``Orbit structure'', the string $2^{10},4^{15},12^{40}$ means that the monodromy action on the conics has ten orbits of size two, fifteen of size four and forty of size twelve.
The strings in the column $\deg\nabla_G$ have the same meaning as in \tabref{tab:main}.
The symbol ``$\rtimes$'' in ``$\operatorname{GL}_2(\Z_4)\rtimes\Z_2^2$'' signifies a semidirect product, with normal subgroup $\operatorname{GL}_2(\Z_4)$ (group id \href{https://www.lmfdb.org/Groups/Abstract/384.20051}{384.20051}). 
The dot ``$\cdot$'' in ``$D_6\cdot W(F_4)$'' signifies an extension with normal subgroup $D_6$ and quotient $W(F_4)$ (group id \href{https://www.lmfdb.org/Groups/Abstract/13824.o}{13824.o}).
Groups for which the column ${\rm Gal}({\cal L}_G)$ is empty are those for which ${\rm Gal}({\cal L}_G)$ could not be identified in the LMFDB group database \cite{lmfdb}.
These groups are given explicitly as permutation groups or matrix groups in \cite{code}, where we also provide the monodromy representation on the full $H_2(X)$.

\begin{proof}[Proof of \thmref{thm:S5_quartics}]
For all cases in \tabref{tab:quartic} except $\Z_1$, we were able to identify $\Pic(X)$ in $H_2(X)$ using the results of \secref{sec:identify_pic} and compute the image of the monodromy action ${\pi_1(L\setminus \nabla_G)\to \operatorname{Aut}\Pic(X)}$ for a random line $L\subset \bbP^n$ using Algorithm \ref{alg:main}.
We found the homology classes of algebraic curves of degree $\le 2$ from the intersection product on $\Pic(X)$ using \propref{prop:identifying_curves}.
The case $\Z_1$ is trivial; the generic Picard lattice is generated by~$H$.
\end{proof}

\begin{remark}
We note here a subtlety for the computation of the analytic continuation of the periods (see \secref{sec:numerical_integration}).
If one sets this up naively, the coefficients of the Picard--Fuchs equations become too large for the computation to be feasible, both because of time complexity and space complexity.
What enabled us to obtain the results in \tabref{tab:quartic} is that we found a generating set of algebraic forms $\omega_{1,t}, \omega_{2,t},\ldots\in H^2(X_t)$ for which the Picard--Fuchs equations are of small enough order to be integrated.
\end{remark}

\begin{remark}
The cases $D_5$, $\Z_5$, $\tS$ and $\Z_3$ do not have smooth conics nor smooth rational cubics, and we identify using \propref{prop:identifying_curves} respectively 41280, 41280, 11808 and 4032 smooth rational quartic curves, which generate the generic Picard lattice.
For $G = \DT$, the entries of the intersection product on the generic Picard lattice are even with multiples of 4 on the diagonal, which implies that there are no classes of self-intersection -2. Hence, by \propref{prop:identifying_curves}, a generic surface in $\calL^4_{\DT}$ contains no smooth rational curves.
\end{remark}

\begin{remark}
The case $\Z_2$ should be compared to the results of \cite{MedranoMartinDelCampo_2022a} on monodromy of a certain family of Del Pezzo surfaces with 56 lines.
There the author also considers quartic surfaces given as four-fold covers of $\bbP^2$ ramified along a smooth quartic curve, i.e., with defining equation of the form $x_0^4 - f_4=0$ with $f_4 \in \C[x_1,x_2,x_3]_4$.
Up to a change of coordinates, it is a subsystem of the case $\Z_2$.
The linear system of such surfaces has projective dimension 15, and a generic surface in this system has 56 conics in a single monodromy orbit, as well as a $\Z_4$ automorphism group induced by $w\mapsto iw$.
The discriminant has degree $1^1,27^1$ and the Galois group of the $56$ conics is also $W(E_7)$.
\end{remark}

Compared to the cubic surface case, as the number of conics depends on the specific family that is being considered, the monodromy groups cannot be directly compared to one another. 
However, if $G<G'$, then the lattice of classes of curves on a generic element of $\mathcal L^4_{G'}$ that are imposed by the automorphisms induced by $G$ is stable under the monodromy on $\mathcal L^4_{G'}$ as it is a restriction of the monodromy in $\mathcal L^4_{G}$.
For example, among the 560 conics of a generic element of $\mathcal L^4_{S_5}$, 400 come from the fact that $\mathcal L^4_{S_5}\subset \mathcal L^4_{F_5}$.
In particular, the monodromy group in $\mathcal L^4_{S_5}$ can be restricted to the 400 conics stemming from the $F_5$-invariance.
It turns out that the restriction is still isomorphic to $\Z_2\times D_6$, and the orbit structure is $4^{10},12^{30}$.
Such a comparison can be efficiently recovered from the data we have available at \cite{code}\\

We now turn to \emph{Heisenberg-invariant} K3 surfaces as defined in \cite{Eklund_2018, Bouyer_2018}. 
The Heisenberg group $\mathcal H$ is the group generated by the linear automorphisms
\begin{equation}
(x_0, x_1, x_2, x_3)\mapsto (x_1, x_0, x_3, x_2), (x_2, x_3, x_0, x_1), (x_0, x_1, -x_2, -x_3), (x_0, -x_1, x_2, -x_3)
\end{equation}
Following the notation of \cite{Bouyer_2018}, we parametrise Heisenberg invariant quartic surfaces by
\begin{equation}
A (x_0^4+x_1^4+x_2^4+x_3^4) + Bx_0x_1x_2x_3 + C(x_0^2x_1^2 + x_2^2x_3^2) + D(x_0^2x_2^2 + x_1^2x_3^2) + E(x_0^2x_3^2 + x_1^2x_2^2).
\end{equation}
Here $[A:B:C:D:E]$ are coordinates on $\mathcal L^4_{\mathcal H}\simeq \bbP^4$.
We will take interest in the following subfamilies of $\mathcal L^4_{\mathcal H}$, introduced in \cite{Bouyer_2018}:
\begin{itemize}
\item the space of Heisenberg invariant K3 surfaces $\mathcal L^4_{\mathcal H}$, denoted $\mathcal X$ in \cite{Bouyer_2018};
\item the three parameter subfamily $[A: (DE-2AD)/A: C: D: E]$ in $ \mathcal L^4_{\mathcal H}$, denoted $\mathcal X_{C,D,E}$;
\item the two parameter subfamily $[A:0:C:D:2AC/D]$ in $ \mathcal L^4_{\mathcal H}$, denoted $\mathcal X_{C,D}$;
\item the one parameter subfamily $[A: B(2A-B)/A: B:B:B]$ in $ \mathcal L^4_{\mathcal H}$, denoted $\mathcal X_{B}$;
\item the one parameter subfamily $[A: 0:C:0:0]$ in $ \mathcal L^4_{\mathcal H}$, denoted $\mathcal X_{C}$.
\end{itemize}
Note that among these families, only ${\cal X}$ and ${\cal X}_C$ are linear systems. 
Nevertheless, we are still able to apply \algoref{alg:main} to compute their monodromy group. 
Indeed, we are able to use the methods of \secref{sec:discriminants} to compute the defining equation of the discriminant variety of $\mathcal L_{\mathcal H}^4$ (which was in fact already computed in \cite[Proposition 2.1]{Eklund_2018}), and we may simply restrict it to the different subfamilies to recover their discriminant.
We generate the fundamental group by taking a generic line in the parameter space of each family.

\begin{theorem}
A generic surface in $\mathcal L^4_{\mathcal H}$ has Picard rank 16 and contains 320 conics.
The monodromy group on the conics is isomorphic to $\Z_2^{10}$.
When restricted to $\mathcal X_{C,D,E}$, $\mathcal X_{C,D}$, $\mathcal X_{C}$, and $\mathcal X_{B}$, the monodromy group on the 320 conics is reduced to $\Z_2^9$, $\Z_2^6$, $\Z_2^2$ and $\Z_2^3$ respectively.
\end{theorem}

In some of these families, some conics degenerate to pairs of lines, and certain transcendental classes become algebraic and produce more conics, see \cite{Bouyer_2018} for details.
We can detect this heuristically by recovering the Picard lattice from integer linear relations between the holomorphic periods \cite[\S2]{LairezSertoz_2019}.
However, as we only get numerical approximations of the periods, we cannot certify the computation of the (generic) Picard lattice with this approach.
Nevertheless, with \emph{a priori} knowledge of the structure of the Picard lattice, we can sometimes identify it in our computation of $H_2(X)$ from the structure of the monodromy representation.

\newpage 
\appendix
\section{Monodromy action on symmetric surfaces - data}\label{appendix:permutations_symmetric}

\begin{table}
\scriptsize 
\centering 
\begin{tabular}{ccc}
\toprule
\thead{Group} & \thead{Generator} & \thead{Action on lines} \\ 
\midrule\\[-.5em]
\multirow{2}{*}{$S_5$}&$(1\,4\,5\,3\,2)$&$(2,5,3,4,6)(8,11,9,10,12)(13,16,14,25,15)(17,23,27,18,24)(19,26,21,20,22)$ \\
&$(1\,2)$&$(1,10)(2,12)(3,11)(4,7)(5,9)(6,8)(13,22)(14,23)(15,18)(16,20)(17,24)(26,27)$ \\[1em]
\multirow{2}{*}{$A_5$}&$(1\,4\,5\,3\,2)$&$(2,5,3,4,6)(8,11,9,10,12)(13,16,14,25,15)(17,23,27,18,24)(19,26,21,20,22)$ \\
&$(1\,2\,3)$&$(1,3,2)(4,6,5)(7,9,8)(10,12,11)(13,14,17)(15,21,18)(16,20,19)(22,24,23)(25,27,26)$ \\[1em]
\multirow{2}{*}{$F_5$}&$(1\,2\,3\,4\,5)$&$(1,6,2,3,4)(7,12,8,9,10)(13,27,18,14,22)(15,19,17,20,25)(16,24,26,21,23)$ \\
&$(2\,3\,5\,4)$&$(1,9,2,10)(3,8,4,7)(5,11)(6,12)(13,20)(14,17,18,25)(15,27,19,22)(16,21,26,23)$ \\[1em]
\multirow{3}{*}{$D_6$}&$(1\,2\,3)$&$(1,3,2)(4,6,5)(7,9,8)(10,12,11)(13,14,17)(15,21,18)(16,20,19)(22,24,23)(25,27,26)$ \\
&$(1\,2)$&$(1,10)(2,12)(3,11)(4,7)(5,9)(6,8)(13,22)(14,23)(15,18)(16,20)(17,24)(26,27)$ \\
&$(4\,5)$&$(1,11)(2,12)(3,10)(4,9)(5,7)(6,8)(13,24)(14,23)(15,26)(17,22)(18,27)(21,25)$ \\[1em]
\multirow{2}{*}{$D_5$}&$(1\,4\,5\,3\,2)$&$(2,5,3,4,6)(8,11,9,10,12)(13,16,14,25,15)(17,23,27,18,24)(19,26,21,20,22)$ \\
&$(1\,2)(3\,4)$&$(2,5)(3,6)(8,11)(9,12)(13,16)(14,15)(17,24)(18,23)(19,21)(20,22)$ \\[1em]
\multirow{2}{*}{$\tS$}&$(1\,2\,5)$&$(1,2,5)(3,6,4)(7,8,11)(9,12,10)(13,26,16)(14,19,23)(15,18,21)(17,24,25)(20,27,22)$ \\
&$(1\,2)(3\,4)$&$(2,5)(3,6)(8,11)(9,12)(13,16)(14,15)(17,24)(18,23)(19,21)(20,22)$ \\[1em]
\multirow{2}{*}{$\Z_6$}&$(1\,2\,3)$&$(1,3,2)(4,6,5)(7,9,8)(10,12,11)(13,14,17)(15,21,18)(16,20,19)(22,24,23)(25,27,26)$ \\
&$(4\,5)$&$(1,11)(2,12)(3,10)(4,9)(5,7)(6,8)(13,24)(14,23)(15,26)(17,22)(18,27)(21,25)$ \\[1em]
\multirow{1}{*}{$\Z_5$}&$(1\,4\,5\,3\,2)$&$(2,5,3,4,6)(8,11,9,10,12)(13,16,14,25,15)(17,23,27,18,24)(19,26,21,20,22)$ \\[.5em]
\midrule\\[-.5em]
\multirow{2}{*}{$S_4$}&$(1\,2\,3\,4)$&$(1,11,4,8)(2,7,5,10)(3,9)(6,12)(13,16,23,18)(14,21,20,17)(15,24,22,19)(25,26)$ \\
&$(1\,2)$&$(1,10)(2,12)(3,11)(4,7)(5,9)(6,8)(13,22)(14,23)(15,18)(16,20)(17,24)(26,27)$ \\[1em]
\multirow{2}{*}{$A_4$}&$(1\,2\,3)$&$(1,2,3)(4,5,6)(7,8,9)(10,11,12)(13,17,14)(15,18,21)(16,19,20)(22,23,24)(25,26,27)$ \\
&$(1\,2)(3\,4)$&$(2,5)(3,6)(8,11)(9,12)(13,16)(14,15)(17,24)(18,23)(19,21)(20,22)$ \\[1em]
\multirow{2}{*}{$D_4$}&$(1\,3\,2\,4)$&$(1,7)(2,9,5,12)(3,11,6,8)(4,10)(13,14,16,15)(17,21,24,19)(18,20,23,22)(26,27)$ \\
&$(1\,2)$&$(1,10)(2,12)(3,11)(4,7)(5,9)(6,8)(13,22)(14,23)(15,18)(16,20)(17,24)(26,27)$ \\[1em]
\multirow{2}{*}{$S_3$}&$(1\,2\,3)$&$(1,2,3)(4,5,6)(7,8,9)(10,11,12)(13,17,14)(15,18,21)(16,19,20)(22,23,24)(25,26,27)$ \\
&$(1\,2)$&$(1,10)(2,12)(3,11)(4,7)(5,9)(6,8)(13,22)(14,23)(15,18)(16,20)(17,24)(26,27)$ \\[1em]
\multirow{2}{*}{$K_4$}&$(1\,2)(3\,4)$&$(2,5)(3,6)(8,11)(9,12)(13,16)(14,15)(17,24)(18,23)(19,21)(20,22)$ \\
&$(1\,3)(2\,4)$&$(1,4)(2,5)(7,10)(8,11)(13,23)(14,20)(15,22)(16,18)(17,21)(19,24)$ \\[1em]
\multirow{2}{*}{$\nK$}&$(1\,2)$&$(1,10)(2,12)(3,11)(4,7)(5,9)(6,8)(13,22)(14,23)(15,18)(16,20)(17,24)(26,27)$ \\
&$(3\,4)$&$(1,10)(2,9)(3,8)(4,7)(5,12)(6,11)(13,20)(14,18)(15,23)(16,22)(19,21)(26,27)$ \\[1em]
\multirow{1}{*}{$\Z_4$}&$(1\,3\,2\,4)$&$(1,7)(2,9,5,12)(3,11,6,8)(4,10)(13,14,16,15)(17,21,24,19)(18,20,23,22)(26,27)$ \\[1em]
\multirow{1}{*}{$\Z_3$}&$(1\,2\,3)$&$(1,2,3)(4,5,6)(7,8,9)(10,11,12)(13,17,14)(15,18,21)(16,19,20)(22,23,24)(25,26,27)$ \\[1em]
\multirow{1}{*}{$\Z_2$}&$(1\,2)$&$(1,10)(2,12)(3,11)(4,7)(5,9)(6,8)(13,22)(14,23)(15,18)(16,20)(17,24)(26,27)$ \\[1em]
\multirow{1}{*}{$\DT$}&$(1\,2)(3\,4)$&$(2,5)(3,6)(8,11)(9,12)(13,16)(14,15)(17,24)(18,23)(19,21)(20,22)$ \\[1em]
\multirow{1}{*}{$\Z_1$}&$\operatorname{id}$&$\operatorname{id}$ \\[.5em]

\bottomrule
\end{tabular}
\caption{The action of the subgroups $G$ of $S_5$ on the 27 lines of the cubic surface. 
The realisation of the groups are those given in \tabref{tab:subgroups_S5} and the lines are numbered as in \secref{subsec:homology_cubic_surface}.}\label{tab:actions}
\end{table}

\begin{table}
\scriptsize 
\centering 
\begin{tabular}{cc}
\toprule
\thead{$G$} & \thead{Generators of $\Gal(\mathcal L_{G})$} \\ 
\midrule\\[-.5em]
\multirow{2}{*}{$D_6$}&$(4,24)(5,22)(6,23)(7,17)(8,14)(9,13)$ \\
&$(1,5)(2,6)(3,4)(7,11)(8,12)(9,10)(13,24)(14,23)(15,26)(17,22)(18,27)(21,25)$ \\[1em]
\multirow{1}{*}{$D_5$}&$\operatorname{id}$ \\[1em]
\multirow{2}{*}{$\tS$}&$(3,22)(4,27)(6,20)(7,26)(8,16)(11,13)$ \\
&$(1,4)(2,3)(5,6)(7,10)(8,9)(11,12)(13,20)(14,18)(15,23)(16,22)(19,21)(26,27)$ \\[1em]
\multirow{2}{*}{$\Z_6$}&$(1,2,3)(4,22,6,24,5,23)(7,14,9,17,8,13)(10,11,12)(15,18,21)(16,19,20)(25,26,27)$ \\
&$(1,5)(2,6)(3,4)(7,11)(8,12)(9,10)(13,24)(14,23)(15,26)(17,22)(18,27)(21,25)$ \\[1em]
\multirow{1}{*}{$\Z_5$}&$(2,3,6,5,4)(8,9,12,11,10)(13,14,15,16,25)(17,27,24,23,18)(19,21,22,26,20)$ \\[1em]
\midrule\\[-.5em]
\multirow{2}{*}{$S_4$}&$(1,4)(2,5)(3,6)(7,10)(8,11)(9,12)(13,23)(14,22)(15,20)(16,18)(17,24)(19,21)$ \\
&$(1,7)(2,8)(3,9)(4,10)(5,11)(6,12)$ \\[1em]
\multirow{2}{*}{$A_4$}&$(1,4)(2,5)(3,6)(7,10)(8,11)(9,12)(13,23)(14,22)(15,20)(16,18)(17,24)(19,21)$ \\
&$(1,7)(2,8)(3,9)(4,10)(5,11)(6,12)$ \\[1em]
\multirow{3}{*}{$D_4$}&$(2,5)(3,6)(8,11)(9,12)(13,16)(14,15)(17,24)(18,23)(19,21)(20,22)$ \\
&$(1,4)(7,10)(13,18)(14,20)(15,22)(16,23) \qquad (1,7)(2,8)(3,9)(4,10)(5,11)(6,12)$ \\[1em]
\multirow{2}{*}{$S_3$}&$(1,3,2)(4,22,6,24,5,23)(7,13,8,17,9,14)(10,11,12)(15,20,26)(16,27,18)(19,25,21)$ \\
&$(1,5)(2,6)(3,4)(7,11)(8,12)(9,10)(13,24)(14,23)(15,26)(17,22)(18,27)(21,25)$ \\[1em]
\multirow{2}{*}{$K_4$}&$(3,6)(9,12)(14,15)(17,19)(20,22)(21,24)\qquad (2,5)(8,11)(13,16)(17,21)(18,23)(19,24)$ \\
&$(1,4)(7,10)(13,18)(14,20)(15,22)(16,23)\qquad (1,7)(2,8)(3,9)(4,10)(5,11)(6,12)$ \\[1em]
\multirow{3}{*}{$\nK$}&$(3,22)(4,27)(6,20)(7,26)(8,16)(11,13)$ \\
&$(2,3)(5,6)(8,9)(11,12)(13,14)(15,16)(18,20)(19,21)(22,23)(26,27)$ \\
&$(1,4,27)(2,5)(3,14,22,6,15,20)(7,26,10)(8,18,16,11,23,13)(9,12)(17,24)(19,21)$ \\[1em]
\multirow{3}{*}{$\Z_4$}&$(2,3,5,6)(8,9,11,12)(13,14,16,15)(17,21,24,19)(18,20,23,22)(26,27)$ \\
&$(1,4)(7,10)(13,18)(14,20)(15,22)(16,23) \qquad (1,7)(2,8)(3,9)(4,10)(5,11)(6,12)$ \\[1em]
\multirow{2}{*}{$\Z_3$}&$(1,6,2,4,3,5)(7,12,8,10,9,11)(13,22,17,23,14,24)(15,19,18,20,21,16)(25,27,26)$ \\
&$(1,13,2,17,3,14)(4,5,6)(7,9,8)(10,22,12,24,11,23)(15,20,26)(16,27,18)(19,25,21)$ \\[1em]
\multirow{2}{*}{$\Z_2$}&$(1,15,12)(2,10,18)(3,16,26)(4,9,6,24,13,23)(5,8,17,22,14,7)(11,20,27)(19,25,21)$ \\
&$(1,26,4)(2,18,16,5,23,13)(3,6)(7,10,27)(8,11)(9,14,22,12,15,20)(17,24)(19,21)$ \\[1em]
\multirow{3}{*}{$\DT$}&$(3,6)(9,12)(14,15)(17,19)(20,22)(21,24) \qquad (1,4)(7,10)(13,18)(14,20)(15,22)(16,23)$ \\
&$(1,4,27,10,7,26)(2,15,8,9,23,3)(5,14,11,12,18,6)(13,20)(16,22)(19,21)$ \\[1em]
\multirow{2}{*}{$\Z_1$}&$(1,10,23,17,16,18,4,9)(2,11)(3,19,24,7,21,12,6,13)(5,15,22,27,8,20,14,25)$ \\
&$(1,7,11,15,2,23,12,13,4)(3,21,18,10,5,24,20,9,14)(6,26,17,19,27,22,8,25,16)$ \\ \\[-.5em]
\bottomrule
\end{tabular}
\caption{The monodromy groups $\Gal(\cal L_{G})$ of the subgroups $G$ of $S_5$ as permutation subgroups of the 27 lines. The cases corresponding to $S_5$, $A_5$ and $F_5$ have $\dim\mathcal{L}_G=0$. The realisations of the groups are those given in \tabref{tab:subgroups_S5} and the lines are numbered as in \secref{subsec:homology_cubic_surface}.}\label{tab:monodromy_groups}
\end{table}

Let $G$ be one of the subgroups of $S_5$ given in \tabref{tab:subgroups_S5} and let $X_b$ be the Clebsch surface in $\mathcal L_G$.
In \cite{code} we give the monodromy representation $\pi_1(\mathcal L_G\setminus \nabla_G, b)\to S_{27}$ by giving
a line $L$ of $\mathcal L_G$ containing $b$ that is generic in the sense of \thmref{thm:lefschetz_zariski};
and a set of generators of the monodromy group given as permutations of the $27$ lines of $X_b$ described in \secref{sec:cubicsurfaces}.
We also provide the action of $S_5$ on the same lines.

We summarise these results here by giving the actions of $G$ on the 27 lines of the Clebsch surface \eqref{eq:clebsch} in \tabref{tab:actions}, and generators of the corresponding Galois groups $\Gal(\cal L_G)$ in \tabref{tab:monodromy_groups}.


\section{Galois groups of crystallographic K3 surfaces}\label{sec:crystallographic_quartic}

In this section, we extend the results of \secref{sec:quartic} to crystallographic K3 surfaces (see \secref{sec:crystal_new}).
Of the 227 crystallographic groups of dimension $4$, 111 families have singular generic fibres. 
Of the remaining 116, there are (at most\footnote{We did not try to match linear systems beyond direct comparison of the ones corresponding to the first $\Z$-class of each $\Q$-classes. We have  for example no evidence that the linear systems corresponding to 16.1 and 32.1 are not conjugate. Similarly 3.1 and 10.1 have discriminants that do not have the same component structure over $\Q$, but could be conjugate over $\Q[i]$.}) only 48 distinct linear systems up to conjugation.
We give in \tabref{tab:equivalent_crystallographic} the chosen representative for each of the 48 families. 
The corresponding families of quartics are described in \tabref{tab:quartic_crystallographic}, following the notation of \tabref{tab:quartic}.

\begin{theorem}
The statements of \thmref{thm:S5_quartics} also hold for \tabref{tab:quartic_crystallographic}.
\end{theorem}

\begin{table}[h]
\scriptsize
\centering
\makebox[\linewidth][c]{
\setlength{\tabcolsep}{4.7pt}
\begin{tabular}{ccc}
\toprule
$N$ & \thead{Other groups giving \\the same linear system} & \thead{Matching group\\ in \tabref{tab:quartic}} \\\midrule
1.1 & 1.2 & $\Z_1$ \\
2.1 & 2.2, 2.3 & $\Z_2$ \\
3.1 & 3.2 & $\DT$ \\
4.1 & 4.2, 4.3, 4.4 & $\nK$ \\
5.1 & 5.2 & $K_4$ \\
6.1 & 6.2, 6.3 & -- \\
7.1 & 7.2, 7.3 & -- \\
7.4 & 7.5, 7.6, 7.7 & -- \\
8.1 & 8.2 & $\Z_3$ \\
8.3 & 8.4, 8.5 & $S_3$ \\
10.1 &  --  & -- \\
12.1 & 12.2 & $\Z_4$ \\
12.3 & 12.4, 12.5 & $D_4$ \\
13.1 & 13.2, 13.5 & -- \\
13.3 & 13.4, 13.6, 13.7, 13.8, 13.9, 13.10 & -- \\
14.1 & 14.2, 14.4 & $\Z_6$ \\
14.3 & 14.5 & $\tS$ \\
14.6 & 14.7, 14.8, 14.9, 14.10 & $D_6$ \\
16.1 &  --  & -- \\
18.1 &  --  & -- \\
18.2 &  --  & -- \\
18.3 & 18.5 & -- \\
18.4 &  --  & -- \\
19.1 & 19.3 & -- \\
19.2 &  --  & -- \\
\bottomrule
\end{tabular}
\quad 
\begin{tabular}{ccc}
\toprule
$N$ & \thead{Other groups giving \\the same linear system}& \thead{Matching group\\ in \tabref{tab:quartic}}  \\\midrule
19.4 & 19.5, 19.6 & -- \\
24.1 & 24.2, 24.3, 24.4, 24.5 & $S_4$ and $A_4$ \\
\multirow{2}{*}{25.1} & 25.2, 25.3, 25.4, 25.5, 25.6, & \multirow{2}{*}{--} \\
& 25.7, 25.8, 25.9, 25.10, 25.11   \\
26.1 &  --  & -- \\
26.2 &  --  & -- \\
27.1 & 27.2 & $\Z_5$ \\
27.3 & 27.4 & $D_5$ \\
31.1 &  -- & $F_5$   \\
31.2 &  -- & $F_5$   \\
31.3 & 31.4, 31.5, 31.6, 31.7 & $S_5$ and $A_5$ \\
32.1 &  --  & -- \\
32.2 &  --  & -- \\
32.3 &  --  & -- \\
32.4 &  --  & -- \\
32.5 &  --  & -- \\
32.6 &  --  & -- \\
32.7 & 32.12, 32.13, 32.15, 32.17 & -- \\
32.8 &  --  & -- \\
32.9 & 32.14 & -- \\
32.10 &  --  & $\mathcal H$ \\
32.11 &  --  & -- \\
32.16 & 32.19 & -- \\
32.18 & 32.20, 32.21 & -- \\\\
\bottomrule
\end{tabular}
}
\caption{Equivalent families of smooth crystallographic quartic surfaces.}
\label{tab:equivalent_crystallographic}
\end{table}

Note that crystallographic groups do not always act faithfully on $\C[x_0,x_1,x_2,x_3]_4$.
For example the action of the linear automorphism $(x_0:x_1:x_2:x_3)\mapsto (-x_0:-x_1:-x_2:-x_3)$ fixes every monomial of degree 4.
In the column labelled ``$C_N$'', the group that is given is the faithful part, i.e., the image of the group action inside the group of automorphisms of a generic surface in $\mathcal L_{C_N}$.
In the column ``$\operatorname{Symp}(C_N)$'', we give the subgroup of symplectic automorphisms of this action.

We end by remarking that contrary to the case of symmetric quartic surfaces, we are able to compute a majority of discriminants of crystallographic families of quartic surfaces.

\begin{table}[htp]
\scriptsize
\centering
\makebox[\linewidth][c]{
\setlength{\tabcolsep}{3.5pt}
\begin{tabular}{cccccccccccc}
\toprule
$N$ & $C_N$ & $\operatorname{Symp}(C_N)$ & $n$ & $\rho$& ${\rm Gal}({\calL}^4_{C_N})$ &$|{\rm Gal}({\calL}^4_{C_N})|$& $|c_{C_N}|$ & Orbit structure  & $\mathcal{I}_{C_N}$ & \multicolumn{2}{c}{$\deg \nabla_{C_N}$}  \\ \midrule
32.18 & \href{https://www.lmfdb.org/Groups/Abstract/96.70}{$\mathbb Z_2^3\rtimes A_4$} & \href{https://www.lmfdb.org/Groups/Abstract/48.50}{$\mathbb Z_2^2\rtimes A_4$} & 1 & 19 & \href{https://www.lmfdb.org/Groups/Abstract/16.14}{$\Z_2^4$} & $ 16 $ & 512 & $2^{16},4^{24},8^{48}$ & 2304 & $1^{5}$ \\[.2em]
32.16 & \href{https://www.lmfdb.org/Groups/Abstract/48.50}{$\mathbb Z_2^2\rtimes A_4$} & \href{https://www.lmfdb.org/Groups/Abstract/48.50}{$\mathbb Z_2^2\rtimes A_4$} & 2 & 18 & \href{https://www.lmfdb.org/Groups/Abstract/32.51}{$\Z_2^2\rtimes \Z_8$} & $ 32 $ & 464 & $2^{16},8^{54}$ & 11520 & $1^{7},3^{1}$ \\[.2em]
32.11 & \href{https://www.lmfdb.org/Groups/Abstract/24.12}{$S_4$} & \href{https://www.lmfdb.org/Groups/Abstract/24.12}{$S_4$} & 2 & 18 & \href{https://www.lmfdb.org/Groups/Abstract/384.20051}{$\operatorname{GL}_2(\Z_4)\rtimes \Z_2^2$} & $ 384 $ & 0 & $0^1$ & 48 & $1^{2},2^{2},4^{1}$ \\[.2em]
32.10 & \href{https://www.lmfdb.org/Groups/Abstract/16.14}{$\mathbb Z_2^4$} & \href{https://www.lmfdb.org/Groups/Abstract/16.14}{$\mathbb Z_2^4$} & 4 & 16 & \href{https://www.lmfdb.org/Groups/Abstract/1024.djt}{$\Z_2^{10}$} & $ 1024 $ & 320 & $32^{10}$ & 40320 & $1^{15},3^{1}$ \\[.2em]
32.9 & \href{https://www.lmfdb.org/Groups/Abstract/16.3}{$\mathbb Z_2^2\rtimes\mathbb Z_4$} & \href{https://www.lmfdb.org/Groups/Abstract/8.5}{$\mathbb Z_2^3$} & 3 & 17 & \href{https://www.lmfdb.org/Groups/Abstract/256.56092}{$\Z_2^8$} & $ 256 $ & 368 & $8^{22},32^{6}$ & 5760 & $1^{11},3^{1}$ \\[.2em]
32.8 & \href{https://www.lmfdb.org/Groups/Abstract/16.3}{$\mathbb Z_2^2\rtimes\mathbb Z_4$} & \href{https://www.lmfdb.org/Groups/Abstract/16.3}{$\mathbb Z_2^2\rtimes\mathbb Z_4$} & 3 & 18 & \href{https://www.lmfdb.org/Groups/Abstract/256.26991}{$\Z_2^6\rtimes\Z_4$} & $ 256 $ & 416 & $16^{6},32^{8},64^{1}$ & 768 & $2^{5}$ \\[.2em]
32.7 & \href{https://www.lmfdb.org/Groups/Abstract/16.3}{$\mathbb Z_2^2\rtimes\mathbb Z_4$} & \href{https://www.lmfdb.org/Groups/Abstract/16.3}{$\mathbb Z_2^2\rtimes\mathbb Z_4$} & 2 & 18 & \href{https://www.lmfdb.org/Groups/Abstract/64.267}{$\Z_2^6$} & $ 64 $ & 416 & $4^{16},8^{28},32^{4}$ & 3072 & $1^{7},2^{1}$ \\[.2em]
32.6 & \href{https://www.lmfdb.org/Groups/Abstract/16.11}{$\mathbb Z_2\times D_4$} & \href{https://www.lmfdb.org/Groups/Abstract/16.11}{$\mathbb Z_2\times D_4$} & 3 & 17 & $\Z_2\times\href{https://www.lmfdb.org/Groups/Abstract/256.27633}{\Z_2^5\rtimes D_4}$ & $ 512 $ & 256 & $32^{4},64^{2}$ & 2048 & $1^{3},2^{5}$ \\[.2em]
32.5 & \href{https://www.lmfdb.org/Groups/Abstract/12.3}{$A_4$} & \href{https://www.lmfdb.org/Groups/Abstract/12.3}{$A_4$} & 4 & 17 & \href{https://www.lmfdb.org/Groups/Abstract/4608.bq}{$\Z_2^2\rtimes O^+_4(\Z_3)$} & $ 4608 $ & 0 & $0^1$ & 24 & \textcolor{gray}{$4^{1},6^{2}$}&\textcolor{gray}{(?)} \\[.2em]
32.4 & \href{https://www.lmfdb.org/Groups/Abstract/8.5}{$\mathbb Z_2^3$} & \href{https://www.lmfdb.org/Groups/Abstract/8.5}{$\mathbb Z_2^3$} & 6 & 15 & -- & $ 16384 $ & 192 & $64^{3}$ & 2688 & $2^{9},6^{1}$ \\[.2em]
32.3 & \href{https://www.lmfdb.org/Groups/Abstract/8.3}{$D_4$} & \href{https://www.lmfdb.org/Groups/Abstract/8.3}{$D_4$} & 5 & 16 & -- & $ 73728 $ & 0 & $0^1$ & 16 & $2^{1},6^{3}$ \\[.2em]
32.2 & \href{https://www.lmfdb.org/Groups/Abstract/8.2}{$\mathbb Z_2\times\mathbb Z_4$} & \href{https://www.lmfdb.org/Groups/Abstract/8.2}{$\mathbb Z_2\times\mathbb Z_4$} & 4 & 17 & -- & $ 1024 $ & 256 & $64^{4}$ & 1024 & $2^{7}$ \\[.2em]
32.1 & \href{https://www.lmfdb.org/Groups/Abstract/4.2}{$\mathbb Z_2^2$} & \href{https://www.lmfdb.org/Groups/Abstract/4.2}{$\mathbb Z_2^2$} & 10 & 13 & -- & $ 42467328 $ & 0 & $0^1$ & 6 & \textcolor{gray}{$6^{3},18^{1}$}&\textcolor{gray}{(?)} \\[.2em]
\textbf{31.3} & \href{https://www.lmfdb.org/Groups/Abstract/60.5}{$A_5$} & \href{https://www.lmfdb.org/Groups/Abstract/60.5}{$A_5$} & 1 & 19 & \href{https://www.lmfdb.org/Groups/Abstract/24.14}{$\Z_2\times D_6$} & $ 24 $ & 560 & $2^{10},4^{15},12^{40}$ & 360 & $1^{5}$ \\[.2em]
\textbf{31.2} & \href{https://www.lmfdb.org/Groups/Abstract/20.3}{$F_5$} & \href{https://www.lmfdb.org/Groups/Abstract/10.1}{$D_5$} & 2 & 18 & \href{https://www.lmfdb.org/Groups/Abstract/240.189}{$\Z_2\times S_5$} & $ 240 $ & 400 & $40^{10}$ & 20 & $1^{3},2^{3}$ \\[.2em]
\textbf{31.1} & \href{https://www.lmfdb.org/Groups/Abstract/20.3}{$F_5$} & \href{https://www.lmfdb.org/Groups/Abstract/10.1}{$D_5$} & 2 & 18 & \href{https://www.lmfdb.org/Groups/Abstract/240.189}{$\Z_2\times S_5$} & $ 240 $ & 400 & $40^{10}$ & 20 & $1^{3},2^{3}$ \\[.2em]
\textbf{27.3} & \href{https://www.lmfdb.org/Groups/Abstract/10.1}{$D_5$} & \href{https://www.lmfdb.org/Groups/Abstract/10.1}{$D_5$} & 4 & 17 & -- & $ 14400 $ & 0 & $0^1$ & 20 & \textcolor{gray}{$2^{1},6^{2}$}&\textcolor{gray}{(?)} \\[.2em]
\textbf{27.1} & \href{https://www.lmfdb.org/Groups/Abstract/5.1}{$\Z_5$} & \href{https://www.lmfdb.org/Groups/Abstract/5.1}{$\Z_5$} & 6 & 17  & -- & $72000$ & 0 & $0^1$ & $4$ & \textcolor{gray}{$4^{1},12^{1}$}&\textcolor{gray}{(?)} \\[.2em]
26.2 & \href{https://www.lmfdb.org/Groups/Abstract/8.3}{$D_4$} & \href{https://www.lmfdb.org/Groups/Abstract/8.3}{$D_4$} & 5 & 16 & -- & $ 73728 $ & 0 & $0^1$ & 16 & $2^{1},6^{3}$ \\[.2em]
26.1 & \href{https://www.lmfdb.org/Groups/Abstract/4.1}{$\mathbb Z_4$} & \href{https://www.lmfdb.org/Groups/Abstract/4.1}{$\mathbb Z_4$} & 8 & 15 & -- & $ 1474560 $ & 0 & $0^1$ & 4 & \textcolor{gray}{$4^{1},24^{1}$}&\textcolor{gray}{(?)} \\[.2em]
25.1 & \href{https://www.lmfdb.org/Groups/Abstract/24.13}{$\mathbb Z_2\times A_4$} & \href{https://www.lmfdb.org/Groups/Abstract/12.3}{$A_4$} & 3 & 18 & \href{https://www.lmfdb.org/Groups/Abstract/256.55683}{$\Z_2^2\times D_4^2$} & $ 256 $ & 392 & $2^{4},4^{6},8^{15},16^{3},32^{6}$ & 72 & $1^{5},2^{3}$ \\[.2em]
\textbf{24.1} & \href{https://www.lmfdb.org/Groups/Abstract/12.3}{$A_4$} & \href{https://www.lmfdb.org/Groups/Abstract/12.3}{$A_4$} & 4 & 17 & \href{https://www.lmfdb.org/Groups/Abstract/384.20051}{$\operatorname{GL}_2(\Z_4)\rtimes \Z_2^2$} & $ 384 $ & 344 & $2^{4},8^{3},12^{6},32^{6},48^{1}$ & 288 & $1^{4},2^{1},5^{2}$& \\[.2em]
19.4 & \href{https://www.lmfdb.org/Groups/Abstract/16.13}{$D_4\rtimes\mathbb Z_2$} & \href{https://www.lmfdb.org/Groups/Abstract/8.3}{$D_4$} & 4 & 17 & \href{https://www.lmfdb.org/Groups/Abstract/1024.djk}{$\Z_2^7\times D_4$} & $ 1024 $ & 288 & $8^{20},16^{8}$ & 1024 & $1^{6},2^{4}$ \\[.2em]
19.2 & \href{https://www.lmfdb.org/Groups/Abstract/8.2}{$\mathbb Z_2\times\mathbb Z_4$} & \href{https://www.lmfdb.org/Groups/Abstract/8.2}{$\mathbb Z_2\times\mathbb Z_4$} & 6 & 17 & -- & $ 16384 $ & 288 & $16^{2},64^{2},128^{1}$ & 64 & $2^{4},8^{1}$ \\[.2em]
19.1 & \href{https://www.lmfdb.org/Groups/Abstract/8.4}{$Q_8$} & \href{https://www.lmfdb.org/Groups/Abstract/4.1}{$\mathbb Z_4$} & 5 & 17 & -- & $ 4096 $ & 288 & $8^{2},16^{5},32^{4},64^{1}$ & 256 & $1^{3},2^{2},4^{2}$ \\[.2em]
18.4 & \href{https://www.lmfdb.org/Groups/Abstract/8.5}{$\mathbb Z_2^3$} & \href{https://www.lmfdb.org/Groups/Abstract/8.5}{$\mathbb Z_2^3$} & 6 & 15 & -- & $ 16384 $ & 192 & $64^{3}$ & 2688 & $1^{6},2^{6},6^{1}$ \\[.2em]
18.3 & \href{https://www.lmfdb.org/Groups/Abstract/8.2}{$\mathbb Z_2\times\mathbb Z_4$} & \href{https://www.lmfdb.org/Groups/Abstract/4.2}{$\mathbb Z_2^2$} & 5 & 16 & -- & $ 4096 $ & 240 & $8^{6},16^{8},64^{1}$ & 192 & $1^{6},2^{4},3^{2}$ \\[.2em]
18.2 & \href{https://www.lmfdb.org/Groups/Abstract/8.3}{$D_4$} & \href{https://www.lmfdb.org/Groups/Abstract/4.2}{$\mathbb Z_2^2$} & 7 & 15 & -- & $ 24576 $ & 200 & $8^{1},16^{2},32^{2},96^{1}$ & 16 & \textcolor{gray}{$1^{3},2^{2},4^{1},5^{1},10^{1}$}&\textcolor{gray}{(?)} \\[.2em]
18.1 & \href{https://www.lmfdb.org/Groups/Abstract/4.2}{$\mathbb Z_2^2$} & \href{https://www.lmfdb.org/Groups/Abstract/4.2}{$\mathbb Z_2^2$} & 10 & 13 & -- & $ 393216 $ & 128 & $128^{1}$ & 648 & \textcolor{gray}{$2^{4},4^{2},20^{1}$}&\textcolor{gray}{(?)} \\[.2em]
16.1 & \href{https://www.lmfdb.org/Groups/Abstract/4.2}{$\mathbb Z_2^2$} & \href{https://www.lmfdb.org/Groups/Abstract/4.2}{$\mathbb Z_2^2$} & 10 & 13 & -- & $ 42467328 $ & 0 & $0^1$ & 6 & \textcolor{gray}{$6^{3},18^{1}$}&\textcolor{gray}{(?)} \\[.2em]
\textbf{14.6} & \href{https://www.lmfdb.org/Groups/Abstract/12.4}{$D_6$} & \href{https://www.lmfdb.org/Groups/Abstract/6.1}{$S_3$} & 6 & 16 &  \href{https://www.lmfdb.org/Groups/Abstract/13824.o}{$D_6\cdot W(F_4)$} & $ 13824 $ & 200 & $2^{1},6^{3},36^{1},48^{3}$ & 6 & $1^{4},2^{2},5^{1},8^{1}$ \\[.2em]
\textbf{14.3} & \href{https://www.lmfdb.org/Groups/Abstract/6.1}{$S_3$} & \href{https://www.lmfdb.org/Groups/Abstract/6.1}{$S_3$} & 7 & 15 & -- & $ 622080 $ & 0 & $0^1$ & 6 & \textcolor{gray}{$1^{2},2^{1},5^{2},12^{1}$}&\textcolor{gray}{(?)} \\[.2em]
\textbf{14.1} & \href{https://www.lmfdb.org/Groups/Abstract/6.2}{$\mathbb Z_6$} & \href{https://www.lmfdb.org/Groups/Abstract/3.1}{$\mathbb Z_3$} & 7 & 16 & -- & $ 41472 $ & 200 & $2^{1},18^{1},36^{1},144^{1}$ & 2 & $1^{3},2^{3},5^{1},8^{1}$ \\[.2em]
13.3 & \href{https://www.lmfdb.org/Groups/Abstract/8.3}{$D_4$} & \href{https://www.lmfdb.org/Groups/Abstract/8.3}{$D_4$} & 6 & 16 & -- & $ 8192 $ & 248 & $8^{11},16^{2},32^{4}$ & 144 & $1^{5},2^{5},3^{2}$ \\[.2em]
13.1 & \href{https://www.lmfdb.org/Groups/Abstract/8.2}{$\mathbb Z_2\times\mathbb Z_4$} & \href{https://www.lmfdb.org/Groups/Abstract/4.1}{$\mathbb Z_4$} & 7 & 16 & -- & $ 32768 $ & 248 & $8^{1},16^{3},32^{2},128^{1}$ & 36 & $1^{2},2^{3},4^{2},6^{1}$ \\[.2em]
\textbf{12.3} & \href{https://www.lmfdb.org/Groups/Abstract/8.3}{$D_4$} & \href{https://www.lmfdb.org/Groups/Abstract/4.2}{$\mathbb Z_2^2$} & 7 & 15 & -- & $ 24576 $ & 200 & $8^{1},16^{2},32^{2},96^{1}$ & 16 & \textcolor{gray}{$1^{5},2^{3},5^{1},10^{1}$}&\textcolor{gray}{(?)} \\[.2em]
\textbf{12.1} & \href{https://www.lmfdb.org/Groups/Abstract/4.1}{$\mathbb Z_4$} & \href{https://www.lmfdb.org/Groups/Abstract/2.1}{$\mathbb Z_2$} & 9 & 14 & -- & $ 245760 $ & 160 & $160^{1}$ & 2 & \textcolor{gray}{$1^{2},2^{3},24^{1}$}&\textcolor{gray}{(?)} \\[.2em]
10.1 & \href{https://www.lmfdb.org/Groups/Abstract/2.1}{$\Z_2$} & \href{https://www.lmfdb.org/Groups/Abstract/2.1}{$\Z_2$} & 18 & 9  & $W(E_8)$ & 696729600 & 0 & $0^1$ & $1$ &  \textcolor{gray}{$12^{1},48^{1}$}&\textcolor{gray}{(?)} \\[.2em]
\textbf{8.3} & \href{https://www.lmfdb.org/Groups/Abstract/6.1}{$S_3$} & \href{https://www.lmfdb.org/Groups/Abstract/3.1}{$\mathbb Z_3$} & 10 & 14 & $\mathbb Z_2\times W(E_6)$ & $ 103680 $ & 164 & $2^{1},54^{3}$ & 6 & \textcolor{gray}{$1^{1},6^{2},21^{1}$}&\textcolor{gray}{(?)} \\[.2em]
\textbf{8.1} & \href{https://www.lmfdb.org/Groups/Abstract/3.1}{$\Z_3$} & \href{https://www.lmfdb.org/Groups/Abstract/3.1}{$\Z_3$} & 12 & 13  & -- & $39191040$ & 0 & $0^1$ & $2$ & \textcolor{gray}{$6^{1},34^{1}$}&\textcolor{gray}{(?)} \\[.2em]
7.4 & \href{https://www.lmfdb.org/Groups/Abstract/8.3}{$D_4$} & \href{https://www.lmfdb.org/Groups/Abstract/4.1}{$\mathbb Z_4$} & 9 & 15 & -- & $ 73728 $ & 208 & $8^{2},48^{4}$ & 80 & $1^{3},6^{1},10^{2}$ \\[.2em]
7.1 & \href{https://www.lmfdb.org/Groups/Abstract/4.1}{$\mathbb Z_4$} & \href{https://www.lmfdb.org/Groups/Abstract/4.1}{$\mathbb Z_4$} & 10 & 15 & -- & $ 294912 $ & 208 & $16^{1},192^{1}$ & 20 & \textcolor{gray}{$2^{2},6^{1},20^{1}$}&\textcolor{gray}{(?)} \\[.2em]
6.1 & \href{https://www.lmfdb.org/Groups/Abstract/8.5}{$\mathbb Z_2^3$} & \href{https://www.lmfdb.org/Groups/Abstract/4.2}{$\mathbb Z_2^2$} & 9 & 14 & -- & $ 65536 $ & 176 & $8^{6},32^{4}$ & 72 & $1^{4},2^{6},3^{4},4^{1}$ \\[.2em]
\textbf{5.1} & \href{https://www.lmfdb.org/Groups/Abstract/4.2}{$\mathbb Z_2^2$} & \href{https://www.lmfdb.org/Groups/Abstract/4.2}{$\mathbb Z_2^2$} & 10 & 13 & -- & $ 393216 $ & 128 & $128^{1}$ & 648 & \textcolor{gray}{$1^{4},2^{6},20^{1}$}&\textcolor{gray}{(?)} \\[.2em]
\textbf{4.1} & \href{https://www.lmfdb.org/Groups/Abstract/4.2}{$\mathbb Z_2^2$} & \href{https://www.lmfdb.org/Groups/Abstract/2.1}{$\mathbb Z_2$} & 13 & 12 & -- & $ 1769472 $ & 104 & $8^{1},48^{2}$ & 2 & \textcolor{gray}{$1^{2},2^{1},6^{1},10^{2},14^{1}$}&\textcolor{gray}{(?)} \\[.2em]
\textbf{3.1} & \href{https://www.lmfdb.org/Groups/Abstract/2.1}{$\Z_2$} & \href{https://www.lmfdb.org/Groups/Abstract/2.1}{$\Z_2$} & 18 & 9  & $W(E_8)$ & $696729600$ & 0 & $0^1$ & 1 & \textcolor{gray}{$6^{2},48^{1}$}&\textcolor{gray}{(?)} \\[.2em]
\textbf{2.1} & \href{https://www.lmfdb.org/Groups/Abstract/2.1}{$\Z_2$} & \href{https://www.lmfdb.org/Groups/Abstract/1.1}{$\Z_1$}& 21 & 8 &  $W(E_7)$ & 2903040 & 56 &$56^{1}$ & 1 & \textcolor{gray}{$1^{1},27^{1},40^{1}$}&\textcolor{gray}{(?)} \\[.2em]
\textbf{1.1} & \href{https://www.lmfdb.org/Groups/Abstract/1.1}{$\Z_1$} & \href{https://www.lmfdb.org/Groups/Abstract/1.1}{$\Z_1$} & 34 & 1 &  \href{https://www.lmfdb.org/Groups/Abstract/1.1}{$\Z_1$} & 1 & 0 & $0^{1}$ & 1 & $108^1$ \\[.2em]
\bottomrule 
\end{tabular}
}
\caption{Galois groups, conics and discriminant degrees of crystallographic quartic surfaces. When the groups could be identified in the LMFDB group database \cite{lmfdb}, a link is provided.}
\label{tab:quartic_crystallographic}
\end{table}

\clearpage
\small
\bibliographystyle{abbrv}

\bibliography{references}

@misc{BrazeltonRaman_2024,
  author = {Brazelton, Thomas and Raman, Sidhanth},
  title  = {Monodromy in the Space of Symmetric Cubic Surfaces with a Line},
  year   = {2024},
  note   = {arXiv:2410.09270},
  url    = {https://arxiv.org/abs/2410.09270}
}

@ARTICLE{LairezEtAl_2024,
  AUTHOR = {Lairez, Pierre and {Pichon-Pharabod}, Eric and Vanhove, Pierre},
  year = {2024},
  month = nov,
  DOI = {10.1090/mcom/3947},
  ISSN = {0025-5718, 1088-6842},
  journal = {Mathematics of Computation},
  LANGID = {english},
  NUMBER = {350},
  PAGES = {2985--3025},
  TITLE = {Effective Homology and Periods of Complex Projective Hypersurfaces},
  URLDATE = {2024-12-05},
  VOLUME = {93},
}

@misc{Mezzarobba_2016,
  author = {Mezzarobba, Marc},
  title  = {Rigorous Multiple-Precision Evaluation of {D}-Finite Functions in {S}age{M}ath},
  year   = {2016},
  month  = jul,
  note   = {arXiv:1607.01967},
  doi    = {10.48550/arXiv.1607.01967},
  url    = {https://arxiv.org/abs/1607.01967}
}

@article{cayley1849triple,
  author  = {Arthur Cayley},
  title   = {On the Triple Tangent Planes of Surfaces of the Third Order},
  journal = {Cambridge and Dublin Mathematical Journal},
  volume  = {4},
  pages   = {118--132},
  year    = {1849}
}

@article{salmon1849triple,
  author  = {George Salmon},
  title   = {On the Triple Tangent Planes to a Surface of the Third Order},
  journal = {Cambridge and Dublin Mathematical Journal},
  volume  = {4},
  pages   = {252--260},
  year    = {1849}
}

@book{dolgachev2012classical,
  title={Classical algebraic geometry: a modern view},
  author={Dolgachev, Igor V},
  year={2012},
  publisher={Cambridge University Press}
}

@article{Sottile_2025,
title = {Galois groups in enumerative geometry and applications},
journal = {Expositiones Mathematicae},
volume = {43},
number = {6},
pages = {125731},
year = {2025},
issn = {0723-0869},
doi = {10.1016/j.exmath.2025.125731},
url = {https://www.sciencedirect.com/science/article/pii/S0723086925000866},
author = {Frank Sottile and Thomas Yahl},
}

@article{ranestad2020twentyseven,
  author  = {Kristian Ranestad and Bernd Sturmfels},
  title   = {Twenty-Seven Questions about the Cubic Surface},
  journal = {Le Matematiche},
  volume  = {75},
  number  = {2},
  pages   = {411--424},
  year    = {2020},
  doi     = {10.4418/2020.75.2.1},
  url     = {https://doi.org/10.4418/2020.75.2.1}
}

@article{harris1979galois,
  author  = {Joe Harris},
  title   = {Galois Groups of Enumerative Problems},
  journal = {Duke Mathematical Journal},
  volume  = {46},
  number  = {4},
  pages   = {685--724},
  year    = {1979}
}

@article{hauenstein2018numerical,
  title={Numerical computation of {G}alois groups},
  author={Hauenstein, Jonathan D and Rodriguez, Jose Israel and Sottile, Frank},
  journal={Foundations of Computational Mathematics},
  volume={18},
  number={4},
  pages={867--890},
  year={2018},
  publisher={Springer}
}

@article{leykin2009galois,
  title={Galois groups of {Schubert} problems via homotopy computation},
  author={Leykin, Anton and Sottile, Frank},
  journal={Mathematics of Computation},
  volume={78},
  number={267},
  pages={1749--1765},
  year={2009}
}

@book{CoxLittleOSheaIVA,
  author    = {David Cox and John Little and Donal O'Shea},
  title     = {Ideals, Varieties, and Algorithms: An Introduction to Computational Algebraic Geometry and Commutative Algebra},
  publisher = {Springer},
  year      = {2015},
  edition   = {4th},
  isbn      = {978-3-319-16720-6},
  series    = {Undergraduate Texts in Mathematics},
}

@book{CoxKatz_1999,
  title = {Mirror {{Symmetry}} and {{Algebraic Geometry}}},
  author = {Cox, David and Katz, Sheldon},
  year = {1999},
  month = mar,
  series = {Mathematical {{Surveys}} and {{Monographs}}},
  volume = {68},
  publisher = {American Mathematical Society},
  address = {Providence, Rhode Island},
  doi = {10.1090/surv/068},
  urldate = {2024-03-20},
  isbn = {978-0-8218-2127-5 978-1-4704-1295-1},
  langid = {english},
}

@article{Griffiths_1969,
  title = {On the {{Periods}} of {{Certain Rational Integrals}}: {{I}}},
  shorttitle = {On the {{Periods}} of {{Certain Rational Integrals}}},
  author = {Griffiths, Phillip A.},
  year = {1969},
  journal = {Annals of Mathematics},
  volume = {90},
  number = {3},
  eprint = {1970746},
  eprinttype = {jstor},
  pages = {460--495},
  publisher = {Annals of Mathematics},
  issn = {0003-486X},
  doi = {10.2307/1970746},
  urldate = {2024-03-20},
}

@article{Grothendieck_1966,
  title = {On the de {{Rham}} Cohomology of Algebraic Varieties},
  author = {Grothendieck, A.},
  year = {1966},
  month = jan,
  journal = {Publications Math{\'e}matiques de l'Institut des Hautes {\'E}tudes Scientifiques},
  volume = {29},
  number = {1},
  pages = {95--103},
  issn = {1618-1913},
  doi = {10.1007/BF02684807},
  urldate = {2024-03-20},
  langid = {english},
}

@inproceedings{GuillemotLairez_2024,
  title = {Validated {{Numerics}} for {{Algebraic Path Tracking}}},
  booktitle = {Proceedings of the 2024 {{International Symposium}} on {{Symbolic}} and {{Algebraic Computation}}},
  author = {Guillemot, Alexandre and Lairez, Pierre},
  year = {2024},
  month = jul,
  series = {{{ISSAC}} '24},
  pages = {36--45},
  publisher = {Association for Computing Machinery},
  address = {New York, NY, USA},
  doi = {10.1145/3666000.3669673},
  urldate = {2025-04-20},
  isbn = {9798400706967},
}

@inproceedings{BostanEtAl_2013,
  title = {Creative Telescoping for Rational Functions Using the {{Griffiths}}--{{Dwork}} Method},
  shorttitle = {Creative Telescoping for Rational Functions Using the Griffiths},
  booktitle = {Proceedings of the 38th {{International Symposium}} on {{Symbolic}} and {{Algebraic Computation}}},
  author = {Bostan, Alin and Lairez, Pierre and Salvy, Bruno},
  year = {2013},
  month = jun,
  series = {{{ISSAC}} '13},
  pages = {93--100},
  publisher = {Association for Computing Machinery},
  address = {New York, NY, USA},
  doi = {10.1145/2465506.2465935},
  urldate = {2024-03-20},
  isbn = {978-1-4503-2059-7},
}

@article{DeRham_1931,
  title = {Sur l'analysis situs des variétés à {$n$} dimensions},
  author = {de Rham, Georges},
  year = {1931},
  journal = {Journal de Mathématiques Pures et Appliquées. Neuvième Série},
  volume = {10},
  pages = {115--200},
  issn = {0021-7824},
  langid = {french},
}

@article{Lairez_2016,
  title = {Computing Periods of Rational Integrals},
  author = {Lairez, Pierre},
  year = {2016},
  month = jul,
  journal = {Mathematics of Computation},
  volume = {85},
  number = {300},
  pages = {1719--1752},
  issn = {0025-5718, 1088-6842},
  doi = {10.1090/mcom/3054},
  urldate = {2025-04-20},
  langid = {english},
}

@incollection{ChudnovskyChudnovsky_1990,
  title = {Computer {{Algebra}} in the {{Service}} of {{Mathematical Physics}} and {{Number Theory}}},
  booktitle = {Computers in {{Mathematics}}},
  author = {Chudnovsky, David V. and Chudnovsky, Gregory V.},
  year = {1990},
  publisher = {CRC Press},
  isbn = {978-1-00-307215-7}
}

@article{VanDerHoeven_1999,
  title = {Fast Evaluation of Holonomic Functions},
  author = {{van der Hoeven}, Joris},
  year = {1999},
  month = jan,
  journal = {Theoretical Computer Science},
  series = {Real {{Numbers}} and {{Computers}}},
  volume = {210},
  number = {1},
  pages = {199--215},
  issn = {0304-3975},
  doi = {10.1016/S0304-3975(98)00102-9},
  urldate = {2024-03-20},
}

@incollection{KauersEtAl_2015,
  title = {Ore {{Polynomials}} in {{Sage}}},
  booktitle = {Computer {{Algebra}} and {{Polynomials}}: {{Applications}} of {{Algebra}} and {{Number Theory}}},
  author = {Kauers, Manuel and Jaroschek, Maximilian and Johansson, Fredrik},
  editor = {Gutierrez, Jaime and Schicho, Josef and Weimann, Martin},
  year = {2015},
  pages = {105--125},
  publisher = {Springer International Publishing},
  address = {Cham},
  doi = {10.1007/978-3-319-15081-9_6},
  urldate = {2024-03-20},
  isbn = {978-3-319-15081-9},
  langid = {english}
}

@misc{sagemath,
  title = {{{SageMath}}, the {{Sage Mathematics Software}}},
  author = {{The Sage Developers}},
  year = {2025},
  version = {10.7},
  doi = {10/gr8dhc}
}

@article{Bouyer_2020,
  title = {On the {{Monodromy}} and {{Galois Group}} of {{Conics Lying}} on {{Heisenberg Invariant Quartic}} {{K3}} {{Surfaces}}},
  author = {Bouyer, Florian},
  year = {2020},
  month = sep,
  journal = {Glasgow Mathematical Journal},
  volume = {62},
  number = {3},
  pages = {640--660},
  issn = {0017-0895, 1469-509X},
  doi = {10.1017/S0017089519000399},
  urldate = {2025-04-22},
  langid = {english},
}

@book{BrownEtAl_1978,
  author    = {Brown, Harold and B{\"u}low, Rolf and Neub{\"u}ser, Joachim and Wondratschek, Hans and Zassenhaus, Hans},
  title     = {Crystallographic Groups of Four-Dimensional Space},
  year      = {1978},
  publisher = {John Wiley \& Sons},
  address   = {New York},
  series    = {Wiley Monographs in Crystallography}
}

@article{VanKampen_1933,
  title = {On the {{Fundamental Group}} of an {{Algebraic Curve}}},
  author = {Van Kampen, Egbert R.},
  year = {1933},
  journal = {American Journal of Mathematics},
  volume = {55},
  number = {1/4},
  eprint = {2371128},
  eprinttype = {jstor},
  pages = {255--260},
  issn = {00029327},
  doi = {10.2307/2371128},
  urldate = {2024-10-17},
  langid = {english},
}

@article{Zariski_1937,
  title = {A {{Theorem}} on the {{Poincar{\'e} Group}} of an {{Algebraic Hypersurface}}},
  author = {Zariski, Oscar},
  year = {1937},
  journal = {Annals of Mathematics},
  volume = {38},
  number = {1},
  eprint = {1968515},
  eprinttype = {jstor},
  pages = {131--141},
  publisher = {[Annals of Mathematics, Trustees of Princeton University on Behalf of the Annals of Mathematics, Mathematics Department, Princeton University]},
  issn = {0003-486X},
  doi = {10.2307/1968515},
  urldate = {2024-10-17},
}

@article{Lamotke_1981,
  title = {The Topology of Complex Projective Varieties after {{S}}. {{Lefschetz}}},
  author = {Lamotke, Klaus},
  year = {1981},
  month = jan,
  journal = {Topology},
  volume = {20},
  number = {1},
  pages = {15--51},
  issn = {0040-9383},
  doi = {10.1016/0040-9383(81)90013-6},
  urldate = {2023-11-24}
}

@book{Lefschetz_1924,
  title = {L'analysis Situs et La G{\'e}om{\'e}trie Alg{\'e}brique},
  author = {Lefschetz, Solomon},
  year = {1924},
  publisher = {Gauthier-Villars et cie}
}

@article{Cheniot_1973a,
  title = {{Une d{\'e}monstration du th{\'e}or{\`e}me de Zariski sur les sections hyperplanes d'une hypersurface projective et du th{\'e}or{\`e}me de Van Kampen sur le groupe fondamental du compl{\'e}mentaire d'une courbe projective plane}},
  author = {Cheniot, D.},
  year = {1973},
  journal = {Compositio Mathematica},
  volume = {27},
  number = {2},
  pages = {141--158},
  publisher = {Noordhoff International Publishing},
  issn = {0010-437X},
  urldate = {2025-04-24},
  langid = {fra},
}

@inproceedings{Cheniot_1974,
  title = {{Le th\'eor\`eme de Van Kampen sur le groupe fondamental du compl\'ementaire d'une courbe alg\'ebrique projective plane}},
  booktitle = {{Fonctions de Plusieurs Variables Complexes}},
  author = {Cheniot, Denis},
  editor = {Norguet, Fran{\c c}ois},
  year = {1974},
  pages = {394--417},
  publisher = {Springer},
  address = {Berlin, Heidelberg},
  doi = {10.1007/BFb0068112},
  isbn = {978-3-540-37814-3},
  langid = {french},
}

@article{HammTrang_1973,
  title = {Un Th{\'e}or{\`e}me de {{Zariski}} Du Type de {{Lefschetz}}},
  author = {Hamm, Helmut A. and Tr{\'a}ng, L{\^e} D{\~u}ng},
  year = {1973},
  journal = {Annales scientifiques de l'{\'E}cole Normale Sup{\'e}rieure},
  volume = {6},
  number = {3},
  pages = {317--355},
  issn = {1873-2151},
  doi = {10.24033/asens.1250},
  urldate = {2024-10-17},
  langid = {english},
}

@article{Sertoz_2019,
  title = {Computing {{Periods}} of {{Hypersurfaces}}},
  author = {Sert{\"o}z, Emre Can},
  year = {2019},
  month = apr,
  journal = {Mathematics of Computation},
  volume = {88},
  number = {320},
  eprint = {1803.08068},
  primaryclass = {cs, math},
  pages = {2987--3022},
  issn = {0025-5718, 1088-6842},
  doi = {10.1090/mcom/3430},
  urldate = {2023-11-24},
  archiveprefix = {arXiv},
  langid = {english},
}

@misc{OSCAR,
  key          = {OSCAR},
  organization = {The OSCAR Team},
  title        = {OSCAR -- Open Source Computer Algebra Research system, Version 1.3.1},
  year         = {2025},
  author = {{The OSCAR Team}},
  url          = {https://www.oscar-system.org},
}

@misc{code,
  key          = {code},
  title        = {Galois Groups of cubic surfaces -- supporting code and data},
  author = {Pichon-Pharabod, Eric and Telen, Simon},
  year         = {2025},
  howpublished = {\\\url{https://github.com/ericpipha/Galois-Groups-of-cubic-surfaces-supporting-data/}},
  url          = {https://github.com/ericpipha/Galois-Groups-of-cubic-surfaces-supporting-data/},
}

@article{Bouyer_2018,
  title = {The Picard group of various families of {$(\mathbb {Z}/2\mathbb {Z})^{4}$}-invariant quartic {K3} surfaces},
  author = {Bouyer, Florian},
  year = {2018},
  journal = {Acta Arithmetica},
  volume = {186},
  pages = {61--86},
  publisher = {Instytut Matematyczny Polskiej Akademii Nauk},
  issn = {0065-1036, 1730-6264},
  doi = {10.4064/aa170718-7-3},
  urldate = {2024-03-20},
  langid = {english},
}

@article{Eklund_2018,
  title = {Curves on {{Heisenberg}} Invariant Quartic Surfaces in Projective 3-Space},
  author = {Eklund, David},
  year = {2018},
  month = sep,
  journal = {European Journal of Mathematics},
  volume = {4},
  number = {3},
  pages = {931--952},
  issn = {2199-6768},
  doi = {10.1007/s40879-018-0216-2},
  urldate = {2025-06-24},
  langid = {english},
}

@article{LairezSertoz_2019,
  title = {A {{Numerical Transcendental Method}} in {{Algebraic Geometry}}: {{Computation}} of {{Picard Groups}} and {{Related Invariants}}},
  shorttitle = {A {{Numerical Transcendental Method}} in {{Algebraic Geometry}}},
  author = {Lairez, Pierre and Sert{\"o}z, Emre Can},
  year = {2019},
  month = jan,
  journal = {SIAM Journal on Applied Algebra and Geometry},
  volume = {3},
  number = {4},
  pages = {559--584},
  issn = {2470-6566},
  doi = {10.1137/18M122861X},
  urldate = {2024-03-20},
  langid = {english},
}

@article{Nikulin_1979,
  title = {{Finite automorphism groups of K{\"a}hlerian surfaces of type K3}},
  author = {Nikulin, V. V.},
  year = {1979},
  journal = {Trudy Moskovskogo Matematicheskogo Obshchestva},
  volume = {38},
  pages = {75--137},
  issn = {0134-8663},
  langid = {russian},
}

@book{GriffithsHarris_1978,
  author    = {Griffiths, Phillip A. and Harris, Joseph},
  title     = {Principles of Algebraic Geometry},
  year      = {1978},
  publisher = {John Wiley \& Sons},
  address   = {New York},
  series    = {Pure and Applied Mathematics}
}

@book{Huybrechts_2016,
  title = {Lectures on {{K3 Surfaces}}},
  author = {Huybrechts, Daniel},
  year = {2016},
  month = sep,
  edition = {First},
  publisher = {Cambridge University Press},
  doi = {10.1017/CBO9781316594193},
  urldate = {2023-11-24},
  isbn = {978-1-107-15304-2 978-1-316-59419-3},
  langid = {english},
}

@incollection{Dolgachev_1983,
  title = {Integral Quadratic Forms: Applications to Algebraic Geometry},
  booktitle = {S{\'e}minaire {{Bourbaki}} : Volume 1982/83, Expos{\'e}s 597-614},
  author = {Dolgachev, Igor V},
  year = {1983},
  series = {Ast{\'e}risque},
  number = {105-106},
  pages = {251--278},
  publisher = {Soci{\'e}t{\'e} math{\'e}matique de France},
  langid = {english},
  mrnumber = {728992},
  zmnumber = {0535.10018}
}

@book{CarlsonEtAl_2017,
  title = {Period {{Mappings}} and {{Period Domains}}},
  author = {Carlson, James and {M{\"u}ller-Stach}, Stefan and Peters, Chris},
  year = {2017},
  series = {Cambridge {{Studies}} in {{Advanced Mathematics}}},
  edition = {Second},
  publisher = {Cambridge University Press},
  address = {Cambridge},
  doi = {10.1017/9781316995846},
  urldate = {2024-03-28},
  isbn = {978-1-108-42262-8}
}

@article{ElsenhansJahnel_2015,
  title = {Moduli Spaces and the Inverse {{Galois}} Problem for Cubic Surfaces},
  author = {Elsenhans, Andreas-Stephan and Jahnel, J{\"o}rg},
  year = {2015},
  month = nov,
  journal = {Transactions of the American Mathematical Society},
  volume = {367},
  number = {11},
  pages = {7837--7861},
  issn = {0002-9947, 1088-6850},
  doi = {10.1090/S0002-9947-2015-06277-1},
  urldate = {2025-07-09},
  langid = {english},
}

@article{Landi_2025,
author = {Landi, Alberto},
title = {Stacks, Monodromy and Symmetric Cubic Surfaces},
journal = {Mathematische Nachrichten},
volume = {299},
number = {11},
pages = {3392-3426},
doi = {https://doi.org/10.1002/mana.70245},
url = {https://onlinelibrary.wiley.com/doi/abs/10.1002/mana.70245},
eprint = {https://onlinelibrary.wiley.com/doi/pdf/10.1002/mana.70245},
year = {2026}
}

@article{Schubert_1886,
  title = {Die $n$-dimensionale {V}erallgemeinerung der {A}nzahlen für die vielpunktig berührenden {T}angenten einer punktallgemeinen {F}läche $m^{\mathrm{ten}}$ {G}rades},
  author = {Schubert, Hermann},
  year = {1886},
  month = mar,
  journal = {Mathematische Annalen},
  volume = {26},
  number = {1},
  pages = {52--73},
  publisher = {{Springer Science and Business Media LLC}},
  issn = {0025-5831, 1432-1807},
  doi = {10.1007/bf01443569},
  urldate = {2025-07-09},
  copyright = {http://www.springer.com/tdm},
  langid = {ngerman}
}

@book{Segre_1942,
  author    = {Segre, Beniamino},
  title     = {The Non-Singular Cubic Surfaces: A New Method of Investigation with Special Reference to Questions of Reality},
  publisher = {The Clarendon Press},
  address   = {Oxford},
  year      = {1942}
}

@article{MedranoMartinDelCampo_2022,
  title = {Monodromy of the Family of {{Cubic Surfaces}} Branching over {{Smooth Cubic Curves}}},
  author = {{Medrano Mart{\'{\i}}n del Campo}, Ad{\'a}n},
  year = {2022},
  journal = {Annales de l'Institut Fourier},
  volume = {72},
  number = {3},
  pages = {963--987},
  publisher = {Association des Annales de l'Institut Fourier},
  doi = {10.5802/aif.3481},
  langid = {english}
}

@misc{MedranoMartinDelCampo_2022a,
  title = {Monodromy of the Families of Del {{Pezzo}} and {K3} Surfaces Branching over Smooth Quartic Curves},
  author = {{Medrano Mart{\'i}n del Campo}, Ad{\'a}n},
  year = {2022},
  month = feb,
  number = {arXiv:2112.14836},
  eprint = {2112.14836},
  primaryclass = {math},
  publisher = {arXiv},
  doi = {10.48550/arXiv.2112.14836},
  note = {arXiv:2112.14836},
  urldate = {2025-08-01},
  archiveprefix = {arXiv},
  langid = {english},
}

@misc{lmfdb,
  shorthand    = {LMFDB},
  author       = {{The {LMFDB Collaboration}}},
  title        = {The {L}-functions and modular forms database},
  howpublished = {\url{https://www.lmfdb.org}},
  year         = {2025},
  note         = {[Online; accessed 5 August 2025]},
}

@book{Jordan_1870,
 author = {Jordan, Camille},
 title = {Traité des substitutions et des équations algébriques.},
 year = {1870},
 language = {French},
 publisher = {Gauthier-Villars},
 address   = {Paris},
 url = {https://gallica.bnf.fr/ark:/12148/bpt6k290534/f1.item},
 zbMATH = {2719937},
 JFM = {03.0042.02}
}

\noindent{\bf Authors' addresses:}  \\

\noindent Eric Pichon-Pharabod, Oxford University
\hfill {\tt eric.pichon-pharabod@maths.ox.ac.uk}

\noindent Simon Telen, MPI-MiS Leipzig
\hfill {\tt simon.telen@mis.mpg.de}

\end{document}